\documentclass[11pt,twoside]{article}
\usepackage{etex}
\usepackage{anysize}
\usepackage{t4phonet}
\usepackage{amssymb}
\usepackage{amsmath}
\usepackage{amsthm}
\usepackage{graphicx}
\usepackage{mathrsfs}
\usepackage{euscript}
\usepackage{mathtools}
\usepackage{bbm}
\usepackage{graphicx}
\usepackage{enumitem}
\usepackage{tikz}
\usepackage{tikz-cd}
\usepackage[all,cmtip]{xy}
\usetikzlibrary{decorations.markings,positioning}
\usepgflibrary{arrows}
\usepackage{setspace}
\usepackage{fancyhdr}
\usepackage{titlesec}
\usepackage[marginratio=1:1,top=3.2cm]{geometry}
\usepackage{url}
\usepackage{hyperref}
\usepackage{cleveref}
\newtheorem*{myex*}{Exercise}
\newtheorem*{mythm*}{Theorem}

\newtheorem{mythm}{Theorem}[section]
\newtheorem{mycor}[mythm]{Corollary}
\newtheorem{mylem}[mythm]{Lemma}
\newtheorem{mypro}[mythm]{Proposition}

\newtheorem{myrem}[mythm]{Remark}
\theoremstyle{definition}
\newtheorem{mydef}[mythm]{Definition}
\numberwithin{equation}{subsection}
\fancypagestyle{myheadings}{%
	\fancyhf{}
	\fancyhead[OC]{SHEAGAN A. K. A JOHN}
	\fancyhead[EC]{HIGHER INVARIANTS AND POLYNOMIAL GROWTH GROUPS}
	\fancyfoot[C]{\thepage}%
	
}

\titleformat{\section}{\normalsize\centering\scshape}{\S\thesection\quad}{0pt}{}
\titleformat{\subsection}{\normalsize\centering\scshape}{\S\thesubsection\quad}{0pt}{}

\begin{document}
\title{\normalsize{SECONDARY HIGHER INVARIANTS AND CYCLIC COHOMOLOGY FOR GROUPS OF POLYNOMIAL GROWTH}}
\author{SHEAGAN A. K. A. JOHN}
\renewcommand\footnotemark{}
\thanks{The author was partially supported by NSF 1800737, NSF 1952693}
\date{}
\maketitle	

\begin{abstract}
We prove that if $\Gamma$ is a group of polynomial growth then each delocalized cyclic cocycle on the group algebra has a representative of polynomial growth. For each delocalized cocyle we thus define a higher analogue of Lott’s delocalized eta invariant and prove its convergence for invertible differential operators. We also use a determinant map construction of Xie and Yu to prove that if $\Gamma$ is of polynomial growth then there is a well defined pairing between delocalized cyclic cocyles and $K$-theory classes of $C^*$-algebraic secondary higher invariants. When this $K$-theory class is that of a higher rho invariant of an invertible differential operator we show this pairing is precisely the aforementioned higher analogue of Lott’s delocalized eta invariant. As an application of this equivalence we provide a delocalized higher Atiyah-Patodi-Singer index theorem given $M$ is a compact spin manifold with boundary, equipped with a positive scalar metric $g$ and having fundamental group $\Gamma=\pi_1(M)$ which is finitely generated and of polynomial growth.
\end{abstract}
\renewcommand\contentsname{Table of Contents}
\tableofcontents
\pagestyle{myheadings}	

\pagenumbering{arabic}	
\section{Introduction}
\noindent Given a Fredholm operator $T:X\longrightarrow Y$ between two Banach spaces the classic index theory for Fredholm operators provides an integer valued analytic index \[\mathrm{ind}(T)=\dim\ker(T)-\dim\mathrm{coker}(T)\] 
which is invariant under perturbations of $T$ by compact operators. The non-vanishing of $\mathrm{ind}(T)$ is thus an obstruction to invertibility of a Fredholm operator $T$. When $T$ is an elliptic differential operator with $X$ and $Y$ smooth vector bundles over a smooth closed manifold $M$ the work of Atiyah and Singer \cite{AS63} showed the equivalence between $\mathrm{ind}(T)$ and the often more tractable topological index (see \eqref{TopInd} of \Cref{C4}). 

Let $M$ be a complete $n$-dimensional Riemannian manifold with a discrete group $G$ acting on it properly and cocompactly by isometries. Each $G$-equivariant elliptic differential operator $D$ on $M$ gives rise to a higher index class $\mathrm{Ind}_G(D)$ in the $K$-theory group $K_n(C_r^*(G))$ of the reduced group $C^*$-algebra $C^*_r(G)$. Higher index classes are invariant under homotopy, and being an obstruction to the invertibility of $D$, are often referred to as primary invariants. Higher index theory provides a far-reaching generalization of the Fredholm index by taking into consideration the symmetries of the underlying spaces; in particular, if $M$ is a complete compact Riemannian manifold with an associated Dirac-type operator $D$, a higher index theory intrinsically involves the fundamental group $\pi_1(M)$. The higher index theory plays a fundamental role in the studies of many important open problems having relations to geometry and topology, such as the Novikov conjecture, the operator $K$-theoretic Baum-Connes conjecture, and the Gromov-Lawson-Rosenberg conjecture.   

A secondary higher invariant-- so called due to its natural appearance upon the vanishing of a primary invariant such as $\mathrm{Ind}_G(D)$-- was developed by Lott \cite{JL92} within the framework of noncommutative differential forms, for manifolds with fundamental groups of polynomial growth and $D$ invertible. Lott's work was heavily inspired by the work of Bismut and Cheeger on eta forms \cite{BC89}, which naturally arise in the index theory for families of manifolds with boundary \cite{APS75I}. Lott's higher eta invariant, despite being defined by an explicit integral formula of noncommutative differential forms is unfortunately difficult to compute in general. To reduce the computability difficulty and make this higher invariant more applicable to problems in geometry and topology it is useful to introduce (periodic) cyclic cohomology. The delocalized eta invariant of Lott \cite{JL99} can be formally thought of as precisely such a pairing with respect to traces (see the formula \eqref{Ca:LottsPairing}). 

Given a delocalized cyclic cocycle class $[\varphi_\gamma]\in HC^*(\mathbb{C}\pi_1(M),\mathrm{cl}(\gamma))$ of any degree, where $\gamma\in\pi_1(M)$ and the conjugacy class $\mathrm{cl}(\gamma)$ is not trivial, a higher analogue of Lott's delocalized eta invariant $\eta_{[\varphi_\gamma]}(\widetilde{D})$ is given in \Cref{defCa1}. The precise definition of these delocalized cyclic cocyles is found in \Cref{defAb2}, and the explicit formula for $\eta_{[\varphi_\gamma]}(\widetilde{D})$ is described in terms of the transgression formula for Connes-Chern character \cite{AC85,AC94}. The natural problem which arises is in determining when this \textit{delocalized higher eta invariant} can actually be rigorously well-defined and involves some subtle convergence issues. Importantly, we prove that with respect to a group $\Gamma$ of polynomial growth every delocalized cyclic cocyle class on the group algebra has a representative of polynomial growth. It is also essential to prove that the pairing is independent of the choice of representative of any given cocycle class. We are thus able to show that whenever $M$ possesses a finitely generated fundamental group of polynomial growth this higher analogue pairing of Lott’s higher eta invariant with (delocalized) cyclic cocycles is well-defined, under the conditions that the Dirac operator on $\widetilde{M}$ is invertible-- or more generally has a spectral gap at zero. 

\begin{mythm}
Let $M$ be a closed odd-dimensional spin manifold equipped with a positive scalar metric $g$, and fundamental group of polynomial growth. Supposing the Dirac operator $D$ has an associated lift $\widetilde{D}$ to the universal cover $\widetilde{M}$, the higher delocalized eta invariant $\eta_{[\varphi_\gamma]}(\widetilde{D})$ converges absolutely for every $[\varphi_\gamma]\in HC^{2m}(\mathbb{C}\Gamma,\mathrm{cl}(\gamma))$. Moreover, if
\[S_\gamma^*:HC^{2m}(\mathbb{C}{\pi_1(M)},\mathrm{cl}(\gamma))\longrightarrow HC^{2m+2}(\mathbb{C}{\pi_1(M)},\mathrm{cl}(\gamma))\]
denotes the delocalized Connes periodicity operator, then $\eta_{[S_\gamma\varphi_\gamma]}(\widetilde{D})=\eta_{[\varphi_\gamma]}(\widetilde{D})$.
\end{mythm}

When the higher index class of an operator is trivial-- given a specific trivialization-- a secondary index theoretic invariant naturally arises through a $C^*$-algebraic approach.
For example, consider the associated Dirac operator on the universal covering $\widetilde{M}$ of a closed, $n$-dimensional spin manifold $M$ equipped with a positive scalar curvature metric $g$. The Lichnerowicz formula (see \eqref{Ca:Lichnerowicz} of \Cref{C1}) asserts that the Dirac operator on $\widetilde{M}$ is invertible \cite{AL63}, and so $\mathrm{Ind}_G(D)$ must necessarily be trivial. In this case, there is a natural $C^*$-algebraic secondary invariant $\rho(\widetilde{D},\widetilde{g})$ introduced by Higson and Roe \cite{HR05I,HR05II,HR05III}, called the higher rho invariant, which is an obstruction to the inverse of the Dirac operator being local. This higher rho invariant describes a class belonging to the group $K_n(C^*_{L,0}(\widetilde{M})^{\pi_1(M)})$, where $\pi_1(M)$ is the fundamental group of $M$. As mentioned before, such a secondary index theoretic invariant often plays an important role in problems in geometry and topology (cf. \cite{WXY20,WY15,XY17}). The precise description of the geometric $C^*$-algebra $C^*_{L,0}(\widetilde{M})^{\pi_1(M)}$ is provided in \Cref{defAa5}, and the particular construction of the higher rho invariant is given at the beginning of \Cref{C2}. In the case that $\pi_1(M)$ is of polynomial growth we provide-- using the construction of the determinant map of \cite{XY19}-- in \Cref{C2} an explicit formula (see \Cref{defCb3})for a pairing of $C^*$-algebraic secondary invariants and delocalized cyclic cocycles of the group algebra is realized. Moreover, in the particular instance that $[u]\in K_1(C^*_{L,0}(\widetilde{M})^{\pi_1(M)})$ is the $K$-theory class of the higher rho invariant $\rho(\widetilde{D},\widetilde{g})$, then the pairing is given explicitly in terms of the higher delocalized eta invariant $\eta_{[\varphi_\gamma]}(\widetilde{D})$.

\begin{mythm}\label{thm2}
Let $M$ be a closed odd-dimensional spin manifold with fundamental group of polynomial growth, then every delocalized cyclic cocycle $[\varphi_\gamma]\in HC^{2m}(\mathbb{C}{\pi_1(M)},\mathrm{cl}(\gamma))$ induces a natural map
\[\tau_{[\varphi_\gamma]}:K_1(C_{L,0}(\widetilde{M})^{\pi_1(M)})\longrightarrow\mathbb{C}\]
If $M$ has positive scalar curvature metric $g$ then $\tau_{[\varphi_\gamma]}(u)$ converges absolutely. When $[u]=\rho(\widetilde{D},\widetilde{g})$ is the $K$-theory class of the higher rho invariant there is an equivalence 
\[\tau_{[\varphi_\gamma]}(\rho(\widetilde{D},\widetilde{g}))=(-1)^m\eta_{[\varphi_\gamma]}(\widetilde{D})\]
\end{mythm}
The above theorem holds in the more general case that the Dirac operator $D$ on $M$ has an associated lift $\widetilde{D}$ to the universal cover $\widetilde{M}$ which is invertible. Showing that the map $\tau_{[\varphi_\gamma]}$ is well defined occupies the majority of \Cref{C2}; in particular, the extension of $\varphi_\gamma$ from the group algebra to the localization algebra requires the existence of a certain smooth dense subalgebra of $C^*_r(\pi_1(M))$ introduced by Connes and Moscovici \cite{CM90}. In \cite{XY19}, Xie and Yu established such a pairing between delocalized cyclic cocycles of degree $m=0$-- delocalized traces-- and classes $[u]$ belonging to $K_1(C^*_{L,0}(\widetilde{M})^{\pi_1(M)})$, under the assumption that the relevant conjugacy class has polynomial growth. Later, in \cite{CWXY19}, under the assumption that $\pi_1(M)$ is a hyperbolic group, this construction was extended to allow for a pairing between delocalized cyclic cocycles of all degrees and the K-theory classes $[u]\in K_n(C^*_{L,0}(\widetilde{M})^{\pi_1(M)})$. In the hyperbolic case, convergence of $\tau_{[\varphi_\gamma]}$ relies on the properties of Puschnigg's \cite{MP10} smooth dense subalgebra in an essential way.

The $C^*$-algebraic map $\tau_{[\varphi_\gamma]}$ allows for a constructive and explicit approach to a higher delocalized Atiyah-Patodi-Singer index theorem. In \Cref{C4} we prove a direct relationship between pairings of $K$-theory classes $[u]\in K_n(C^*_{L,0}(\widetilde{M})^{\pi_1(M)})$  with $\tau_{[\varphi_\gamma]}$, and the pairing of classes $\partial[u]$ with respect to the delocalized Chern-Connes character map \cite{AC85,AC94} \ (see \eqref{Cd:ChernConnes} for the explicit expression used here), where 
\[\partial:K_n(C^*(\widetilde{M})^{\pi_1(M)})\longrightarrow K_{n-1}(C^*_{L,0}(\widetilde{M})^{\pi_1(M)})\]
is the usual $K$-theory connecting map. Combined with \Cref{thm2} this provides the following version of a higher delocalized Atiyah-Patodi-Singer index theorem.

\begin{mythm}
Let $W$ be a compact spin manifold with boundary, equipped with a scalar curvature metric $g$ which is positive on $\partial W$, and fundamental group which is of polynomial growth. Denote by $\widetilde{D}_W$ and $\widetilde{D}_{\partial W}$ the lifted Dirac operators on $W$ and its boundary, respectively.
\[\mathsf{ch}_{[\varphi_\gamma]}\left(\mathrm{Ind}_{\pi_1(W)}(\widetilde{D}_W)\right)=\frac{(-1)^{m+1}}{2}\eta_{[\varphi_\gamma]}(\widetilde{D}_{\partial W})\]
for any $[\varphi_\gamma]\in HC^{2m}(\mathbb{C}{\pi_1(M)},\mathrm{cl}(\gamma))$,
where  $\mathsf{ch}_{[\varphi_\gamma]}$ is the delocalized Chern-Connes character map which pairs cyclic cocyles with the $K$-theoretic index class.
\end{mythm}

There have been various versions of a higher Atiyah-Patodi-Singer theorem in the literature, such as \cite{LP97,LP99,GMP16} and \cite{CW13}. The form of the this result strongly mirrors that conjectured by Lott \cite[Conjecture 1]{JL92}, and is essentially a general case of that proven by Xie and Yu \cite[Proposition 5.3]{XY19} for zero dimensional cyclic cocyles. In \Cref{C4} we provide the basic background of the original APS index theorem, and show how the above theorem is specifically related to it. See the discussion following \cite[Theorem 7.3]{CWXY19} for more details on the relationships and differences of the above theorem with other existing results of higher APS index theorems.

This paper is organized as follows. In \Cref{A1} we provide the properties of the geometric $C^*$-algebras which shall be used throughout, as well as detail the construction of important smooth dense sub-algebras. \Cref{A2} is concerned primarily with providing the definition of cyclic cohomology and detailing the relationship between this and cohomology for groups; in addition we recall an essential construction for explicit representative of cyclic cocyle classes. In \Cref{B1} we review the long exact sequence of periodic cyclic cohomology involving the (delocalized) Connes periodicity operator (see \Cref{defBa1} and \eqref{Ba:ConnesPer}); combining this with a cohomomological dimension result we prove a necessary torsion argument. We are thus able to construct a rational isomorphism between cohomology groups of a certain complex of cyclic cocyles and group cohomology of a particular subgroup of $\pi_1(M)$. Using the universal classifying description of group cohomology and the previous rational isomorphism, the entirety of \Cref{B2} is devoted to proving that every delocalized cyclic cocyle has a representative of polynomial growth. In \Cref{C1}, given a delocalized cyclic cocyle of polynomial growth we define a higher analogue of Lott's delocalized eta invariant and prove it converges for invertible elliptic operators. In \Cref{C2}, we first review a construction of Higson and Roe's higher rho invariant as an explicit $K$-theory class. We provide an explicit formula for the pairing between $C^*$-algebraic secondary invariants and delocalized cyclic cocycles of the group algebra for groups of polynomial growth, and prove it is well-defined. In particular, in the case that the secondary invariant is a $K$-theoretic higher rho invariant of an invertible elliptic differential operator, we show in \Cref{C3} that this pairing is precisely the higher delocalized eta invariant of the given operator. In \Cref{C4}, we use the determinant map of the previous section to determine a pairing between delocalized cyclic cocycles and $C^*$-algebraic Atiyah-Patodi-Singer index classes for manifolds with boundary, when the fundamental group of the given manifold is of polynomial growth.
\newline

\noindent\textbf{Acknowledgements:} The author would like to thank Zhizhang Xie for his invaluable support and advice.

\section{Preliminaries}
In all that follows we will take $M$ to be a closed odd-dimensional spin manifold, which is equipped with a positive scalar metric $g$. By $D$ we denote the Dirac operator associated to $M$, and analogously by $\widetilde{D}$ the associated lift to the universal cover $\widetilde{M}$. By $\Gamma=\pi_1(M)$ we refer to a countable discrete finitely generated group which is also the fundamental group of $M$. Given $\gamma\in\Gamma$ the centralizer of $\gamma$ will be denoted $Z_\Gamma(\gamma)$, or if there is no confusion as to the group $\Gamma$, by $Z_\gamma$; likewise, if $\gamma^{\mathbb{Z}}$ is the cyclic group generated by $\gamma$, then the quotient group $Z_\gamma/\gamma^{\mathbb{Z}}$ will be denoted by $N_\gamma$. By $\mathbb{C}\Gamma$ and $\mathbb{Z}\Gamma$ we mean the group algebra with complex coefficients and the group ring with integer coefficients, respectively.

We recall that a finitely generated discrete group $\Gamma$ comes equipped with a length function $l_S$-- with respect to some given symmetric generating set $S\subset\Gamma$.
\begin{equation}
l_S(g)=\min\{c\in\mathbb{N}:\exists s_1,\ldots,s_c\in S,s_1\cdots s_c=g\}
\end{equation}
There exists an associated word metric $d_S(g,h)=||g^{-1}h||;=l_S(gh)$ which is left-invariant with respect to the group action. More importantly, since the metric spaces $(\Gamma,S)$ and $(\Gamma, T)$ are quasi-isomorphic for any choice of generating sets $S$ and $T$, we are able to ignore this choice when dealing with the word metric (or length function). Henceforth we will merely refer to \textit{the} length function $l_\Gamma$ or \textit{the} word metric $d_\Gamma$. Unless otherwise stated we will assume throughout that $\Gamma$ is of polynomial growth, which is defined as there existing positive integer constants $C_0$ and $m$ such that
\begin{equation}
|\{g\in\Gamma:||g||\leq n\}|\leq C_0(n+1)^m\quad\forall n\in\mathbb{N}
\end{equation}

\subsection{Geometric $C^*$-algebras, and Smooth Dense Sub-algebras}\label{A1}
Let $X$ be a proper metric space, and $C_0(X)$ the algebra of continuous functions on $X$ which vanish at infinity. An $X$-module $H_X$ is a separable Hilbert space equipped with a $*$-representation $\pi:C_0(X)\longrightarrow\mathcal{B}(H_X)$ into the algebra of bounded operators on $H_X$, and is called non-degenerate if the $*$-representation of $C_0(X)$ is non-degenerate. If no non-zero function $f\in C_0(X)$ acts a compact operator under this $*$-representation, then we call $H_X$ a standard $X$-module.

\begin{mydef}\label{defAa1}
Recall that an operator $T$ acting on a Hilbert space $\mathcal{H}$ belongs to the algebra of compact operators $\mathcal{K}\subset\mathcal{B}(\mathcal{H})$ if the image under $T$ of every bounded subset has compact closure.
\begin{enumerate}[label={\bf(\roman{enumi})}] 
\item Let $T\in\mathcal{B}(H_X)$ be a bounded linear operator acting on $H_X$, then $T$ is locally compact if for all $f\in C_0(X)$ both $fT$ and $Tf$ are compact operators. We similarly call $T$ pseudo-local if the weaker condition, $[T,f]=TF-fT$ is a compact operator for all $f\in C_0(X)$, is satisfied.  
\item Again assume that $T$ belongs to $\mathcal{B}(H_X)$; the propagation of $T$ is defined to be
\[\sup\{d(x,y): (x,y)\in\mathsf{Supp}(T)\}\]
where $\mathsf{Supp}(T)$ denotes the support of $T$, which is the set
\[\{(x,y)\in X\times X:\exists f,g\in C_0(X)\;\mathrm{such\;that}\;gTf=0\;\mathrm{and}\;f(x)\neq0,g(y)\neq0\}^c\]
\end{enumerate}
\end{mydef}

If we further impose that $H_X$ is a standard and non-degenerate $X$-module, then there exist important constructions of certain geometric $C^*$-algebras. The first two of these, described in \Cref{defAa2} were introduced by Roe in \cite{JR93}, and the coarse homotopy invariance of their $K$-theory was subsequently proven by Higson and Roe \cite{HR94}. 
\begin{mydef}\label{defAa2}
The $C^*$-algebra generated by all locally compact operators with finite propagation in $\mathcal{B}(H_X)$, is the Roe algebra of $X$ and is denoted by $C^*(X)$. If we instead consider the $C^*$-algebra generated by all pseudo-local operators with finite propagation in $\mathcal{B}(H_X)$, then we obtain a related algebra $D^*(X)$. In fact, $D^*(X)$ is a subalgebra of the multiplier algebra $\mathcal{M}(C^*(X))$-- which is the largest unital $C^*$-algebra containing $C^*(X)$ as an ideal. 
\end{mydef}
\begin{mydef}\label{defAa3}
Let $\mathsf{prop}(T)$ denote the propagation of an operator $T\in\mathcal{B}(H_X)$. The localization algebras $C^*_L(X)$ and $D^*_L(X)$ introduced by Yu \cite{GY97} are defined as the $C^*$-algebras generated by $S_1$ and $S_2$ respectively, where $f$ is bounded and uniformly norm-continuous
\[S_1=\left\{f:[0,\infty)\longrightarrow C^*(X)|\lim\limits_{t\longrightarrow\infty}\mathsf{prop}(f(t))=0\right\}\]
\[S_2=\left\{f:[0,\infty)\longrightarrow D^*(X)|\lim\limits_{t\longrightarrow\infty}\mathsf{prop}(f(t))=0\right\}\]
Once again $D^*_L(X)$ is a subalgebra of the multiplier algebra $\mathcal{M}(C^*_L(X))$. The kernel of the evaluation map $\mathsf{ev}:C^*_L(X)\longrightarrow C^*(X)$ defined by $\mathsf{ev}(f)=f(0)$ is an ideal of $C^*_L(X)$, and is itself a $C^*$-algebra which we denote by $C_{L,0}^*(X)$. Analogously we also define the $C^*$-algebra $D_{L,0}^*(X)$ as the kernel of $\mathsf{ev}:D^*_L(X)\longrightarrow D^*(X)$. 
\end{mydef}

It follows that the Roe algebra and its localization fit into a short exact sequence-- analogously for $D^*(X)$-- which give rise to a six term $K$-theoretic long exact sequence with connecting map $\partial$, and for which $i=0,1(\mathrm{mod}\;2)$ by Bott periodicity.
\begin{equation}
\begin{tikzcd} 0 \arrow{r} & C_{L,0}^*(X) \arrow[hookrightarrow]{r}  & C_{L}^*(X) \arrow{r}{\mathsf{ev}} &  C^*(X) \arrow{r} & 0\end{tikzcd}\end{equation}
\begin{equation}
\begin{tikzcd}
K_i(C_{L,0}^*(X)) \arrow{r} & K_i(C_L^*(X)) \arrow{r} & K_i(C^*(X)) \arrow{d}{\partial}\\
K_{i-1}(C^*(X)) \arrow{u}{\partial} & K_{i-1}(C_L^*(X)) \arrow{l} & K_{i-1}(C_{L,0}^*(X))\arrow{l}\end{tikzcd}
\end{equation}
Assuming that a group $G$ acts properly and cocompactly on $X$ by isometries, we can equip $H_X$ with a covariant unitary representation of $G$, which we will denote by $\varpi$. Explicitly, if $g\in G,f\in C_0(X)$ and $v\in H_X$
\[\varpi(g)(\pi(f)v)=\pi(f^g)(\varpi(g)v)\]
where $f^g(x)=f(g^{-1}x)$. We call the system $(H_X,\pi,\varpi)$ a covariant system. 
\begin{mydef}\label{defAa4}
Suppose that $H_X$ is a standard and non-degenerate $X$-module, and $G$ acts on $X$ properly and cocompactly. Moreover, for each $x\in X$ the action of the stabilizer group $G_x$ on $H_X$ is isomorphic to the action of $G_x$ on $l^2(G_x)\otimes\mathcal{H}$ for some infinite dimensional Hilbert space $\mathcal{H}$, where $G_x$ acts trivially on $\mathcal{H}$ and by translations on $l^2(G_x)$. Under these conditions a covariant system $(H_X,\pi,\varpi)$ is called \textit{admissible}.
\end{mydef}
If it is not necessary to emphasize the representations we shall simply refer to the admissible system $(H_X,\pi,\varpi)$ by $H_X$, and describe it as an admissible $(X,G)$-module.
\begin{myrem}\label{remAa1}
For every locally compact metric space $X$ which admits a proper and cocompact isometric action of $G$, there exists an admissible covariant system $(H_X,\pi,\varpi)$. 
\end{myrem}
\begin{mydef}\label{defAa5}
Consider a locally compact metric space $X$ which admits a proper and cocompact isometric action of $G$, and fix some admissible $(X,G)$-module $H_X$. The $G$-equivariant Roe algebra $C^*(X)^G$ is the completion in $\mathcal{B}(H_X)$ of the $*$-algebra $\mathbb{C}[X]^G$ of all $G$-invariant locally compact operators with finite propagation in $\mathcal{B}(H_X)$. Replacing $G$-invariant locally compact operators with $G$-invariant pseudo-local operators we similarly obtain $D^*(X)^G$. The $G$-equivariant localization algebras $C^*_L(X)^G$ and $D^*_L(X)^G$ are defined as the $C^*$-algebras generated by $S_1$ and $S_2$ respectively, where $f$ is bounded and uniformly norm-continuous
\[S_1=\left\{f:[0,\infty)\longrightarrow C^*(X)^G|\lim\limits_{t\longrightarrow\infty}\mathsf{prop}(f(t))=0\right\}\]
\[S_2=\left\{f:[0,\infty)\longrightarrow D^*(X)^G|\lim\limits_{t\longrightarrow\infty}\mathsf{prop}(f(t))=0\right\}\]
Analogous to \Cref{defAa3} we can also define the ideals $C^*_{L,0}(X)^G$ and $D^*_{L,0}(X)^G$ as the kernels of the evaluation map. 
\end{mydef}
The equivariant Roe algebra-- analogously for $D^*(X)^G$-- fits into similar short exact sequence as did the original Roe algebra
\[\begin{tikzcd} 0 \arrow{r} & C_{L,0}^*(X)^G \arrow[hookrightarrow]{r}  & C_{L}^*(X)^G \arrow{r}{\mathsf{ev}} &  C^*(X)^G \arrow{r} & 0\end{tikzcd}\]
An especially useful consequence of the cocompact action of $G$ on $X$ is that there exists a $*$-isomorphism between $C_r^*(G)\otimes\mathcal{K}$ and $C^*(X)^G$, where $C_r^*(G)$ is the reduced group $C^*$-algebra of $G$. 
\begin{myrem}\label{remAa2}
The geometric $C^*$-algebras defined in \Cref{defAa3} and \Cref{defAa5} are all unique up to isomorphism, independent of the choice of $H_X$ is a standard and non-degenerate $X$-module. Likewise the $G$-equivariant versions are also, up to isomorphism, independent of the choice of admissible $(X,G)$-module $H_X$.
\end{myrem}

Let $\Gamma$ and $M$ be as described above; we turn our attention to construction of two important smooth dense subalgebras  of $C_r^*(\Gamma)\otimes\mathcal{K}\cong C^*(\widetilde{M})^\Gamma$, the first of which is essentially a slight modification of Connes and Moscovici's \cite{CM90}.

\begin{mydef}\label{defAa6}
Fixing a basis of $L^2(M)$, the algebra $\mathscr{R}$ of smooth operators on $M$ can be identified with the algebra of matrices $(a_{ij})_{i,j\in\mathbb{N}}$ satisfying
\[\sup_{i,j\in\mathbb{N}}i^kj^l|a_{ij}|<\infty\quad\forall k,l\in\mathbb{N}\]
Consider the unbounded operators $\Delta_1:\ell^2(\mathbb{N})\longrightarrow\ell^2(\mathbb{N})$ and $\Delta_2:\ell^2(\Gamma)\longrightarrow\ell^2(\Gamma)$ defined on basis elements according to
\[\Delta_1(\delta_j)=j\delta_j,\;j\in\mathbb{N}\qquad\mathrm{and}\qquad\Delta_2(g)=||g||\cdot g,\;g\in\Gamma\] 
Denoting by $I$ the identity operator and with $[\cdot,\cdot]$ being the usual commutator bracket, we have unbounded derivations $\partial(T)=[\Delta_2, T]$ of operators $T\in\mathcal{B}(\ell^2(\Gamma))$ and unbounded derivations $\widetilde{\partial}(T)=[\Delta_2\otimes I,T]$ of operators $T\in\mathcal{B}(\ell^2(\Gamma)\otimes\ell^2(\mathbb{N}))$. Define an algebra
\[\mathscr{B}(\widetilde{M})^\Gamma=\{A\in C_r^*(\Gamma)\otimes\mathcal{K}:\;\widetilde{\partial}^k(A)\circ(I\otimes\Delta_1)^2\;\mathrm{is\;bounded\;}\forall k\in\mathbb{N}\}\]
\end{mydef}

The crucial property of $\mathscr{B}(\widetilde{M})^\Gamma$ is that it contains $\mathbb{C}\Gamma\otimes\mathscr{R}$ as a dense subalgebra, is itself a smooth dense subalgebra of $C^*(\widetilde{M})^\Gamma$, and it is closed under holomorphic functional calculus. Moreover, $\mathscr{B}(\widetilde{M})^\Gamma$ is a Fr\'{e}chet algebra under the sequence of seminorms $\{||\cdot||_{\mathscr{B},k}:\;k\in\mathbb{N}\}$, where $||A||_{\mathscr{B},k}=||\widetilde{\partial}^k(A)\circ(I\otimes\Delta_1)^2||_{op}$ is the operator norm of $\widetilde{\partial}^k(A)\circ(I\otimes\Delta_1)^2$.

\begin{mydef}\label{defAa7}
We define a kind of localization algebra $\mathscr{B}_L(\widetilde{M})^\Gamma$ associated to $\mathscr{B}(\widetilde{M})^\Gamma$, which by construction is a smooth dense subalgebra of $C^*_L(\widetilde{M})^\Gamma$ and thus is closed under holomorphic functional calculus.
\[\mathscr{B}_L(\widetilde{M})^\Gamma=\{f\in C^*_L(\widetilde{M})^\Gamma:\;f\;\mathrm{is\;piecewise\;smooth\;w.r.t\;}t\;,f(t)\in\mathscr{B}(\widetilde{M})^\Gamma\;\forall t\in[0,\infty)\}\]
and also define $\mathscr{B}_{L,0}(\widetilde{M})^\Gamma$ to be the kernel of the usual evaluation map $\mathsf{ev}:\mathscr{B}_L(\widetilde{M})^\Gamma\longrightarrow\mathscr{B}(\widetilde{M})^\Gamma$ defined by $\mathsf{ev}(f)=f(0)$.
\end{mydef}
\begin{mypro}\label{proAa1}
The inclusions $\mathscr{B}_L(\widetilde{M})^\Gamma\hookrightarrow C^*_L(\widetilde{M})^\Gamma$ and $\mathscr{B}_{L,0}(\widetilde{M})^\Gamma\hookrightarrow C^*_{L,0}(\widetilde{M})^\Gamma$ induce isomorphisms on $K$-theory
\[K_i(\mathscr{B}_L(\widetilde{M})^\Gamma)\cong K_i(C^*_L(\widetilde{M})^\Gamma)\qquad K_i(\mathscr{B}_{L,0}(\widetilde{M})^\Gamma)\cong K_i(C^*_{L,0}(\widetilde{M})^\Gamma)\]
\end{mypro}
\begin{proof}
This follows immediately from the definitions of the smooth dense subalgebras.
\end{proof}
We now look at the second kind of smooth dense subalgebra of $C^*(\widetilde{M})^\Gamma$, this time working more directly with $\widetilde{M}$. Let $A$ belong to the algebra $C^\infty(\widetilde{M}\times\widetilde{M})$ of smooth functions on $\widetilde{M}\times\widetilde{M}$, and assume that $A$ is both $\Gamma$-invariant and of finite propagation. Explicitly, we mean that
\[A(gx,gy)=A(x,y)\quad\forall g\in\Gamma\]
\[\exists R>0\;\mathrm{such\;that\;}A(x,y)=0,\;\forall(x,y)\in\widetilde{M}\times\widetilde{M}\;\mathrm{satisfying\;}d_{\widetilde{M}}(x,y)>R\]
\begin{mydef}\label{defAa8}
Denote by $\mathscr{L}(\widetilde{M})^\Gamma$ the convolution algebra of $A\in C^\infty(\widetilde{M}\times\widetilde{M})$ which are both $\Gamma$-invariant and of finite propagation. The action of $\mathscr{L}(\widetilde{M})^\Gamma$ on $L^2(\widetilde{M})$ is according to
\[(Af)(x)=\int_{\widetilde{M}}A(x,y)f(y)\;dy\qquad\mathrm{for}\qquad A\in\mathscr{L}(\widetilde{M})^\Gamma,\; f\in L^2(\widetilde{M})\]
Denote by $\hat{\rho}:\widetilde{M}\longrightarrow[0,\infty)$ the distance function $\hat{\rho}(x)=\hat{\rho}(x,y_0)$ for some fixed point $y_0\in\widetilde{M}$, with $\rho$ being the modification of $\hat{\rho}$ near $y_0$ to ensure smoothness. The multiplication operator $T_\rho$ thus acts as an unbounded operator on $ L^2(\widetilde{M})$, according to $(T_\rho f)(x)=\rho(x)f(x)$. Using the commutator bracket we can define a derivation $\widetilde{\partial}=[T_\rho,\cdot]:\mathscr{L}(\widetilde{M})^\Gamma\longrightarrow\mathscr{L}(\widetilde{M})^\Gamma$.
\[\mathscr{A}(\widetilde{M})^\Gamma=\{A\in C^*(\widetilde{M})^\Gamma:\;\widetilde{\partial}^k(A)\circ(\Delta+1)^{n_0}\;\mathrm{is\;bounded\;}\forall k\in\mathbb{N}\}\]
where $\Delta$ is the Laplace operator on $\widetilde{M}$, and $n_0$ is a fixed integer greater than $\dim(M)$. The associated norm is given by $||A||_{\mathscr{A},k}=||\widetilde{\partial}^k(A)\circ(\Delta+1)^{n_0}||_{op}$, which is the operator norm of $\widetilde{\partial}^k(A)\circ(\Delta+1)^{n_0}$.
\end{mydef}
The same proof of Connes and Moscovici \cite[Lemma 6.4]{CM90} shows that $\mathscr{A}(\widetilde{M})^\Gamma$ is closed under holomorphic functional calculus, and contains $\mathscr{L}(\widetilde{M})^\Gamma$ as a subalgebra. Before proceeding to define the generalized higher eta invariant in \Cref{C1} we first recall a necessary extension of $\mathscr{A}(\widetilde{M})^\Gamma$, by introducing bundles. Consider the bundle on $\widetilde{M}\times\widetilde{M}$ given by $\mathsf{End}(\mathcal{S})=p_1^*(\mathcal{S})\otimes p_2^*(\mathcal{S}^*)$, where $p_i:\widetilde{M}\times\widetilde{M}\longrightarrow\widetilde{M}$ are the obvious projection maps, with $\mathcal{S}$ and $\mathcal{S}^*$ being the spinor bundle on $\widetilde{M}$ and its dual bundle, respectively. Consider the set $C^\infty(\widetilde{M}\times\widetilde{M},\mathsf{End}(\mathcal{S}))$ of all smooth sections of the bundle $\mathsf{End}(\mathcal{S})$ on $\widetilde{M}\times\widetilde{M}$, and note that there exists a natural diagonal action of $\Gamma$ on $\mathsf{End}(\mathcal{S})$. Thus, we can construct $\mathscr{L}(\widetilde{M},\mathcal{S})^\Gamma$ as the convolution algebra of all $\Gamma$-invariant finite propagation elements of $C^\infty(\widetilde{M}\times\widetilde{M},\mathsf{End}(\mathcal{S}))$. Let $L^2(\widetilde{M},\mathcal{S})$ denote the the space of $L^2$-sections of $\mathcal{S}$ over $\widetilde{M}$; there is an action of $\mathscr{L}(\widetilde{M},\mathcal{S})^\Gamma$ on $L^2(\widetilde{M},\mathcal{S})$
\begin{equation}
(Af)(x)=\int_{\widetilde{M}}A(x,y)f(y)\;dy\qquad A\in\mathscr{L}(\widetilde{M},\mathcal{S})^\Gamma\quad f\in L^2(\widetilde{M},\mathcal{S})\end{equation}
Now since $L^2(\widetilde{M},\mathcal{S})$ is an admissible $(\widetilde{M},\mathcal{S})$-module we can construct the $\Gamma$-equivariant Roe algebra $C^*(\widetilde{M},\mathcal{S})^\Gamma$ associated to it; however by \Cref{remAa2} Roe algebras are up to isomorphism independent of the choice of admissible module. Thus, we will also denote by $C^*(\widetilde{M})^\Gamma$ the $\Gamma$-equivariant Roe algebra constructed with respect to $L^2(\widetilde{M},\mathcal{S})$.

\begin{mydef}\label{defAa9}
Let $\widetilde{D}$ be the Dirac operator on $\widetilde{M}$, and fix some integer $n_0>\mathrm{dim}M$, then 
\[\mathscr{A}(\widetilde{M},\mathcal{S})^\Gamma=\{A\in C^*(\widetilde{M})^\Gamma:\;\widetilde{\partial}^k(A)\circ(\widetilde{D}^{2n_0}+1)\;\mathrm{is\;bounded\;}\forall k\in\mathbb{N}\}\]
where $\widetilde{\partial}=[T_\rho,\cdot]$ is the derivation on $\mathscr{L}(\widetilde{M},\mathcal{S})^\Gamma$ if we take $T_\rho$ to be the multiplication operator on $L^2(\widetilde{M},\mathcal{S})$. The algebras $\mathscr{A}_L(\widetilde{M},\mathcal{S})^\Gamma$ and $\mathscr{A}_{L,0}(\widetilde{M},\mathcal{S})^\Gamma$ are defined analogously to those in \Cref{defAa7}. The associated norm is given by $||A||_{\mathscr{A},\mathcal{S},k}=||\widetilde{\partial}^k(A)\circ(\widetilde{D}^{2n_0}+1)||_{op}$, which is the operator norm of $\widetilde{\partial}(A)^k\circ(\widetilde{D}^{2n_0}+1)$.
\end{mydef}
If there is no cause for confusion, we shall remove the explicit spinor notation and simply denote the above norm on $\mathscr{A}(\widetilde{M},\mathcal{S})^\Gamma$ by $||A||_{\mathscr{A},k}$. We end this section with a brief reminder of the notion of projective tensor product $\mathcal{A}^{\hat{\otimes}_\pi^m}$ with respect to any of the $*$-algebras constructed above. If $\mathcal{A}\otimes\mathcal{B}$ is the algebraic tensor product, then recall that the projective tensor product $\mathcal{A}\hat{\otimes}_\pi\mathcal{B}$ is the completion of $\mathcal{A}\otimes\mathcal{B}$ with respect to the projective cross norm
\begin{equation}
\pi(x)=\inf\left\{\sum_{i=1}^{n_x}||A_i||_{\mathcal{A}}||B_i||_{\mathcal{B}}:\;x=\sum_{i=1}^{n_x}A_i\otimes B_i\right\}\end{equation}
where $||\cdot||_{\mathcal{A}}$ denotes the norm on $\mathcal{A}$. We will denote the norm on $\mathcal{A}^{\hat{\otimes}_\pi^m}$ by $||\cdot||_{\mathcal{A}^{\hat{\otimes}m}}$ and usually write elements of $\mathcal{A}^{\hat{\otimes}_\pi^m}$ as $A_1\hat{\otimes}\cdots\hat{\otimes}A_m$.

\subsection{Cyclic and Group Cohomology}\label{A2}
\begin{mydef}\label{defAb1}
Denote by $C^n(\mathbb{C}\Gamma)$ the cyclic module consisting of all $(n+1)$-functionals $f:(\mathbb{C}\Gamma)^{\otimes n+1}\longrightarrow\mathbb{C}$ together with maps $d_i:(\mathbb{C}\Gamma)^{\otimes n+1}\longrightarrow(\mathbb{C}\Gamma)^{\otimes n}$ defined according to
\[d_i(a_0\otimes\cdots\otimes a_n)=a_0\otimes\cdots\otimes a_{i-1}\otimes(a_ia_{i+1})\otimes a_{i+2}\otimes a_n\qquad\mathrm{for}\quad0\leq i<n\]
\[d_n(a_0\otimes\cdots\otimes a_n)=(a_na_0)\otimes a_1\otimes\cdots\otimes a_{n-1}\]
and a cyclic operator $\mathfrak{t}$, where $\mathfrak{t}f(a_0\otimes\cdots\otimes a_n)=(-1)^nf(a_n\otimes a_0\cdots\otimes a_{n-1})$. Define the coboundary differential $b:C^n(\mathbb{C}\Gamma)\longrightarrow C^n(\mathbb{C}\Gamma)$ by $b=\sum_{i=0}^{n+1}(-1)^i\delta^i$, where $\delta^i$ is the dual to $d_i$; that is $\langle\delta^if,a\rangle=\langle f,d_i(a)\rangle$. Hence we have
\[(bf)(a_0\otimes\cdots\otimes a_{n+1})=\sum_{i=0}^{n}(-1)^if(a_0\otimes\cdots\otimes(a_ia_{i+1})\otimes a_n)+(-1)^{n+1}f(a_{n+1}a_0\otimes a_1\otimes\cdots\otimes a_n)\]
The cohomology of the complex $(C^n(\mathbb{C}\Gamma),b)$ is the cyclic cohomology $HC^*(\mathbb{C}\Gamma)$.
\end{mydef}
\begin{mydef}\label{defAb2}
Fix $\gamma\in\Gamma$ and denote by $(\mathbb{C}\Gamma,\mathrm{cl}(\gamma))^{\otimes n+1}$ the subcomplex of $(\mathbb{C}\Gamma)^{\otimes n+1}$ spanned by all elements $(g_0,\ldots,g_n)\in\Gamma^{n+1}$ satisfying $g_0\cdots g_n\in\mathrm{cl}(\gamma)$, where $\mathrm{cl}(\gamma)$ is the conjugacy class of $\gamma$. This gives rise to a cyclic submodule $C^n(\mathbb{C}\Gamma,\mathrm{cl}(\gamma))$ of $C^n(\mathbb{C}\Gamma)$ which comprises the collection of functionals which vanish on $(g_0,\ldots,g_n)$ if $g_0\cdots g_n\notin\mathrm{cl}(\gamma)$. The coboundary differential $b$ preserves this cyclic subcomplex, and we thus denote the cohomology of $(C^n(\mathbb{C}\Gamma,\mathrm{cl}(\gamma)),b)$ by $HC^*(\mathbb{C}\Gamma,\mathrm{cl}(\gamma))$.
\end{mydef}

\begin{mydef}\label{defAb3}
By $H^*(N_\gamma,\mathbb{C})$ we are referring to the groups $\mathsf{Ext}^*_{\mathbb{Z}N_\gamma}(\mathbb{Z},\mathbb{C})$ defined over the projective $\mathbb{Z}N_\gamma$-resolution of $\mathbb{Z}$. Namely, consider the resolution
\[\begin{tikzcd}
\cdots \arrow{r} & \mathbb{Z}N_\gamma^{k+1} \arrow{r}{\partial_k} & \mathbb{Z}N_\gamma^k \arrow{r}{\partial_{k-1}} & \cdots \arrow{r}{\partial_1} & \mathbb{Z}N_\gamma \arrow{r}{\partial_0} & \mathbb{Z} \arrow{r}{\epsilon} & 0
\end{tikzcd}\]
where, if $\hat{h_i}$ denotes a deleted entry, $\partial_k$ acts on the basis elements according to 
\[\partial_k(h_0,\ldots,h_k)=\sum_{i=0}^{k}(-1)^k(h_0,\ldots,\widehat{h_i},\ldots,h_k)\]
Dropping the $\mathbb{Z}$ term and applying the contravariant functor $\mathsf{Hom}_{N_\gamma}(-,\mathbb{C})$ to this resolution produces a cochain complex with coboundary differential $\hat{b}$
\[\begin{tikzcd}
\cdots & \mathsf{Hom}_{N_\gamma}(\mathbb{Z}N_\gamma^k,\mathbb{C}) \arrow[l,"\hat{b}"'] & \cdots \arrow[l,"\hat{b}"'] & \arrow[l,"\hat{b}"'] \mathsf{Hom}_{N_\gamma}(\mathbb{Z}N_\gamma,\mathbb{C}) & \arrow[l,"\hat{b}"'] 0\end{tikzcd}\]
\[(\hat{b}\phi)(h_0,\cdots,h_{k+1})=\sum_{i=0}^{k+1}(-1)^i\phi(h_0,\ldots,\widehat{h_i},\ldots,h_{k+1})\]
The cohomology of this complex is defined to be the group cohomology $H^*(N_\gamma,\mathbb{C})$. 
\end{mydef}
Note that every cochain $\phi\in H^n(N_\gamma,\mathbb{C})$ satisfies the "homogeneous" condition: that is, for every $h\in N_\gamma$, $h\phi(h_0,\ldots,h_n)=\phi(hh_0,\ldots,hh_n)$. It will be extremely useful to replace the standard cochain complex with the sub-complex of homogeneous skew-cochains
\begin{equation}
\varphi(\sigma(h_0,h_1,\ldots,h_n))=\varphi(h_{\sigma(0)},h_{\sigma(1)},\ldots,h_{\sigma(n)})=\mathrm{sgn}(\sigma)\varphi(h_0,h_1,\ldots,h_n)\quad\forall\sigma\in S_{n+1}
\end{equation}
where $S_{n+1}$ is the symmetric group on $n+1$ letters. It is an immediate consequence of this definition that $\varphi(h_0,\ldots,h_n)$ vanishes whenever $h_i=h_j$ for $i\neq j$; just take $\sigma$ to be the permutation satisfying $\sigma(i)=j,\sigma(j)=i$, and which fixes all other indices. Define the map $F:\mathsf{Hom}_{N_\gamma}(\mathbb{Z}N_\gamma^n,\mathbb{C})\longrightarrow\mathsf{Hom}_{N_\gamma}(\mathbb{Z}N_\gamma^n,\mathbb{C})$ according to
\begin{equation}
(F\phi)(h_0,\ldots,h_n)=\frac{1}{(n+1)!}\sum_{\sigma\in S_{n+1}}\mathrm{sgn}(\sigma)\phi(\sigma(h_0,\ldots,h_n))
\end{equation}
\begin{mypro}\label{proAb1}
For every $\phi\in\mathsf{Hom}_{N_\gamma}(\mathbb{Z}N_\gamma^n,\mathbb{C})$ the cochain $F\phi$ is a skew cochain $\varphi$.
\end{mypro}
\begin{proof}
Let $\sigma_0$ be any fixed even permutation-- that is $\sigma_0$ is decomposable as an even number of 2-cycles, hence $\mathrm{sgn}(\sigma_0)=1$. Since left multiplication of any group on itself is a free and transitive action, it follows that for each $\sigma$ there exists a unique $\tau_\sigma$ such that $\sigma_0\tau_\sigma=\sigma$.
\[(F\phi)(\sigma_0(h_0,\ldots,h_n))=\frac{1}{(n+1)!}\sum_{\sigma\in S_{n+1}}\mathrm{sgn}(\sigma)\phi(\sigma_0\sigma(h_0,\ldots,h_n))\]
\[=\frac{1}{(n+1)!}\sum_{\sigma_0\tau_\sigma\in S_{n+1}}\mathrm{sgn}(\sigma_0\tau_\sigma)\phi(\sigma_0\sigma(h_0,\ldots,h_n))\]
\[=\frac{1}{(n+1)!}\sum_{\tau_\sigma\in S_{n+1}}\mathrm{sgn}(\tau_\sigma)\phi(\sigma(h_0,\ldots,h_n))\]
Since $\mathrm{sign}(\sigma_0)\mathrm{sign}(\tau_\sigma)=\mathrm{sign}(\sigma_0\tau_\sigma)=\mathrm{sgn}(\sigma)$ and $\sigma_0$ is an even permutation then $\tau_\sigma$ must have the same parity as $\sigma$. It follows that
\[(F\phi)(\sigma_0(h_0,\ldots,h_n))=\frac{1}{(n+1)!}\sum_{\tau_\sigma\in S_{n+1}}\mathrm{sgn}(\sigma)\phi(\sigma(h_0,\ldots,h_n))=(F\phi)(h_0,\ldots,h_n)\]
Follow the same argument, if $\sigma_0$ is an odd permutation then again for each $\sigma$ there exists a unique $\tau_\sigma$ such that $\sigma_0\tau_\sigma=\sigma$. However, since $\mathrm{sgn}(\sigma_0)=-1$ it follows that $\tau_\sigma$ must possess opposite parity to $\sigma$, hence  
\[(F\phi)(\sigma_0(h_0,\ldots,h_n))=\frac{1}{(n+1)!}\sum_{\tau_\sigma\in S_{n+1}}-\mathrm{sgn}(\sigma)\phi(\sigma(h_0,\ldots,h_n))=-(F\phi)(h_0,\ldots,h_n)\]
\end{proof}
\begin{mylem}\label{lemAb1}
The induced map $F^*:H^*(N_\gamma,\mathbb{C})\longrightarrow H^*(N_\gamma,\mathbb{C})$ is an isomorphism.
\end{mylem}
\begin{proof}
That $F$ induces an isomorphism on cohomology (with real or complex coefficients) follows if we can show $F\simeq\mathrm{Id}$ as chain complex maps. First a straightforward calculation proves that $F$ is a chain complex map, in the sense that the following diagram commutes for all $n$.
\[\begin{tikzcd}
\mathsf{Hom}_{N_\gamma}(\mathbb{Z}N_\gamma^{n+1},\mathbb{C}) \arrow[d,"F"]  & \mathsf{Hom}_{N_\gamma}(\mathbb{Z}N_\gamma^n,\mathbb{C}) \arrow[l,"\hat{b}"'] \arrow[d,"F"]\\
\mathsf{Hom}_{N_\gamma}(\mathbb{Z}N_\gamma^{n+1},\mathbb{C})  & \mathsf{Hom}_{N_\gamma}(\mathbb{Z}N_\gamma^n,\mathbb{C}) \arrow[l,"\hat{b}"']\end{tikzcd}\]
\[(\hat{b}\circ F\phi)(h_0,\ldots,h_{n+1})=\sum_{i=0}^{n+1}(-1)^i(F\phi)(h_0,\ldots,\widehat{h_i},\ldots,h_{n+1})\]
\[=\frac{1}{(n+1)!}\sum_{\sigma\in S_{n+1}}\mathrm{sgn}(\sigma)\sum_{i=0}^{n+1}(-1)^i\phi(\sigma(h_0,\ldots,\widehat{h_i},\ldots,h_{n+1}))\]
\[=\frac{1}{(n+1)!}\sum_{\sigma\in S_{n+1}}\mathrm{sgn}(\sigma)(\hat{b}\phi)(\sigma(h_0,\ldots,h_{n+1}))=(F\circ\hat{b}\phi)(h_0,\ldots,h_{n+1})\]
Now, $F$ is chain homotopic to $\mathrm{Id}$ on each $\mathsf{Hom}_{N_\gamma}(\mathbb{Z}N_\gamma^n,\mathbb{C})$ if there exists a sequence of maps $\{p_k|\;p_k:\mathsf{Hom}_{N_\gamma}(\mathbb{Z}N_\gamma^k,\mathbb{C})\longrightarrow\mathsf{Hom}_{N_\gamma}(\mathbb{Z}N_\gamma^{k-1},\mathbb{C})\}$ such that $F-\mathrm{Id}=b\circ p_n+p_{n+1}\circ b$. For ease of notation denote $\mathbf{h}_{n-1}=(h_0,\ldots,h_{n-1})$; we will define
\[(p_n\phi)(\mathbf{h}_{n-1})=\frac{(-1)^{n}}{(n+1)!}\sum_{\sigma\in S_{n+1}}\mathrm{sgn}(\sigma)\phi(\sigma(\mathbf{h}_{n-1},eh_{n-1}))-(-1)^n\phi(\mathbf{h}_{n-1},eh_{n-1})\] 
where $eh_{n-1}=h_{n-1}$ denotes a copy of $h_{n-1}$ inserted into the $n$'th position. For further ease of notation we will denote $(h_0,\ldots,\widehat{h_i},\ldots,h_n,eh_n)$ by $(\mathbf{h}_{n,\hat{i}},eh_n)$ for $i\leq n$.
\[(b\circ p_n\phi)(\mathbf{h}_n)=\sum_{i=0}^{n}(-1)^i(p_n\phi)(h_0,\ldots,\widehat{h_i},\ldots,h_n)\]
\[=\sum_{i=0}^{n}(-1)^i\left(\frac{(-1)^{n}}{(n+1)!}\sum_{\sigma\in S_{n+1}}\mathrm{sgn}(\sigma)\phi(\sigma(\mathbf{h}_{n,\hat{i}},eh_n))-(-1)^n\phi(\mathbf{h}_{n,\hat{i}},eh_n)\right)\]
\[(p_{n+1}\circ b\phi)(\mathbf{h}_n)=\frac{(-1)^{n+1}}{(n+2)!}\sum_{\sigma\in S_{n+1}}\mathrm{sgn}(\sigma)(b\phi)(\sigma(\mathbf{h}_n,eh_n))-(-1)^{n+1}(b\phi)(\mathbf{h}_n,eh_n)\]
\[=\sum_{i=0}^{n+1}(-1)^i\left(\frac{(-1)^{n+1}}{(n+2)!}\sum_{\sigma\in S_{n+1}}\mathrm{sgn}(\sigma)\phi(\sigma(\mathbf{h}_{n,\hat{i}},eh_n))-(-1)^{n+1}\phi(\mathbf{h}_{n,\hat{i}},eh_n)\right)\]
Using the fact that by \Cref{proAb1} the expressions
\[\frac{(-1)^{n}}{(n+1)!}\sum_{\sigma\in S_{n+1}}\mathrm{sgn}(\sigma)\phi(\mathbf{h}_{n,\hat{i}},eh_n)\quad\mathrm{and}\quad\frac{(-1)^{n+1}}{(n+2)!}\sum_{\sigma\in S_{n+1}}\mathrm{sgn}(\sigma)\phi(\mathbf{h}_{n,\hat{i}},eh_n)\]
vanish for all $i\leq n-1$ since $h_n=eh_n$, we thus have the reduced identities
\[(b\circ p_n\phi)(\mathbf{h}_n)=\frac{(-1)^{n}(-1)^{n}}{(n+1)!}\sum_{\sigma\in S_{n+1}}\mathrm{sgn}(\sigma)\phi(\sigma(\mathbf{h}_{n,\hat{n}},eh_n))-\sum_{i=0}^{n}(-1)^{i+n}\phi(\mathbf{h}_{n,\hat{i}},eh_n)\]
\[(p_{n+1}\circ b\phi)(\mathbf{h}_n)=\frac{1}{(n+2)!}\sum_{\sigma\in S_{n+1}}\mathrm{sgn}(\sigma)\left((-1)^{2n+2}\phi(\sigma(\mathbf{h}_{n},e\widehat{h_n}))+(-1)^{2n+1}\phi(\sigma(\mathbf{h}_{n,\hat{n}},eh_n))\right)\]
\[-\sum_{i=0}^{n+1}(-1)^{i+n}\phi(\mathbf{h}_{n,\hat{i}},eh_n)\]
\[=\sum_{i=0}^{n+1}(-1)^{i+n}\phi(\mathbf{h}_{n,\hat{i}},eh_n)+\frac{1}{(n+2)!}\sum_{\sigma\in S_{n+1}}0=\sum_{i=0}^{n+1}(-1)^{i+n}\phi(\mathbf{h}_{n,\hat{i}},eh_n)\]
where we have used the fact $(\mathbf{h}_{n,\hat{n}},eh_n)=(\mathbf{h}_{n},e\widehat{h_n})$. Moreover, it is readily apparent that both these tuples are also equal to $\mathbf{h}_n$; it follows that $(b\circ p_n\phi+p_{n+1}\circ b\phi)(\mathbf{h}_n)$ simplifies to exactly the expression for $(F-\mathrm{Id})\phi(\mathbf{h}_n)$
\[\frac{(-1)^{2n}}{(n+1)!}\sum_{\sigma\in S_{n+1}}\mathrm{sgn}(\sigma)\phi(\sigma(\mathbf{h}_{n,\hat{n}},eh_n))+\sum_{i=0}^{n+1}(-1)^{i+n}\phi(\mathbf{h}_{n,\hat{i}},eh_n)-\sum_{i=0}^{n}(-1)^{i+n}\phi(\mathbf{h}_{n,\hat{i}},eh_n)\]
\[=\frac{1}{(n+1)!}\sum_{\sigma\in S_{n+1}}\mathrm{sgn}(\sigma)\phi(\sigma(\mathbf{h}_{n}))-\phi(\mathbf{h}_{n})=(F-\mathrm{Id})\phi(\mathbf{h}_n)\]
\end{proof}

For the remainder of this paper, when referring to group cohomology it will be with respect to the subcomplex of skew cochains. The following splitting of cyclic (co)-homology was proven by Burghelea \cite{DB85} using topological arguments, and Nistor \cite{VN90} provided a later algebraic proof.
\begin{mythm}\label{thmAb1}
\[HC^*(\mathbb{C}\Gamma)\cong\prod_{\mathrm{cl}(\gamma)}HC^*(\mathbb{C}\Gamma,\mathrm{cl}(\gamma))\]
Moreover there exist isomorphisms with group (co)-homology
\[HC^*(\mathbb{C}\Gamma,\mathrm{cl}(\gamma))\cong\left\{\begin{array}{c c}
H^*(N_\gamma,\mathbb{C}) & \gamma\;\mathrm{is\;of\;infinite\;order}\\
H^*(N_\gamma,\mathbb{C})\otimes_{\mathbb{C}}HC^*(\mathbb{C}) & \gamma\;\mathrm{is\;of\;finite\;order}
\end{array}\right.\]
\end{mythm}

\begin{mydef}\label{defAb4}
Fix a group element $\gamma$ with conjugacy class $\mathrm{cl}(\gamma)$, and let $C^k(\Gamma,Z_\gamma,\gamma)$ be the collection of all multilinear forms $\alpha:\Gamma^{k+1}\longrightarrow\mathbb{C}$ satisfying
\[\alpha(g_{\sigma(0)},g_{\sigma(1)},\ldots,g_{\sigma(n)})=\mathrm{sgn}(\sigma)\alpha(g_0,g_1,\ldots,g_k)\quad\forall\sigma\in S_{k+1}\]
\[\alpha(zg_0,zg_1,\ldots,zg_k)=\alpha(g_0,g_1,\ldots,g_k)\quad\forall z\in Z_\gamma\]
\[\alpha(\gamma g_0,g_1,\ldots,g_k)=\alpha(g_0,g_1,\ldots,g_k)\]
The coboundary map $\hat{b}:C^k(\Gamma,Z_\gamma,\gamma)\longrightarrow C^{k+1}(\Gamma,Z_\gamma,\gamma)$ gives rise to a cochain complex $(C^n(\Gamma,Z_\gamma,\gamma),\hat{b})$, the cohomology of which we will denote by $H^*(\EuScript{C},\mathbb{C})$
\[\begin{tikzcd} \cdots & C^{k+1}(\Gamma,Z_\gamma,\gamma) \arrow[l,"\hat{b}"'] & C^k(\Gamma,Z_\gamma,\gamma) \arrow[l,"\hat{b}"'] & \cdots \arrow[l,"\hat{b}"'] & \arrow[l,"\hat{b}"'] C^0(\Gamma,Z_\gamma,\gamma) & \arrow[l,"\hat{b}"'] 0 \end{tikzcd}\]
\[(\hat{b}\phi)(g_0,\cdots,g_{k+1})=\sum_{i=0}^{k+1}(-1)^i\phi(g_0,\ldots,\widehat{g_i},\ldots,g_{k+1})\]
\end{mydef}

Recall that a cyclic cocyle of $(C^n(\mathbb{C}\Gamma),b)$ is a functional $\varphi$ which belongs to the kernel $ZC^n(\mathbb{C}\Gamma)$ of the coboundary differential. If $\mathrm{cl}(\gamma)$ is non-trivial we call the cyclic cocyles $\varphi_{\gamma}$ of $(C^n(\mathbb{C}\Gamma,\mathrm{cl}(\gamma)),b)$ \textit{delocalized} cyclic cocycles. Following the example of Lott \cite[Section 4.1]{JL92} we can construct explicit representations of any delocalized cyclic cocyle as follows: associated to each $\alpha\in H^*(\EuScript{C},\mathbb{C})$ define
\begin{equation}
\varphi_{\alpha,\gamma}(g_0,g_1,\ldots,g_n)=\left\{\begin{array}{c c}
0 & \mathrm{if}\; g_0g_1\cdots g_n\notin\mathrm{cl}(\gamma)\\
\alpha(h,hg_0,\ldots,hg_0g_1\cdots g_{n-1}) & \mathrm{if}\; g_0g_1\cdots g_n=h^{-1}\gamma h
\end{array}\right.
\end{equation}
By multilinearity of $\alpha$, it is immediate that given $a_i=\sum_{g_i\in\Gamma}c_{g_i}\cdot g_i$ in the group algebra
\[\varphi_{\alpha,\gamma}(a_0\otimes\cdots\otimes a_n)=\sum_{g_0g_1\cdots g_n\in\mathrm{cl}(\gamma)}c_{g_0}\cdots c_{g_n}\varphi_{\alpha,\gamma}(g_0,g_1,\ldots,g_n)\]
It is also apparent that the property $\varphi_{\alpha,\gamma}(\gamma g_0,g_1,\ldots,g_k)=\alpha(g_0,g_1,\ldots,g_k)$ generalizes to any element of $\gamma^{\mathbb{Z}}$, that is for any $r\in\mathbb{Z}$-- it suffices to consider $r\geq0$-- we have
\[\varphi_{\alpha,\gamma}(\gamma^r g_0,g_1,\ldots,g_k)=\varphi_{\alpha,\gamma}(\gamma^{r-1}g_0,g_1,\ldots,g_k)=\cdots=\varphi_{\alpha,\gamma}(g_0,g_1,\ldots,g_k)\]
If $\mathcal{A}$ is a unital algebra such that $\varphi_{\gamma}$ admits an extension to $\mathcal{A}$ then we define a unitized version of the cyclic cocyle. Let $\mathcal{A}^+$ be the algebra formed from adjoining a unit to $\mathcal{A}$, then the homomorphism $(A,\lambda)\longmapsto(A+\lambda1_{\mathcal{A}},\lambda)$ provides an isomorphism between $\mathcal{A}^+$ and $\mathcal{A}\oplus\mathbb{C}1$. For any $\varphi\in ZC^n(\mathbb{C}\Gamma,\mathrm{cl}(\gamma))$ we define
\begin{equation}
\overline{\varphi}_{\gamma}(\overline{A}_0\hat{\otimes}\cdots\hat{\otimes}\overline{A}_n)=\varphi_{\gamma}(A\hat{\otimes}\cdots\hat{\otimes} A_n)\quad\mathrm{where}\quad\overline{A}_i=(A_i,\lambda_i)\in\mathcal{A}^+\end{equation}
and as shown in \cite[Chapter 3.3]{AC94} the condition $b\varphi_{\gamma}=0$ still holds.
\begin{myrem}\label{remAb1}
With respect to the delocalized cyclic cocycle representations $\varphi_{\alpha,\gamma}$ there is an elementary way to move between $\alpha(g_0,g_1,\ldots,g_n)$ and the \textit{normalized} form \[\alpha(h,hg_0,\ldots,hg_0g_1\cdots g_{n-1})\]
which clearly vanishes if $g_i=e$ for any $0\leq i\leq n-1$. For each $y\in\mathrm{cl}(\gamma)$ fix some $h^y\in\Gamma$ such that $(h^y)^{-1}\gamma h^y=y$. In particular, the elements $y_0=g_0g_1\cdots g_n$ and $y_i=g_i\cdots g_ng_0\cdots g_{i-1}$ all belong to $\mathrm{cl}(\gamma)$ for $1\leq i\leq n$, since by hypothesis $y_0\in\mathrm{cl}(\gamma)$, and direct computation shows that $y_i=(g_0\cdots g_{i-1})^{-1}y_0(g_0\cdots g_{i-1})$.

The map $F$ defined by $F(g_i)=h^{y_0}(g_i\cdots g_n)^{-1}y_i$ induces a map on $H^*(\EuScript{C},\mathbb{C})$ according to $F^*[\alpha]=[\alpha\circ F]$ where since $g_0g_1\cdots g_n=y_0=(h^{y_0})^{-1}\gamma h^{y_0}$
\[(\alpha\circ F)(g_0,g_1,\ldots,g_n)=\alpha(F(g_0),F(g_1),\ldots,F(g_n))\]
\[=\alpha(h^{y_0}(g_0\cdots g_n)^{-1}y_0,h^{y_0}(g_1\cdots g_n)^{-1}y_1,\ldots,h^{y_0}g_n^{-1}y_n)\]
\[=\alpha(h^{y_0}(g_0\cdots g_n)^{-1}(g_0\cdots g_n),h^{y_0}(g_1\cdots g_n)^{-1}g_1\cdots g_ng_0,\ldots,h^{y_0}g_n^{-1}g_ng_0\cdots g_{n-1})\]
\[=\alpha(h^{y_0},h^{y_0}g_0,\ldots,h^{y_0}g_0\cdots g_{n-1})\]
This property of $F$ carries over to $\varphi_{\alpha,\gamma}$ acting on the group algebra $\mathbb{C}\Gamma$ by extending $F$ linearly.	
\end{myrem}

\section{Cyclic Cohomology of Polynomial Growth Groups}
The convergence properties of the integrals defining the pairing of delocalized cyclic cocyles with higher invariants depend crucially on the growth conditions of the cyclic cocyles. This in turn is linked to the growth properties of conjugacy classes of $\Gamma$, in particular it is proven in \cite{RJ92} that polynomial growth groups are of polynomial cohomology-- with respect to coefficients in $\mathbb{C}$.
\begin{mydef}\label{defB1}
The group $G$ is of polynomial cohomology if for any $[\phi]\in H^*(G,\mathbb{C})$ there exists (a skew cocyle) $\varphi\in Z(\mathsf{Hom}_{\Gamma}(\mathbb{Z}G^{*},\mathbb{C}),\hat{b})$ such that $[\varphi]=[\phi]$, and $\varphi$ is of polynomial growth. That is $\varphi$ satisfies the following bound for positive integer constants $R_\varphi$ and $k$
\begin{equation}\label{B1:CocyPolyGrowth}
|\varphi(g_0,g_1,\ldots,g_n)|\leq R_\varphi(1+||g_0||)^{2k}(1+||g_1||)^{2k}\cdots(1+||g_n||)^{2k}
\end{equation}
\end{mydef}
By \Cref{remAb1} it follows that any normalized group cocyle $\alpha$ also has polynomial growth if the non-normalized version does, since
\[|\alpha(h,hg_0,\ldots,hg_0g_1\cdots g_{n-1})|=|\alpha(F(g_0),F(g_1),\ldots,F(g_n))|\]
The splitting of delocalized cyclic cohomology as shown in \Cref{thmAb1} provides an abstract isomorphism between group cohomology and cyclic cohomology, but we desire an explicit construction of this mapping, so as to prove that polynomial growth group cocyles are mapped to polynomial growth cyclic cocyles. This shall be proven in \Cref{B2} through the use of the classifying space approach to group cohomology, while in the section imediately following we show that our attention can be restricted to the case where $\gamma$ is a torsion element.

\subsection{Cohomological Dimension and Connes Periodicity Map}\label{B1}
In the next section our results depend crucially on $H^*(N_\gamma,\mathbb{C})$ not contributing to the delocalized cyclic cohomology of $\mathbb{C}\Gamma$ whenever $\gamma$ is of infinite order; in this section we make this notion precise.  
For any unital associative algebra $A$ over a field containing $\mathbb{Q}$-- hence particularly for the group algebra $\mathbb{C}\Gamma$-- the cyclic and Hochschild homology fit into a long exact sequence
\begin{equation}
\begin{tikzcd}[column sep=2em]
\cdots \arrow{r} & HH^n(A) \arrow{r} & HC^{n-1}(A) \arrow[r,"S^*"] & HC^{n+1}(A) \arrow{r} & HH^{n+1}(A) \arrow{r} & \cdots
\end{tikzcd}
\end{equation}
where $S$ is the Connes periodicity operator introduced in \cite[II.1]{AC85}. We will use the explicit construction in terms of maps of complexes that is provided in \cite[Chapter 2]{JL94}, and so provide the following expression for $S$ when $A=\mathbb{C}\Gamma$. Let $b^*$ be the homomorphism induced by the boundary map $b$, and define a map $\beta:HC^{*}(\mathbb{C}\Gamma)\longrightarrow HC^{*+1}(\mathbb{C}\Gamma)$ according to
\begin{equation}
\begin{gathered}
(\beta\varphi)(g_0,g_1,\ldots,g_{k+1})=\sum_{i=0}^{k+1}(-1)^ii(\delta^i\varphi)(g_0,g_1,\ldots,g_{k+1})
\end{gathered}
\end{equation}
Hence $(\beta b)^*:HC^{*}(\mathbb{C}\Gamma)\longrightarrow HC^{*+2}(\mathbb{C}\Gamma)$ and similarly for the map induced by $b\beta$. Dual to the result given in \cite[Theorem 2.2.7]{JL94} we have that for any cohomology class $[\varphi]\in HC^{n}(\mathbb{C}\Gamma)$ its image under the periodicity operator\footnote{Note that our choice of constant differs from that of Connes \cite{AC85} due to the constants involved in the definition (see \Cref{Cd:ChernConnes}) of the Chern-Connes character} is
\begin{equation}\label{Ba:ConnesPer}
S^*[\varphi]=[S\varphi]\in HC^{n+2}(\mathbb{C}\Gamma)\quad\mathrm{where}\quad S=\frac{1}{(n+1)(n+2)}\left(\beta b+b\beta\right)
\end{equation}
\begin{equation}\label{Ba:bBeta}
b\beta=\sum_{0\leq i<j\leq n+2}(-1)^{i+j}(j-i)\delta^i\delta^j\qquad \beta b=\sum_{0\leq i<j\leq n+2}(-1)^{i+j}(i-j+1)\delta^i\delta^j
\end{equation}
\begin{mydef}\label{defBa1}
The delocalized Connes periodicity operator $S_\gamma$ is obtained by the restriction of $S$ to the sub-complex $(C^n(\mathbb{C}\Gamma,\mathrm{cl}(\gamma)),b)$.
\[S^*_\gamma:HC^{n}(\mathbb{C}\Gamma,\mathrm{cl}(\gamma))\longrightarrow HC^{n+2}(\mathbb{C}\Gamma,\mathrm{cl}(\gamma))\]
\end{mydef}

In the construction of the group cohomology $H^*(G,\mathbb{C})=\mathsf{Ext}^*_{\mathbb{Z}G}(\mathbb{Z},\mathbb{C})$ the minimal length of the projective $\mathbb{Z}G$-resolution over $\mathbb{Z}$ is called the cohomological dimension of the group. Denoting this by $\mathrm{cd}_{\mathbb{Z}}(G)$, it is immediate from the definition that if $\mathrm{cd}_{\mathbb{Z}}(G)=n$ then $H^k(G,\mathbb{C})=0$ for all $k>n$. If we consider projective $\mathbb{Q}G$-resolutions instead, then there is notion of rational cohomology and rational cohomological dimension. Recall that $\mathbb{Q}$ is a flat $\mathbb{Z}$-module, hence tensoring with $\mathbb{Q}$ preserves exactness, and $\mathrm{cd}_{\mathbb{Q}}(G)$ denotes the minimal length of the projective resolution defining the groups
\[H^*(G,\mathbb{C})\otimes_{\mathbb{Z}}\mathbb{Q}=\mathsf{Ext}^*_{\mathbb{Q}G}(\mathbb{Q},\mathbb{C})\]
When $G$ is a group of polynomial growth then it belongs to the class $\mathcal{C}$ of groups satisfying the following conditions (see \cite[Section 4]{RJ95})
\begin{enumerate}[label={(\roman{enumi})}] 
\item $G$ is of finite rational cohomological dimension. 
\item The rational cohomological dimension of $N_g=Z_G(g)/g^{\mathbb{Z}}$ is finite whenever $g$ is not a torsion element.
\end{enumerate}
It is now important to note that if $\gamma$ is of infinite order then $\mathrm{cl}(\gamma)$ is torsion free, for otherwise there exists $g^{-1}\gamma g\in\mathrm{cl}(\gamma)$ of finite order, and thus $e=(g^{-1}\gamma g)^k=g^{-1}\gamma^kg$, which implies $\gamma^k=e$. In this torsion free case the nilpotency of $S_\gamma$ with respect to the long exact sequence  
\begin{equation}
\begin{tikzcd}
\cdots \arrow[r] & HH^{n-1}(\mathbb{C}\Gamma,\mathrm{cl}(\gamma)) \ar[r] & HC^n(\mathbb{C}\Gamma,\mathrm{cl}(\gamma)) \arrow[r,"S^*"] & HC^{n+2}(\mathbb{C}\Gamma,\mathrm{cl}(\gamma))\\ \arrow[r] & HH^{n+2}(\mathbb{C}\Gamma,\mathrm{cl}(\gamma)) \arrow[r] & \cdots\qquad\qquad
\end{tikzcd}\end{equation}
follows from the proof of \cite[Theorem 4.2]{RJ95}, and so we have that $HC^{n}(\mathbb{C}\Gamma,\mathrm{cl}(\gamma))=0$ for $n>0$ and $\gamma$ of infinite order.

\begin{mylem}\label{lemBa1}
If $\Gamma$ is of polynomial growth, then $N_\gamma$ is of polynomial cohomology.
\end{mylem}
\begin{proof}
Since $\Gamma$ is of polynomial growth, then a theorem of Gromov \cite{MG81} states that $\Gamma$ is virtually nilpotent. We take $N$ to be a normal nilpotent subgroup of finite index, and so $Z_\gamma\cap N\leq N$ is finitely generated and nilpotent. The exact sequence 
\[Z_\gamma\cap N\longrightarrow Z_\gamma\longrightarrow\Gamma/N\]
shows that $Z_\gamma\cap N$ is of finite index in $Z_\gamma$, hence $Z_\gamma$ is finitely generated and admits a word length function $l_{Z_\gamma}$ which is bounded by $l_\Gamma$. It follows that $Z_\gamma$ is of polynomial growth with respect to any word length function on it. Taking the quotient by a central cyclic group preserves polynomial growth, and thus by \cite[Corollary 4.2]{RJ92} $N_\gamma$ is of polynomial cohomology.
\end{proof}

It should be made explicit that in the above proof we also obtained that $Z_\gamma$ was of polynomial cohomology-- or equivalently that $H^*(Z_\gamma,\mathbb{C})$ is polynomial bounded. Extending the notion of polynomial cohomology to that of delocalized cyclic cocyles, we call $HC^*(\mathbb{C}\Gamma,\mathrm{cl}(\gamma))$ polynomially bounded if every cohomology class admits a representative $\varphi_\gamma\in(C^n(\mathbb{C}\Gamma,\mathrm{cl}(\gamma)),b)$ which is of polynomial growth. 

\begin{mylem}\label{lemBa2}
There is a rational isomorphism between the cohomology group $H^n(\EuScript{C},\mathbb{C})$ (see \Cref{defAb4}) and $H^n(Z_\gamma,\mathbb{C})$ for $n\geq1$.
\end{mylem}
\begin{proof}
We begin with an alteration of the complex $((C^n(\Gamma,Z_\gamma,\gamma),\hat{b})$ constructed in \Cref{defAb4}, by removing the third condition: $\alpha(\gamma g_0,g_1,\ldots,g_n)=\alpha(g_0,g_1,\ldots,g_n)$.
This produces a larger cochain complex which shall be denoted by $((D^n(\Gamma,Z_\gamma,\gamma),\hat{b})$, and there is a natural inclusion map $\imath:((C^n(\Gamma,Z_\gamma,\gamma),\hat{b})\hookrightarrow((D^n(\Gamma,Z_\gamma,\gamma),\hat{b})$. In the other direction we consider an "averaging" map $\mathcal{R}:((D^n(\Gamma,Z_\gamma,\gamma),\hat{b})\longrightarrow((C^n(\Gamma,Z_\gamma,\gamma),\hat{b})$ similar to that from \cite[Theorem 5.2]{CWXY19}, defined according to
\begin{equation}
(\mathcal{R}\alpha)(g_0,g_1,\ldots,g_n)=\sum_{r_0,r_1,\ldots,r_n=1}^{\mathrm{ord}(\gamma)}\alpha(\gamma^{r_0}g_0,\gamma^{r_1}g_1,\ldots,\gamma^{r_n}g_n)
\end{equation}
where $\mathrm{ord}(\gamma)$ is the order of $\gamma$. By the above results of this section we only need to concern ourselves with $\gamma$ being a torsion element, and thus the above sum is finite, so $\mathcal{R}$ is well defined. To show that $\mathcal{R}$ is a surjective map it is first necessary to prove that $(\mathcal{R}\alpha)$ actually belongs to the complex $((C^n(\Gamma,Z_\gamma,\gamma),\hat{b})$; namely if $\gamma$ is torsion, then $\gamma^{\mathbb{Z}}=\{e,\gamma,\ldots,\gamma^{\mathrm{ord}(\gamma)-1}\}=\gamma\cdot\gamma^{\mathbb{Z}}$ and
\[(\mathcal{R}\alpha)(\gamma g_0,g_1,\ldots,g_n)=\sum_{r_0,r_1,\ldots,r_n=1}^{\mathrm{ord}(\gamma)}\alpha(\gamma^{r_0+1}g_0,\gamma^{r_1}g_1,\ldots,\gamma^{r_n}g_n)\]
\[\sum_{r_0+1,r_1,\ldots,r_n=1}^{\mathrm{ord}(\gamma)}\alpha(\gamma^{r_0+1}g_0,\gamma^{r_1}g_1,\ldots,\gamma^{r_n}g_n)=(\mathcal{R}\alpha)(g_0,g_1,\ldots,g_n)\]
It is similarly straightforward to prove that the following diagram commutes
\[\begin{tikzcd} D^{n+1}(\Gamma,Z_\gamma,\gamma) \arrow[d,"\mathcal{R}"]  &  D^n(\Gamma,Z_\gamma,\gamma) \ar[l,swap,"\hat{b}"] \arrow[d,"\mathcal{R}"]\\
 C^{n+1}(\Gamma,Z_\gamma,\gamma)  &  C^n(\Gamma,Z_\gamma,\gamma) \ar[l,swap,"\hat{b}"]\end{tikzcd}\]
and so $\mathcal{R}$ is indeed a chain complex map. We now prove that the composition $\mathcal{R}\circ\imath:((C^n(\Gamma,Z_\gamma,\gamma),\hat{b})\longrightarrow((C^n(\Gamma,Z_\gamma,\gamma),\hat{b})$ is rationally equivalent to the identity map $Id_C$, and thus by extension that $\mathcal{R}$ is a rational surjection. Taking any $\alpha\in C^n(\Gamma,Z_\gamma,\gamma)$ and using the property $\alpha(\gamma^rg_0,g_1,\ldots,g_n)=\alpha(g_0,g_1,\ldots,g_n)$
\[(\mathcal{R}\circ\imath\alpha)(g_0,g_1,\ldots,g_n)=\sum_{r_0,r_1,\ldots,r_n=1}^{\mathrm{ord}(\gamma)}\alpha(\gamma^{r_0}g_0,\gamma^{r_1}g_1,\ldots,\gamma^{r_n}g_n)\]
\[=\sum_{r_0=1}^{\mathrm{ord}(\gamma)}\sum_{r_1,\ldots,r_n=1}^{\mathrm{ord}(\gamma)}\alpha(\gamma^{r_0}g_0,\gamma^{r_1}g_1,\ldots,\gamma^{r_n}g_n)\]
\[=\frac{\mathrm{ord}(\gamma)(\mathrm{ord}(\gamma)+1)}{2}\sum_{r_1,\ldots,r_n=1}^{\mathrm{ord}(\gamma)}\alpha(g_0,\gamma^{r_1}g_1,\ldots,\gamma^{r_n}g_n)\]
For convenience denote the coefficient $\mathrm{ord}(\gamma)(\mathrm{ord}(\gamma)+1)/2$ by $A$, then by repeated shifting of the $g_0$ element the above sum becomes
\[=(-1)A\sum_{r_1=1}^{\mathrm{ord}(\gamma)}\sum_{r_2,\ldots,r_n=1}^{\mathrm{ord}(\gamma)}\alpha(g_1,g_0,\gamma^{r_2}g_2,\ldots,\gamma^{r_n}g_n)\]
\[=(-1)A^2\sum_{r_2,\ldots,r_n=1}^{\mathrm{ord}(\gamma)}\alpha(g_1,g_0,\gamma^{r_2}g_2,\ldots,\gamma^{r_n}g_n)=\cdots=(-1)^nA^{n+1}\alpha(g_1,\ldots,g_n,g_0)\]
\[=(-1)^n(-1)^nA^{n+1}\alpha(g_0,g_1,\ldots,g_n)=A^{n+1}\alpha(g_0,g_1,\ldots,g_n)\]
It thus follows that as maps from $(C^n(\Gamma,Z_\gamma,\gamma),\hat{b})\otimes_{\mathbb{{Z}}}\mathbb{Q}$ to itself, we have the equality
\begin{equation}
(\mathcal{R}\circ\imath)\otimes_{\mathbb{Z}}\frac{1}{A^{n+1}}=Id_C\otimes_{\mathbb{Z}}1
\end{equation}
which establishes the isomorphism on cohomology $H^*(\EuScript{C},\mathbb{C})\otimes_{\mathbb{{Z}}}\mathbb{Q}\stackrel{\mathcal{R}}{\cong}H^*(\EuScript{D},\mathbb{C})\otimes_{\mathbb{{Z}}}\mathbb{Q}$. The desired result now follows from Nistor's \cite[Section 2.7]{VN90} application of spectral sequences to prove that $H^*(\EuScript{D},\mathbb{C})$ is isomorphic to the group cohomology $H^*(Z_\gamma,\mathbb{C})$. 
\end{proof}

\subsection{Classifying Space Construction}\label{B2}
The isomorphism $H^*(\EuScript{D},\mathbb{C})\cong H^*(Z_\gamma,\mathbb{C})$ mentioned in \Cref{lemBa2} unfortunately does not provide an explicit way to realize preservation of polynomial cohomology. It is, however, easy to show that as defined in \Cref{lemBa2} the map $\mathcal{R}$ behaves as desired in this respect.
\begin{mypro}\label{proBb1}
The map $(\mathcal{R}\otimes_{\mathbb{Z}}1/A^{n+1})^*$ preserves polynomial growth. 
\end{mypro} 
\begin{proof}
It suffices to show that if $\alpha$ is of polynomial growth then so is $\mathcal{R}\alpha$, which follows directly from $\gamma$ having finite order.
\[|(\mathcal{R}\alpha)(g_0,g_1,\ldots,g_n)|=\left|\sum_{r_0,r_1,\ldots,r_n=1}^{\mathrm{ord}(\gamma)}\alpha(\gamma^{r_0}g_0,\gamma^{r_1}g_1,\ldots,\gamma^{r_n}g_n)\right|\]
\[\leq\sum_{r_0,r_1,\ldots,r_n=1}^{\mathrm{ord}(\gamma)}|\alpha(\gamma^{r_0}g_0,\gamma^{r_1}g_1,\ldots,\gamma^{r_n}g_n)|\leq\sum_{r_0,r_1,\ldots,r_n=1}^{\mathrm{ord}(\gamma)} R_\alpha(1+||\gamma^{r_0}g_0||)^{2k}\cdots(1+||\gamma^{r_n}g_n||)^{2k}\]
\[\leq(\mathrm{ord}(\gamma))^{n+1}\max_{r_0,r_1,\ldots,r_n\in\{1,2\ldots,\mathrm{ord}(\gamma)\}}R_\alpha(1+||\gamma^{r_0}g_0||)^{2k}\cdots(1+||\gamma^{r_n}g_n||)^{2k}\]
\[=(\mathrm{ord}(\gamma))^{n+1}R_\alpha(1+||\gamma^{m_0}g_0||)^{2k}\cdots(1+||\gamma^{m_n}g_n||)^{2k}\]
\end{proof}

\begin{mythm}\label{thmBb1}
For every nontrivial conjugacy class $\mathrm{cl}(\gamma)$ for $\gamma$ of finite order, if $H^*(Z_\gamma,\mathbb{C})$ is of polynomial cohomology then $H^*(\EuScript{D},\mathbb{C})$ is polynomially bounded.
\end{mythm}
\begin{proof}
Consider the simplicial set $E_*\Gamma$, the nerve of $\Gamma$, with simplices $E_k\Gamma:=\Gamma\times\Gamma^{k}$ and relations
\begin{equation}\label{Bb:NerveRel}
\begin{gathered}
d_i(g_0,\ldots,g_k)=\left\{\begin{array}{cc}
(g_0,\ldots,g_{i-1},g_ig_{i+1},g_{i+2},\ldots,g_k) & 0\leq i\leq k-1\\
(g_0,g_1\ldots,g_{k-1}) & i=k
\end{array}\right.\\
s_j(g_0,\ldots,g_k)=(g_0,\ldots,g_j,e,g_{j+1},\ldots,g_k)\\
(g_0,\ldots,g_{k-1},g_k)g=(g_0g,\ldots,g_{k-1},g_k)\;\forall g\in\Gamma
\end{gathered}
\end{equation}
The last equation defines a free $\Gamma$-action on the nerve according to right multiplication. The contractible space $E\Gamma=|E_*\Gamma|$-- which is a locally finite CW-complex-- is the geometrical realization of the nerve under the relations
\begin{equation}
|E_*\Gamma|=\bigsqcup_{k\geq0}E_k\Gamma\times\triangle^k/\sim\qquad\begin{array}{cc}
(x,\delta_it)\sim(d_ix,t) & \mathrm{for}\;x\in E_k\Gamma,t\in\triangle^{k-1}\\
(x,\sigma_jt)\sim(s_jx,t) & \mathrm{for}\;x\in E_k\Gamma,t\in\triangle^{k+1}
\end{array}
\end{equation}
where $\triangle^k=\{(t_0,\ldots,t_k)\in\mathbb{R}^{k+1}:t_i\in[0,1],\sum_{i=0}^{k}t_i=1\}$ is the standard $k$-simplex with degeneracy maps $\sigma_j$ and face maps $\delta_i$. Since right multiplication is a free action, it is immediate that $\Gamma$ acts freely on $E\Gamma$, thus we define the classifying space $B\Gamma=|B_*\Gamma|$ to be the orbit space $E\Gamma/\Gamma$; the simplices of $B_*\Gamma$ are thus of the form $E_k\Gamma/\Gamma$. More generally, it is useful to view $E\Gamma$ as a the universal principal bundle $p:E\Gamma\longrightarrow B\Gamma$, which provides the relation of balanced products
\[B\Gamma=E\Gamma\times_\Gamma\Gamma=\{(v,w)\in E\Gamma\times\Gamma/((v,gw)\sim(vg,w))\}\]
Applying the above simplicial construction to $Z_\gamma$ we can similarly construct the universal principal $Z_\gamma$-bundle $p:EZ_\gamma\longrightarrow BZ_\gamma$. The geometric realization is functorial, hence any group homomorphism induces a homomorphism on $E\Gamma$ at the simplex level. Since $\Gamma$ is a discrete group then $Z_\gamma$ is \textit{admissible} as a subgroup, and there is the principal $Z_\gamma$-bundle $q:\Gamma\longrightarrow\Gamma/Z_\gamma$ where $q$ is the quotient map. In particular, since $p:E\Gamma\longrightarrow B\Gamma$ is a universal principal $\Gamma$-bundle
\[p':E\Gamma\times_{Z_\gamma}\Gamma\longrightarrow E\Gamma\times_\Gamma(\Gamma/Z_\gamma)=(Z_\gamma\curvearrowright E\Gamma)/Z_\gamma\]
is a universal principal $Z$-bundle, where $Z_\gamma\curvearrowright E\Gamma$ has the same nerve $E_*\Gamma$, except with the third relation in \eqref{Bb:NerveRel} replaced by a $Z_\gamma$ group action
\[(g_0,\ldots,g_{k-1},g_k)z=(g_0z,\ldots,g_{k-1},g_k)\;\forall z\in Z_\gamma\] 
By universality of the principal bundle there is a homotopy equivalence between $BZ_\gamma$ and $(Z_\gamma\curvearrowright E\Gamma)/Z_\gamma$ as CW-complexes, which implies equivalence of (co)homology groups. (See \cite{JM99} for a more detailed discussion on classifying spaces and fibre bundles)

Every point $v$ in an oriented $m$-simplex $V_m$ of $E\Gamma$ can be expressed as a formal sum, using barycentric coordinates and the vertices $\mathbf{g}_i=(e,\ldots,g_i,\ldots,e)$. 
\[\sum_{i=0}^{m}t_i=1\quad\mathrm{and}\quad v=\sum_{i=0}^{m}t_i\mathbf{g}_i\qquad\qquad V_m=g_0[g_1|\cdots|g_m]\]
The particular notation for the $m$-simplex is important in emphasizing the group action on the first coordinate. As we shall see shortly, this is also useful when describing boundary maps of simplicial chain complexes. First consider the CW-complex $B\Gamma$ as the image of $E\Gamma$ under the projection map $p$ acting on the nerve by deletion of the first coordinate. Hence, an oriented $m$-simplex $V_m$ of $B\Gamma$ is of the form $p(V_m)=[g_1|\cdots|g_m]$; hence we construct the following chain complex, with $C_m(B\Gamma)$ the free module generated by the basis of $m$-simplexes. 
\[\begin{tikzcd}
\cdots \arrow[r] & C_m(B\Gamma) \arrow[r,"\partial"] & C_{m-1}(B\Gamma) \arrow[r,"\partial"] & \cdots \arrow[r,"\partial"] & C_0(B\Gamma) \arrow[r,"\partial"] & 0
\end{tikzcd}\]
\[c_\alpha\in\mathbb{C},\quad\sum_\alpha c_\alpha p(V_m)_\alpha\in C_m(B\Gamma)\qquad\partial=\sum_{i=0}^{m}(-1)^id_i\]
Thus, explicitly expressing the boundary map action shows that $\partial(p(V_m)_\alpha)$ is equal to
\[p(g_0g_1[g_2|\cdots|g_m])+\sum_{i=1}^{m-1}(-1)^ip(g_0[g_1|\cdots|g_ig_{i+1}|\cdots|g_m])+(-1)^{m}p(g_mg_0[g_1|\cdots|g_{m-1}])\]
\[=[g_2|\cdots|g_m]+\sum_{i=1}^{m-1}(-1)^i[g_1|\cdots|g_ig_{i+1}|\cdots|g_m]+(-1)^{m}[g_1|\cdots|g_{m-1}]\]
Dualizing, the associated simplicial cochain complex $C^*_i(B\Gamma):=\mathsf{Hom}(C_i(B\Gamma),\mathbb{C})$ is obtained, with coboundary map $b=\sum_{i=0}^{n+1}\delta^i$; recall that $\langle\delta^if,v\rangle=\langle f,d_i(v)\rangle$.
\[\begin{tikzcd}0\arrow[r] & C_0^*(B\Gamma) \arrow[r,"b"] & C_1^*(B\Gamma) \arrow[r,"b"]  & \cdots \arrow[r,"b"] & C_m^*(B\Gamma) \arrow[r,"b"] & \cdots\end{tikzcd}\]
The simplicial cohomology groups $H^*(B\Gamma,\mathbb{C})$ of this complex give precisely the same characteristic classes as the group cohomology $H^*(\Gamma,\mathbb{C})$. More usefully, the expression of bundles in the language of balanced products allows us to similarly obtain equivalences
\[H^*(EZ_\gamma\times_{Z_\gamma}Z_\gamma,\mathbb{C})=H^*(BZ_\gamma,\mathbb{C})=H^*(Z_\gamma,\mathbb{C})\]
\[H^*(E\Gamma\times_\Gamma(\Gamma/Z_\gamma),\mathbb{C})=H^*((Z_\gamma\curvearrowright E\Gamma)/Z_\gamma,\mathbb{C})={H^{Z_\gamma}}^*(\Gamma,\mathbb{C})=H^*(\EuScript{D},\mathbb{C})\]
It follows that exhibiting an explicit $Z_\gamma$-equivariant map $\psi:E\Gamma\times_\Gamma(\Gamma/Z_\gamma)\longrightarrow EZ_\gamma\times_{Z_\gamma}Z_\gamma$ which induces a polynomial growth preserving isomorphism on cohomology
\begin{equation}
\psi^*:H^*(EZ_\gamma\times_{Z_\gamma}Z_\gamma,\mathbb{C})\longrightarrow H^*(E\Gamma\times_\Gamma(\Gamma/Z_\gamma),\mathbb{C})
\end{equation}
also provides an isomorphism $\psi^*:H^*(Z_\gamma,\mathbb{C})\longrightarrow H^*(\EuScript{D},\mathbb{C})$ preserving polynomial cohomology. It is of course necessary to be precise by what is meant by polynomial cohomology in the context of the classifying space construction. If for every class $[\varphi]\in H^m(B\Gamma,\mathbb{C})$ there exists a representative $\widetilde{\varphi}\in(C^*_m(B\Gamma),b)$ which is of polynomial growth, then $H^*(B\Gamma,\mathbb{C})$ is polynomial cohomology. Viewing $\widetilde{\varphi}$ as a function on basis elements, it is of polynomial growth given
\begin{equation}\label{Bb:PolGrowth}
\begin{gathered}
\widetilde{\varphi}=\sum_\alpha c_\alpha p(V_m)_\alpha=\sum_\alpha c_\alpha p(g_{0_\alpha}[g_{1_\alpha}|\cdots|g_{m_\alpha}]):=\widetilde{\varphi}(g_{0_\alpha},g_{1_\alpha},\ldots,g_{m_\alpha})=c_\alpha\\
|\widetilde{\varphi}(g_{0_\alpha},g_{1_\alpha},\ldots,g_{m_\alpha})|=|c_\alpha|\leq R_{\widetilde{\varphi}}(1+||g_{0_\alpha}||)^{2k}(1+||g_{1_\alpha}||)^{2k}\cdots(1+||g_{m_\alpha}||)^{2k}
\end{gathered}
\end{equation}
where $R_{\widetilde{\varphi}}$ and $k$ are positive integer constants. Recall that the word length function $l_Z$ is bounded above by $l_\Gamma$, hence we shall view $Z_\gamma$ as metrically embedded in $\Gamma$. Fix some generating set $S$ of $\Gamma$ and also fix an ordering for $S$, then every $g\in\Gamma$ can be lexicographically ordered; denote by $\mathsf{lex}(\Gamma)$ the lexicographic ordering of $\Gamma$ under that of $S$. We introduce a map $f:\Gamma\longrightarrow Z_\gamma$ defined as the following minimizing $z\in\mathsf{lex}(\Gamma)\cap Z_\gamma$.
\begin{equation}\label{Bb:Lexicographic}
f(g):=\min_{z\in\mathsf{lex}(\Gamma)\cap Z_\gamma}\left\{\min_{z\in Z_\gamma}\{d_\Gamma(z,g)\leq||g||\}\right\}
\end{equation}
Due to the lexicographic ordering this provides a unique element $f(g)\in Z_\gamma$. Such an element always exists, since there is at least the option $z=e$; in particular, it is a direct consequence that $f(z)=z$ for $z\in Z_\gamma$. Moreover, using the fact that left multiplication of any group on itself is a free and transitive action it follows that $f$ is $Z_\gamma$-equivariant, since if we fix $z_0\in Z_\gamma$
\[f(z_0g):=\min_{z\in\mathsf{lex}(\Gamma)\cap Z_\gamma}\left\{\min_{z\in Z_\gamma}\{d_\Gamma(z,z_0g)\leq||z_0g||\}\right\}\]
\[=\min_{z_0z\in\mathsf{lex}(\Gamma)\cap Z_\gamma}\left\{\min_{z_0z\in Z_\gamma}\{d_\Gamma(z_0z,z_0g)\leq||z_0g||\}\right\}\]
\[=\min_{z_0z\in\mathsf{lex}(\Gamma)\cap Z_\gamma}\left\{\min_{z_0z\in Z_\gamma}\{d_\Gamma(z,g)\leq||z_0g||\}\right\}=:z_0f(g)\]
It is similarly possible to express a $Z_\gamma$-invariant map $\tilde{f}:\Gamma/Z_\gamma\longrightarrow Z_\gamma$ in terms of the map $f$ constructed above. We recall that for $h,g\in\Gamma$
\[d_\Gamma(hZ_\gamma,g)=\min_{z\in Z_\gamma}\{d_\Gamma(hz,g)\}\qquad\mathrm{and}\qquad d_\Gamma(h_0Z_\gamma,h_1Z_\gamma)=\min_{z,z'\in Z_\gamma}\{d_\Gamma(h_0z,h_1z')\}\]
and thus define $\tilde{f}(hZ_\gamma)$ to be the value of $f(hz)$ present in the minimization
\[\min_{z\in Z_\gamma}\{d_\Gamma(hz,f(hz))\}\]
By the bijection between cosets of $\Gamma/Z_\gamma$ and conjugacy classes of $\gamma$ arising from the map $hZ_\gamma\longmapsto h^{-1}\gamma h$ we ensure well-definedness of $\tilde{f}$. The map $\psi$ acts on $E\Gamma\times_\Gamma(\Gamma/Z_\gamma)$ according to
\begin{equation}\label{Bb:EquiMap}
\psi(x,hZ_\gamma)=\psi\left(\sum_it_i\mathbf{g}_i,hZ_\gamma\right)=\left(\sum_it_i(e,\ldots,f(g_i),\ldots,e),\tilde{f}(hZ_\gamma)\right)
\end{equation}
From this we can show that $\psi$ is $Z_\gamma$-equivariant on $E\Gamma$ and $Z_\gamma$-invariant in the second argument. Pick any $(z,z')\in Z_\gamma\times Z_\gamma$, then converting the right $\Gamma$-action on $E\Gamma$ to a left one
\[\psi((z,z')(x,hZ_\gamma))=\psi(z^{-1}x,z'hZ_\gamma)=\psi\left(\sum_it_i(z^{-1}\mathbf{g}_i),z'hZ_\gamma\right)\]
\[=\left(\sum_it_i(z^{-1},\ldots,f(z^{-1}g_i),\ldots,z^{-1}),\tilde{f}(z'hZ_\gamma)\right)\]
\[=\left(\sum_it_i(z^{-1}e,\ldots,z^{-1}f(g_i),\ldots,z^{-1}e),\tilde{f}(hZ_\gamma)\right)\]
\[=\left(\sum_it_i\cdot z^{-1}(e,\ldots,f(g_i),\ldots,e),\tilde{f}(hZ_\gamma)\right)=(z,e)\cdot\psi(x,hZ_\gamma)\]
By the definition provided by equation \eqref{Bb:PolGrowth} the map $\psi:E\Gamma\times_\Gamma(\Gamma/Z_\gamma)\longrightarrow EZ_\gamma\times_{Z_\gamma}Z_\gamma$ preserves polynomial growth if there exists a positive integer $r$ and constant $K>0$ such that
\begin{equation}\label{Bb:DistIneq}
(\rho,d_\Gamma)_p\left(\psi(x,h_0Z_\gamma),\psi(y,h_1Z_\gamma)\right)\leq K\cdot[(\rho,d_\Gamma)_p\left((x,h_0Z_\gamma),(y,h_1Z_\gamma)\right)]^r\end{equation}
This ensures that for any change $\lambda|\phi|$ in the value of a cyclic cocyle representative of the class $[\varphi]\in H^*(EZ_\gamma\times_{Z_\gamma}Z_\gamma,\mathbb{C})$ there exists a representative of $\psi^*[\varphi]\in H^*(E\Gamma\times_{\Gamma}(\Gamma/Z_\gamma),\mathbb{C})$ whose absolute value $|\psi^*\phi|$ changes no more than $K(\lambda|\phi|)^r$. As a preliminary necessity to proving this property of $\psi$, we recall that a path metric can be put on $E\Gamma$ such that for $x,y$ not in the same connected component $\rho(x,y)=1$, and otherwise if $x$ and $y$ are joined by a union of paths $\bigcup_{l=1}^k\alpha_l$, where each $\alpha_l$ belongs to a single simplex
\[\rho(x,y)=\inf_{\alpha_l}\sum_{l=1}^k\mathrm{length}(\alpha_l)\]
Hence there exists points $v$ and $v'$ in this simplex such that $\mathrm{length}(\alpha_l)=\rho(v,v')$; the metric $\rho$ is defined as $\min\{\rho_1,\rho_2\}$, where
\[\rho_1\left(\sum_it_i\mathbf{g}_i,\sum_it'_i\mathbf{g}_i\right)=\sum_i|t_i-t'_i|||g_i||\qquad\rho_2\left(\sum_it_i\mathbf{g}_i,\sum_it'_i\mathbf{g}_i\right)=\sum_{i,j}t_it'_jd_\Gamma(g_i,g_j)\]
The metric placed on $\Gamma/Z_\gamma$ will be the usual word metric $d_\Gamma$, and we assign to $E\Gamma\times_\Gamma(\Gamma/Z_\gamma)$ the $p$-product metric $(\rho,d_\Gamma)_p$, for $p\in[1,\infty)$. By the left $\Gamma$-invariance of $d_\Gamma$, and the fact that $(xg,w)\sim(x,gw)$ it follows that $(\rho,d_\Gamma)_p$ is also left $\Gamma$-invariant. Without loss of generality we may take $x$ and $y$ to belong to the same connected component, and since the collection of left cosets partition $\Gamma$, we assume that $h_0Z_\gamma$ is distinct from $h_1Z_\gamma$. The proof of the inequality \eqref{Bb:DistIneq} thus follows immediately from the definition of $\psi$: explicitly,  $(\rho_1,d_\Gamma)_p\left(\psi(x,h_0Z_\gamma),\psi(y,h_1Z_\gamma)\right)$ has the expression
\[\left|\left|\rho_1\left(\sum_{l=1}^k\sum_{i_l}t_{i_l}f(\mathbf{g}_{i_l}),\sum_{l=1}^k\sum_{i_l}t'_{i_l}f(\mathbf{g}_{i_l})\right),d_\Gamma(\tilde{f}(h_0Z_\gamma),\tilde{f}(h_1Z_\gamma))\right|\right|_p\]
\[=\left|\left|\sum_{l=1}^k\sum_{i_l,j_l}t_{i_l}t'_{j_l}d_\Gamma(f(g_{i_l}),f(g_{j_l})),d_\Gamma(\tilde{f}(h_0Z_\gamma),\tilde{f}(h_1Z_\gamma))\right|\right|_p\]
\[\leq\left|\left|\sum_{l=1}^k\sum_{i_l,j_l}4t_{i_l}t'_{j_l}d_\Gamma(g_{i_l},g_{j_l}),4d_\Gamma(h_0Z_\gamma,h_1Z_\gamma)\right|\right|_p\]
The inequality in the last line follows from the bound $d_\Gamma(f(g_i),f(g_j)\leq4d_\Gamma(g_i,g_j)$ proven below in \Cref{proBb2} and the analogous one for $\tilde{f}$. Similarly, with respect to the metric $\rho_2$ we have
\[\left|\left|\rho_2\left(\sum_{l=1}^k\sum_{i_l}t_{i_l}f(\mathbf{g}_{i_l}),\sum_{l=1}^k\sum_{i_l}t'_{i_l}f(\mathbf{g}_{i_l})\right),d_\Gamma(\tilde{f}(h_0Z_\gamma),\tilde{f}(h_1Z_\gamma))\right|\right|_p\]
\[=\left|\left|\sum_{l=1}^k\sum_{i_l}|t_{i_l}-t'_{i_l}|||f(g_{i_l})||,d_\Gamma(\tilde{f}(h_0Z_\gamma),\tilde{f}(h_1Z_\gamma))\right|\right|_p\]
\[\leq\left|\left|\sum_{l=1}^k\sum_{i_l}2|t_{i_l}-t'_{i_l}|||g_{i_l}||,4d_\Gamma(h_0Z_\gamma,h_1Z_\gamma)\right|\right|_p\]
The factor of 2 present in the last line stems from the fact that 
\[||f(g_i)||=d_\Gamma(f(g_i),e)\leq d_\Gamma(f(g_i),g_i) + d_\Gamma(g_i,e)\leq d_\Gamma(e,g_i) + d_\Gamma(g_i,e)=2||g_i||\]
In combination with the result for $\rho_1$ this proves the inequality \eqref{Bb:DistIneq} for $K=4$ and $r=1$. 
\end{proof}

\begin{mycor}\label{corBb1}
Every delocalized cyclic cocyle class $[\varphi_{\gamma}]\in HC^*(\mathbb{C}\Gamma,\mathrm{cl}(\gamma))$ has a representative $\varphi_{\alpha,\gamma}$ of polynomial growth, hence $H^*(\mathbb{C}\Gamma,\mathrm{cl}(\gamma))$ is polynomially bounded.
\end{mycor}
\begin{proof}
Since $\Gamma$ is of polynomial growth, then the proof of \Cref{lemBa1} asserts that $H^*(Z_\gamma,\mathbb{C})$ is of polynomial cohomology. Furthermore, if we denote by $\mathcal{R}^{-1}$ the inverse isomorphism to that constructed in \Cref{lemBa2}, then by \Cref{proBb1} and \Cref{thmBb1} there exists a chain of isomorphisms which preserve polynomial cohomology for all $n\geq1$.
\[H^n(Z_\gamma,\mathbb{C})\otimes_{\mathbb{{Z}}}\mathbb{Q}\xrightarrow{(\psi\otimes_{\mathbb{{Z}}}1)^*}H^n(\EuScript{D},\mathbb{C})\otimes_{\mathbb{{Z}}}\mathbb{Q}\xrightarrow{(\mathcal{R}^{-1}\otimes_{\mathbb{Z}}1/A^{n+1})^*}H^n(\EuScript{C},\mathbb{C})\otimes_{\mathbb{{Z}}}\mathbb{Q}\]
The desired result now follows from recalling that by \Cref{defAb4} there exists an explicit representation $\varphi_{\alpha,\gamma}\in(C^n(\Gamma,Z_\gamma,\gamma),\hat{b})$ for each $\varphi_\gamma\in (C^n(\mathbb{C}\Gamma,\mathrm{cl}(\gamma)),b)$.
\end{proof}

\begin{mypro}\label{proBb2}
Let $f:\Gamma\longrightarrow Z_\gamma$ and $\tilde{f}:\Gamma/Z_\gamma\longrightarrow Z_\gamma$ be as described in \Cref{thmBb1}, then both have a Lipschitz constant of 4.
\end{mypro}
\begin{proof}
Given distinct $g_i,g_j\in\Gamma$ with word representations $g_i=s_{i_1}s_{i_2}\cdots s_{i_k}$ and $g_j=s_{j_1}s_{j_2}\cdots s_{j_k}$
\[d_\Gamma(g_i,g_j)=||g_ig_j^{-1}||=l_\Gamma(s_{i_1}s_{i_2}\cdots s_{i_{k}-m}s_{j_{k}-m}^{-1}\cdots s_{j_2}^{-1}s_{j_1}^{-1})=j_k+i_k-2m\]
where $m$ is the cancellation length. By definition of $f$ provided above in \eqref{Bb:Lexicographic} there exist lexicographically minimal $z_i=f(g_i)$ and $z_j=f(g_j)$ such that $d_\Gamma(z_i,g_i)$ and $d_\Gamma(z_j,g_j)$ are minimized. By the properties of the metric, it is immediate that
\[d_\Gamma(z_i,z_j)\leq d_\Gamma(z_i,g_i)+d_\Gamma(g_i,g_j)+d_\Gamma(g_j,z_j)\leq i_k+(j_k+i_k-2m)+j_k\leq 2(i_k+j_k)\]
On the other hand, expressing $z_i$ and $z_j$ as the words $z_i=s'_{i_1}s'_{i_2}\cdots s'_{i_{k'}}$ and $z_j=s'_{j_1}s'_{j_2}\cdots s'_{j_{k'}}$
\[d_\Gamma(z_i,g_i)=||z_ig_i^{-1}||=l_\Gamma(s'_{i_1}s'_{i_2}\cdots s'_{i_{k'}-b}s_{i_{k}-b}^{-1}\cdots s_{i_2}^{-1}s_{i_1}^{-1})=i_{k'}+i_k-2b\]
\[d_\Gamma(g_j,z_j)=||g_jz_j^{-1}||=l_\Gamma(s_{j_1}s_{j_2}\cdots s_{j_k-a}{s'}_{j_{k'}-a}^{-1}\cdots {s'}_{j_2}^{-1}{s'}_{j_1}^{-1})=j_{k'}+j_k-2a\]
We can thus make use of the decomposition $z_iz_j^{-1}=z_ig_i^{-1}g_ig_j^{-1}g_jz_j^{-1}$ and obtain the reduced word expression for $z_iz_j^{-1}$ as
\[s'_{i_1}\cdots s'_{i_{k'}-b}s_{i_{k}-b}^{-1}\cdots s_{i_1}^{-1}s_{i_1}\cdots s_{i_{k}-m}s_{j_{k}-m}^{-1}\cdots s_{j_1}^{-1}s_{j_1}\cdots s_{j_k-a}{s'}_{j_{k'}-a}^{-1}\cdots {s'}_{j_1}^{-1}\]
\[=s'_{i_1}s'_{i_2}\cdots s'_{i_{k'}-b}s_{i_{k}-b}^{-1}\cdots s^{-1}_{i_{k}-m-1}s_{j_{k}-m-1}\cdots s_{j_k-a}{s'}_{j_{k'}-a}^{-1}\cdots {s'}_{j_2}^{-1}{s'}_{j_1}^{-1}\] 
where-- without loss of generality-- we have assumed that $m\geq a,b$. It follows that 
\[l_\Gamma(z_iz_j^{-1})=i_{k'}-b+(i_k-b-i_k+m+1)+(j_k-a-j_k+m+1)+j_{k'}-a\]
\[=i_{k'}+j_{k'}-2a-2b+2m+2\]
However, since the identity element is always a possible choice for $z_i$ we know that $i_{k'}+i_k-2b\leq i_k$, from which it follows that $i_{k'}-2b\leq0$, and analogously $j_{k'}+j_k-2a\leq j_k$ implies that $j_{k'}-2a\leq0$; this provides the bound $d_\Gamma(z_i,z_j)\leq 2m+2$. We thus have the two following cases: 
\begin{enumerate}[label={(\roman{enumi})}] 
\item If $m\geq i_k+j_k$ then $d_\Gamma(g_i,g_j)=j_k+i_k-2m\geq\frac{1}{2}(j_k+i_k)$, and the bound $d_\Gamma(z_i,z_j)\leq2(i_k+j_k)$ provides: $d_\Gamma(z_i,z_j)\leq4d_\Gamma(g_i,g_j)$
\item If $m\leq i_k+j_k$ then $d_\Gamma(g_i,g_j)=j_k+i_k-2m\geq2m$, and the bound $d_\Gamma(z_i,z_j)\leq2m+2$ provides: $d_\Gamma(z_i,z_j)\leq2d_\Gamma(g_i,g_j)$
\end{enumerate} 
An analogous result is obtained for $\tilde{f}$ using its definition with respect to $f$ and the inequalities already proven for $f$. The proof follows along the exact same lines, explicitly providing
\[\min_{z,z'\in Z_\gamma}\{d_\Gamma(f(h_0z),f(h_1z'))\}\leq4\min_{z,z'\in Z_\gamma}\{d_\Gamma(h_0z,h_1z')\}\]
\[d_\Gamma(\tilde{f}(h_0Z_\gamma),\tilde{f}(h_1Z_\gamma))\leq4d_\Gamma(h_0Z_\gamma,h_1Z_\gamma)\]
\end{proof}

When choosing a representative of polynomial growth for the purposes of the following sections, it is paramount that this choice is independent of representative in a way that respects polynomial growth. Explicitly, let $[\varphi_\gamma]$ be a delocalized cyclic cocyle class with polynomial growth representatives $\psi_{\gamma,1},\psi_{\gamma,2}\in(C^n(\mathbb{C}\Gamma,\mathrm{cl}(\gamma)),b)$. Since by definition the group $H^n(\mathbb{C}\Gamma,\mathrm{cl}(\gamma))$ is the quotient
\[ZC^n(\mathbb{C}\Gamma,\mathrm{cl}(\gamma))/BC^n(\mathbb{C}\Gamma,\mathrm{cl}(\gamma))\]
then $\psi_{\gamma,1}$ and $\psi_{\gamma,2}$ being cohomologous implies existence of a delocalized cyclic cocyle $\phi$ belonging to $(C^{n-1}(\mathbb{C}\Gamma,\mathrm{cl}(\gamma)),b)$ such that $\psi_{\gamma,1}-\psi_{\gamma,2}=b\phi$.  
\begin{myrem}
If $[\varphi_\gamma]\in H^n(\mathbb{C}\Gamma,\mathrm{cl}(\gamma))$ has polynomial growth representatives $\psi_{\gamma,1}$ and $\psi_{\gamma,2}$, then there exists a cyclic cocyle $\phi$ of polynomial growth such that $\psi_{\gamma,1}-\psi_{\gamma,2}=b\phi$.
\end{myrem}
\begin{proof}
For any length function $l$ on $\Gamma$ it is proven by Ji \cite[Theorem 2.23]{RJ92} that the inclusion $i:\mathbb{C}\Gamma\hookrightarrow S^l_1(\Gamma)$ induces an isomorphism between the Schwartz cohomology $^sH^n_1(\Gamma,\mathbb{C})=H^n(S^l_1(\Gamma),\mathbb{C})$ and the group cohomology $H^n(\Gamma,\mathbb{C})$ for $\Gamma$ a discrete countable group of polynomial growth. In particular, we have $^sH^n_l(Z_\gamma,\mathbb{C})\cong H^n(Z_\gamma,\mathbb{C})$ and so \Cref{proBb1} along with the strategy of \Cref{thmBb1} provides for a polynomial growth preserving map such that $\phi$ is the image of a representative of an element in $^sH^n_l(Z_\gamma,\mathbb{C})$. 
\end{proof}

\section{Pairing of Cyclic Cohomology Classes in Odd Dimension}
\subsection{Delocalized Higher Eta Invariant}\label{C1}
Let $\widetilde{D}$ be the Dirac operator lifted to $\widetilde{M}$, $s$ a section of the spinor bundle $\mathcal{S}$, and $\nabla:C^\infty(\widetilde{M},\mathcal{S})\longrightarrow C^\infty(\widetilde{M},T^*\widetilde{M}\otimes\mathcal{S})$ the connection on $\mathcal{S}$. Since $M$ has positive scalar curvature $\kappa>0$ associated to $\widetilde{g}$, then Lichnerowicz's formula \cite{AL63} 
\begin{equation}\label{Ca:Lichnerowicz}
\widetilde{D}^2s=\nabla\nabla^*\psi+\frac{\kappa(s)}{4}
\end{equation}
implies that $\widetilde{D}$ is invertible. Moreover, $\widetilde{D}$ is a self-adjoint elliptic operator, and so possesses a real spectrum: $\sigma(\widetilde{D})\subset\mathbb{R}$. The invertibility condition particularly provides existence of a spectral gap at 0, which will be necessary in ensuring convergence of the integral introduced in \Cref{defCa1}. This will be a higher analogue of the delocalized higher eta invariant Lott \cite{JL99} introduced in the case of 0-dimensional cyclic cocyles-- that is for traces. Given a non-trivial conjugacy class $\mathrm{cl}(\gamma)$ of the fundamental group $\Gamma=\pi_1(M)$ Lott's delocalized higher eta invariant can be formally defined as the pairing between Lott's higher eta invariant and traces.
\begin{equation}\label{Ca:LottsPairing}
\eta_{\mathrm{tr}_\gamma}(\widetilde{D}):=\frac{2}{\sqrt{\pi}}\int_{0}^{\infty}\mathrm{tr}_\gamma\left(\widetilde{D}e^{-t^2\widetilde{D}^2}\right)
\end{equation}
Here the trace map $\mathrm{tr}:\mathbb{C}\Gamma\longrightarrow\mathbb{C}$ continuously extends to a suitable smooth dense subalgebra of $C^*_r(\Gamma)$ to which $\widetilde{D}e^{-t^2\widetilde{D}^2}$ belongs. Generally, if $\mathcal{F}$ is a fundamental domain of $\widetilde{M}$ under the action of $\Gamma$, then for $\Gamma$-equivariant kernels $A\in C^\infty(\widetilde{M}\times\widetilde{M})$
\begin{equation}
\mathrm{tr}_\gamma(A)=\sum_{g\in\mathrm{cl}(\gamma)}\int_{\mathcal{F}}A(x,gx)\;dx
\end{equation}
Under the assumption of hyperbolicity or polynomial growth of the conjugacy class of $\gamma$, Lott \cite{JL99} showed convergence of the above integral. Invertibility of $\widetilde{D}$ is in general a necessary condition for this convergence, as was shown by the construction of a divergent counterexample by Piazza and Schick \cite[Section 3]{PS07}. However, it was proven by Cheng, Wang, Xie and Yu \cite[Theorem 1.1]{CWXY19} that as long the spectral gap of $\widetilde{D}$ is sufficiently large, then $\eta_{\mathrm{tr}_\gamma}(\widetilde{D})$ converges absolutely, and does not require any restriction on the fundamental group of the manifold.

Since we shall have occasion to use their properties often, we shall briefly recall the most important aspects of the space $S(\mathbb{R})$ of Schwartz functions. By definition, $f$ belongs to $S(\mathbb{R})$ if $f:\mathbb{R}\longrightarrow\mathbb{C}$ is a smooth function such that for every $k,m\in\mathbb{N}$
\[\lim_{|x|\longrightarrow\infty}x^k\frac{d^m}{dx^m}(f(x))=0\]
This implies that $f$ is bounded with respect to the family of semi-norms
\begin{equation}\label{Ca:ScwhartzNorm}
||f||_{k,m}=\sup_{x\in\mathbb{R}}\left|x^k\frac{d^m}{dx^m}(f(x))\right|
\end{equation}
Moreover, the Fourier transform $f\longmapsto\hat{f}$ is an automorphism of the Schwartz space, thus $\hat{f}\in S(\mathbb{R})$ for every Schwartz function $f$, where
\begin{equation}\label{Ca:Fourier}
\hat{f}(\xi)=\frac{1}{2\pi}\int_{-\infty}^{\infty}f(x)e^{-i\xi x}\;dx
\end{equation}
\begin{mylem}\label{lemCa1}
If $\Phi$ is a Schwartz function and $\widetilde{D}$ is the lifted Dirac operator associated to $\widetilde{M}$, then $\Phi(\widetilde{D})\in\mathscr{A}(\widetilde{M},\mathcal{S})^\Gamma$.
\end{mylem}
\begin{proof}
This is proven as Proposition 4.6 in \cite{XY19}
\end{proof}

\begin{mydef}\label{defCa1}
For any delocalized cyclic cocyle class $[\varphi_\gamma]\in HC^{2m}(\mathbb{C}\Gamma,\mathrm{cl}(\gamma))$ the delocalized higher eta invariant of $\widetilde{D}$ with respect to $[\varphi_\gamma]$ is defined as
\begin{equation}
\eta_{\varphi_\gamma}(\widetilde{D}):=\frac{m!}{\pi i}\int_{0}^{\infty}\eta_{\varphi_\gamma}(\widetilde{D},t)\;dt
\end{equation}
where $\eta_{\varphi_\gamma}(\widetilde{D},t)=\varphi_\gamma((\dot{u}_t(\widetilde{D}) u_t^{-1}(\widetilde{D}))\hat{\otimes}((u_t(\widetilde{D})-\mathbbm{1})\hat{\otimes}(u_t^{-1}(\widetilde{D})-\mathbbm{1}))^{\hat{\otimes} m})$ and
\[F_t(x)=\frac{1}{\sqrt{\pi}}\int_{-\infty}^{tx}e^{-s^2}\;ds\qquad u_t(x)=e^{2\pi iF_t(x)}\qquad \dot{u}_t(x)=\frac{d}{dt}u_t(x)\]
\end{mydef} 
Note that the arguments of $\eta_{\varphi_\gamma}(\widetilde{D},t)$ all belong to $\mathscr{A}(\widetilde{M},\mathcal{S})^\Gamma$, since $u_t(x)-1,u_t^{-1}(x)-1$ and $\dot{u}_t(x)u_t^{-1}(x)$ are all Schwartz functions. In particular, we have the simplification
\[\dot{u}_t(x)u_t^{-1}(x)=2\pi i\left(\frac{d}{dt}F_t(x)\right)e^{2\pi iF_t(x)}e^{-2\pi iF_t(x)}=2\pi i\left(\frac{d}{dt}F_t(x)\right)=2i\sqrt{\pi}xe^{-t^2x^2}\]
It is also useful to consider the representation of the delocalized higher eta invariant in terms of smooth Schwartz kernels, namely if $L_i$ is an element of the convolution algebra $\mathscr{L}(\widetilde{M},\mathcal{S})^\Gamma$, the action of $\varphi_{\gamma}$ on $\mathbb{C}\Gamma$ can be extended to $\mathscr{L}(\widetilde{M},\mathcal{S})^\Gamma$ by-- abusing notation a little we will denote the Schwartz kernel of $L_i$ by $L_i$ also-- defining $\varphi_{\gamma}(L_0\hat{\otimes}L_1\hat{\otimes}\ldots\hat{\otimes}L_n)$ to be
\begin{equation}
\sum_{g_0g_1\cdots g_n\in\mathrm{cl}(\gamma)}\varphi_\gamma(g_0,\ldots, g_n)\int_{\mathcal{F}^{n+1}}\mathsf{tr}\left(\prod_{i=0}^{n}L_i(x_i,g_ix_{i+1})\right)\;dx_0\cdots dx_n\quad:\;x_{n+1}=x_0
\end{equation}
where $\mathcal{F}$ is the fundamental domain of $\widetilde{M}$ under the action of $\Gamma=\pi_1(M)$, and $\mathsf{tr}$ denotes the pointwise matrix trace, not to be confused with the trace norm $||\cdot||_{tr}$ for trace class operators. Denoting by $a_t(x,y),b_t(x,y)$ and $k_t(x,y)$ the Schwartz kernels of the operators $u_t(\widetilde{D})-\mathbbm{1},u_t^{-1}(\widetilde{D})-\mathbbm{1}$ and $\dot{u}_t(\widetilde{D})u_t^{-1}(\widetilde{D})$ respectively, then $\eta_{\varphi_\gamma}(\widetilde{D},t)$ is given by 
\begin{equation}
\begin{gathered}
\sum_{g_0g_1\cdots g_{2m}\in\mathrm{cl}(\gamma)}\varphi_\gamma(\mathbf{g}_{2m})\int_{\mathcal{F}^{2m+1}}\mathsf{tr}\left(k_t(x_0,g_0x_1)\prod_{i=1}^{2m-1}a_t(x_i,g_ix_{i+1})b_t(x_{i+1},g_{i+1}x_{i+2})\right)\;d\mathbf{x}_{2m}\\
\mathbf{g}_{2m}:=(g_0,\ldots,g_{2m})\qquad d\mathbf{x}_{2m}:=dx_0\cdots dx_{2m}
\end{gathered}
\end{equation}
In this form we can better exploit the properties of $\mathscr{A}(\widetilde{M},\mathcal{S})^\Gamma$, in order to prove that $\eta_{\varphi_\gamma}(\widetilde{D})$ converges for $\widetilde{D}$ invertible and $\Gamma$ of polynomial growth. The first step is proving extension of delocalized cyclic cocycles on the smooth dense subalgebra $\mathscr{A}(\widetilde{M},\mathcal{S})^\Gamma$, in terms of kernel operators. Since the fundamental group of a manifold has a cocompact, isometric, and properly
discontinuous action on the universal cover, by the \v{S}varc-Milnor lemma \cite{AS55,JM68} there is a quasi-isometry $f:\pi_1(M)\longrightarrow\widetilde{M}$; for every $g,h\in\Gamma$ there exists $K\geq1,\ell\geq0$ such that
\[d_{\Gamma}(g,h)-K\ell\leq Kd_{\widetilde{M}}(f(g),f(h))\leq K^2d_{\Gamma}(g,h)+K\ell\] 
and for every $y\in\widetilde{M}$ there exists $g_y\in\Gamma$ such that $d_{\widetilde{M}}(f(g_y),y)\leq\ell$. In particular, we may fix some $p\in\widetilde{M}$ and define $f(g)=gp$; moreover, restricting our attention to points belonging to $\mathcal{F}$ the value of $(K,\ell)$ can be taken to be $(1,\mathrm{diam}(\mathcal{F}))$ since each orbit is cobounded.

We note that in \cite[Section 8]{CWXY19}, under the assumption that $\pi_1(M)$ is of polynomial growth, the authors used the techniques of \cite{AC88,AC91} to establish that the above definition of the de-localized higher eta invariant agrees with Lott’s higher eta invariant \cite[Section 4.4 \& 4.6]{JL92} up to a constant.

\begin{mythm}\label{thmCa1}
Let $\Gamma=\pi_1(M)$ and $\varphi_\gamma\in (C^n(\mathbb{C}\Gamma,\mathrm{cl}(\gamma)),b)$ be a delocalized cyclic cocyle of polynomial growth, then $\varphi_\gamma$ extends continuously on the algebra $(\mathscr{A}(\widetilde{M},\mathcal{S})^\Gamma)^{\hat{\otimes}_\pi^{n+1}}$.
\end{mythm}
\begin{proof}
Denote by $\hat{\rho}:\widetilde{M}\longrightarrow[0,\infty)$ the distance function $\hat{\rho}(x)=\hat{\rho}(x,y_0)$ for some fixed point $y_0\in\widetilde{M}$, with $\rho$ being the modification of $\hat{\rho}$ near $y_0$ to ensure smoothness. Let $B\in\mathscr{A}(\widetilde{M},\mathcal{S})^\Gamma$ and recall that we have the norm $||B||_{\mathscr{A},k}=||\widetilde{\partial}^k(B)\circ(\widetilde{D}^{2n_0}+1)||_{op}$; for any $f\in L^2(\widetilde{M},\mathcal{S})$ the Sobolev embedding theorem provides existence of some constant $C$ such that
\[|B(f)(x)|\leq C||(1+\widetilde{D}^{2n_0})B(f)||_{L^2(\widetilde{M},\mathcal{S})}\leq C||(\widetilde{D}^{2n_0}+1)B||_{op}||f||_{L^2(\widetilde{M},\mathcal{S})}\]
In particular, since $f$ is arbitrary, the bound $||B(x,\cdot)||_{L^2(\widetilde{M},\mathcal{S})}\leq C||(\widetilde{D}^{2n_0}+1)B||_{op}$ shows that taking the supremum over all $(x,y)\in\widetilde{M}\times\widetilde{M}$ the Schwartz kernel
\[\widetilde{\partial}^k(B)(x,y)=(\rho(x)-\rho(y))^kB(x,y)\]
has operator norm bounded by $||B||_{\mathscr{A},k}$, hence it is a uniformly bounded continuous function for all $k\in\mathbb{N}$. Now view $B(x,y)$ as a matrix acting on the spinors $f(y)$, where each section $f$ has a representation as a matrix in the complex Clifford algebra $\mathbb{C}\ell_{\mathrm{dim}(M)}$. If $I$ is the identity matrix, then by the Holder inequality for Schatten $p$-norms 
\[|\mathsf{tr}(B(x,y))|\leq||B(x,y)||_{tr}=||B(x,y)I||_{tr}\leq||B(x,y)||_{op}||I||_{tr}<2^{n_0}||B(x,y)||_{op}\]
Since all points of $\widetilde{M}$ belong to some orbit of the fundamental domain we have the bound $|\rho(x_{i})-\rho(x_{i+1})|\leq\mathrm{diam}(\mathcal{F})$, and by quasi-isometry of $\Gamma$ and the universal cover, we have (taking a family of quasi-isometries $f_i(g)=gx_i$)
\[|\rho(x_{i+1})-\rho(gx_{i+1})|\geq d_\Gamma(e,g)-\mathrm{diam}(\mathcal{F})=||g||-\mathrm{diam}(\mathcal{F})\]
From this, an application of the reverse triangle inequality provides the bound
\[|\rho(x_i)-\rho(gx_{i+1})|=|\rho(x_i)-\rho(x_{i+1})+\rho(x_{i+1})-\rho(gx_{i+1})|\]
\[\geq|\rho(x_{i+1})-\rho(gx_{i+1})|-|\rho(x_{i})-\rho(x_{i+1})|\geq||g||-\mathrm{diam}(\mathcal{F})-\mathrm{diam}(\mathcal{F})\]
Denoting the matrix norm $||\cdot||_{op}$ by $|\cdot|$, the boundedness properties of the Schwartz kernel implies existence of a constant $C_k>0$ such that for each $k\in\mathbb{N}$ 
\[C_k|\widetilde{\partial}^{3k}(B_i)(x_i,gx_{i+1})|^2=(\rho(x_i)-\rho(gx_{i+1}))^{6k}|B_i(x_i,gx_{i+1})|^2\geq (1+||g||)^{6k}|B_i(x_i,gx_{i+1})|^2\]
We will use the explicit representation $\varphi_{\alpha,\gamma}$ of $\varphi_\gamma$, and for ease of notation, we shorten the argument of $\alpha$ by writing $\alpha(\mathbf{g}_n)$; we wish to prove convergence of the following sum.
\begin{equation}
\varphi_{\alpha,\gamma}(B_0\hat{\otimes}B_1\hat{\otimes}\cdots\hat{\otimes}B_n)=\sum_{g_0g_1\cdots g_n\in\mathrm{cl}(\gamma)}\alpha(\mathbf{g})\int_{\mathcal{F}^{n+1}}\mathsf{tr}\left(\prod_{i=0}^{n}B_i(x_i,g_ix_{i+1})\right)\;dx_0\cdots dx_n
\end{equation}
From the fact that $\alpha$ is of polynomial growth, and using the above inequalities coupled with Cauchy-Schwartz, $|\varphi_{\alpha,\gamma}(B_0\hat{\otimes}B_1\hat{\otimes}\cdots\hat{\otimes}B_n)|$ is bounded above by
\[\sum_{g_0g_1\cdots g_n\in\mathrm{cl}(\gamma)}|\alpha(\mathbf{g}_n)|\int_{\mathcal{F}^{n+1}}\left|\mathsf{tr}\left(\prod_{i=0}^{n}B_i(x_i,g_ix_{i+1})\right)\right|\;dx_0\cdots dx_n\]
\[\leq\sum_{g_0g_1\cdots g_n\in\mathrm{cl}(\gamma)}2^{n_0}|\alpha(\mathbf{g}_n)|\int_{\mathcal{F}^{n+1}}\left|\prod_{i=0}^{n}B_i(x_i,g_ix_{i+1})\right|\;dx_0\cdots dx_n\]
\[\leq\sum_{g_0g_1\cdots g_n\in\mathrm{cl}(\gamma)}R_\alpha\prod_{i=0}^{n}(1+||g_i||)^{2k}\prod_{i=0}^{n}\left(\int_{\mathcal{F}^{2}}|B_i(x_i,g_ix_{i+1})|^2\;dx_idx_{i+1}\right)^{1/2}\]
\[=\sum_{g_0g_1\cdots g_n\in\mathrm{cl}(\gamma)}R_\alpha\prod_{i=0}^{n}\left(\int_{\mathcal{F}^{2}}(1+||g_i||)^{4k}|B_i(x_i,g_ix_{i+1})|^2\;dx_idx_{i+1}\right)^{1/2}\]
\[\leq R_\alpha C_k^{1/2}\prod_{i=0}^{n}\left(\sum_{g_i\in\Gamma}(1+||g_i||)^{-2k}\int_{\mathcal{F}^{2}}\sup_{(x_i,g_ix_{i+1})\in\mathcal{F}\times\mathcal{F}}|\widetilde{\partial}^{2k}(B_i)(x_i,g_ix_{i+1})|^2\;dx_idx_{i+1}\right)^{1/2}\]
For each $g_i$ the integral over the fundamental domain is finite, since $\widetilde{\partial}^{2k}(B_i)(x_i,g_ix_{i+1})$ is uniformly bounded; explicitly there exists a constant $\Lambda_k$ such that for $g_0g_1\cdots g_n\in\mathrm{cl}(\gamma)$ the value $|\varphi_{\alpha,\gamma}(B_0\hat{\otimes}B_1\hat{\otimes}\cdots\hat{\otimes}B_n)|$ is bounded above by
\begin{equation}\label{Ca:CocyleBound}
\begin{gathered}
R_\alpha C_k^{1/2}\prod_{i=0}^{n}\left(\sum_{g_i\in\Gamma}(1+||g_i||)^{-2k}\int_{\mathcal{F}^{2}}\Lambda_k^2||B_i||_{\mathscr{A},k}^2\;dx_idx_{i+1}\right)^{1/2}\\
\leq R_\alpha C_k^{1/2}\mathrm{diam}(\mathcal{F})\Lambda_k\prod_{i=0}^{n}\left(\sum_{g_i\in\Gamma}(1+||g_i||)^{-2k}||B_i||_{\mathscr{A},k}^2\right)^{1/2}
\end{gathered}
\end{equation}
where we denote $R_{\alpha,k}=R_\alpha C_k^{1/2}\mathrm{diam}(\mathcal{F})\Lambda_k$. Moreover, due to $\Gamma$ being of polynomial growth, there exists $k_i$ such that 
\[(1+||g_i||)^{-2k_i}|\{g_i\in\Gamma:||g_i||\leq c\}|<\frac{1}{c^2}\]
It follows that each of the sums in the final expression \eqref{Ca:CocyleBound} are finite for sufficiently large $k$, and thus so is any finite product of them. Now, by construction $\mathscr{L}(\widetilde{M})^\Gamma$ is a smooth dense sub-algebra of $\mathscr{A}(\widetilde{M})^\Gamma$, and this relationship also extends when considering their projective tensor products. We have just proven that $\varphi_{\alpha,\gamma}$ is continuous on $(\mathscr{A}(\widetilde{M})^\Gamma)^{\hat{\otimes}_\pi^{n+1}}$, hence to obtain the desired result it suffices to prove that for operators $B_0,\ldots,B_n\in\mathscr{L}(\widetilde{M})^\Gamma$
\begin{equation}
\mathrm{sgn}(\sigma)\varphi_{\alpha,\gamma}(B_0\hat{\otimes}B_1\hat{\otimes}\cdots\hat{\otimes}B_n)=\varphi_{\alpha,\gamma}(B_{\sigma(0)}\hat{\otimes}B_{\sigma(1)}\hat{\otimes}\cdots\hat{\otimes}B_{\sigma(n)})\end{equation}
whenever $\sigma\in S_{n+1}$ is a cyclic shift. By application of the Fubini-Tonelli theorem, and since $\varphi_{\alpha,\gamma}$ is a cyclic cocyle on $\mathbb{C}\Gamma$ we obtain $\varphi_{\alpha,\gamma}(B_{\sigma(0)}\hat{\otimes}B_{\sigma(1)}\hat{\otimes}\cdots\hat{\otimes}B_{\sigma(n)})$ is equal to
\[\sum_{g_0g_1\cdots g_n\in\mathrm{cl}(\gamma)}\mathrm{sgn}(\sigma)\varphi_{\alpha,\gamma}(\mathbf{g}_n)\int_{\mathcal{F}^{n+1}}\mathsf{tr}\left(\prod_{i=0}^{n}B_{\sigma(i)}(x_{\sigma(i)},g_{\sigma(i)}x_{\sigma(i+1)})\right)\;dx_{\sigma(0)}\cdots dx_{\sigma(n)}\]
\[=\sum_{g_0g_1\cdots g_n\in\mathrm{cl}(\gamma)}\mathrm{sgn}(\sigma)\varphi_{\alpha,\gamma}(\mathbf{g}_n)\int_{\mathcal{F}^{n+1}}\mathsf{tr}\left(\prod_{i=0}^{n}B_i(x_i,g_ix_{i+1})\right)\;dx_0\cdots dx_n\]
\end{proof}

The following technical result is one which we will have occasion to use often, both in the remainder of this section and elsewhere.
\begin{mypro}\label{proCa1}
For any collection of Schwartz functions $f_0,f_1,\ldots,f_n\in S(\mathbb{R})$ and any delocalized cyclic cocyle $\varphi_\gamma\in (C^n(\mathbb{C}\Gamma,\mathrm{cl}(\gamma)),b)$ of polynomial growth
\[\lim_{t\longrightarrow0}\varphi_\gamma(f_0(t\widetilde{D})\hat{\otimes}f_0(t\widetilde{D})\hat{\otimes}\cdots\hat{\otimes}f_n(t\widetilde{D}))=0\]	
\end{mypro}
\begin{proof}
Fix some $t\neq0$ and consider the Schwartz functions $f_{i,t}(x)=f_i(tx)$ for $1\leq i\leq n$; since the Fourier transform is an automorphism of $S(\mathbb{R})$ there exists $g_{i,t}\in S(\mathbb{R})$ such that $\hat{g}_{i,t}\equiv f_{i,t}$. Using the change of variables $x=y/t$, and the definition from \eqref{Ca:Fourier}, we obtain
\[f_{i,t}(\xi)=\hat{g}_{i,t}(\xi)=\frac{1}{2\pi}\int_{-\infty}^{\infty}g_{i,t}(x)e^{-i\xi x}\;dx=\frac{1}{2\pi}\int_{-\infty}^{\infty}g_i(tx)e^{-i\xi y/t}\;\frac{dy}{t}\]
\[=\frac{1}{2\pi t}\int_{-\infty}^{\infty}g_{i}(y)e^{-iy(\xi/t)}\;dy=\frac{\hat{g}_{i}(\xi/t)}{t}=f_i(t\xi)\]
Now since each of these functions is Schwartz, \eqref{Ca:ScwhartzNorm} asserts that the following limit exists and is finite; in particular,  $\hat{g}_{i}(t\xi)\longrightarrow0$ faster than any inverse power of $t$ as $t\longrightarrow\infty$, from which it follows that
\[\lim_{t\longrightarrow0}f_{i,t}(\xi)=\lim_{t\longrightarrow0}f_{i}(t\xi)=\lim_{t\longrightarrow0}\frac{\hat{g}_i(\xi/t)}{t}=\lim_{t\longrightarrow\infty}t\cdot\hat{g}_i(t\xi)=0\]
Turning to functional calculus, by \Cref{lemCa1} each $f_i(t\widetilde{D})$ belongs to $\mathscr{A}(\widetilde{M},\mathcal{S})^\Gamma$, and the spectral gap at 0 of $\widetilde{D}$ ensures that $||f_{i}(t\widetilde{D})||_{\mathscr{A},k}$ converges to 0 as $t\longrightarrow0$. From \eqref{Ca:CocyleBound} of \Cref{thmCa1} it follows that there exists a positive constant $R_{\alpha,k}$ such that
\[\lim_{t\longrightarrow0}|\varphi_\gamma(f_0(t\widetilde{D})\hat{\otimes}f_0(t\widetilde{D})\hat{\otimes}\cdots\hat{\otimes}f_n(t\widetilde{D}))|\]
\[\leq\lim_{t\longrightarrow0}R_{\alpha,k} \prod_{i=0}^{n}\left(\sum_{g_i\in\Gamma}(1+||g_i||)^{-2k}||f_i(t\widetilde{D})||_{\mathscr{A},k}^2\right)^{1/2}=0\]
\end{proof}

\begin{mylem}\label{lemCa2}
Let $[\varphi_\gamma]\in HC^{2m}(\mathbb{C}\Gamma,\mathrm{cl}(\gamma))$, then if $\widetilde{D}$ is invertible and $\varphi_\gamma$ is of polynomial growth, then $\eta_{\varphi_\gamma}(\widetilde{D})$ converges absolutely. 
\end{mylem}
\begin{proof}
The higher delocalized eta invariant can be split into two integrals, as follows.
\[\eta_{\varphi_\gamma}(\widetilde{D}):=\frac{m!}{\pi i}\int_{0}^{\infty}\eta_{\varphi_\gamma}(\widetilde{D},t)\;dt=\frac{m!}{\pi i}\left(\int_{0}^{1}\eta_{\varphi_\gamma}(\widetilde{D},t)\;dt+\int_{1}^{\infty}\eta_{\varphi_\gamma}(\widetilde{D},t)\;dt\right)\]
For the first integral, absolute convergence follows from \Cref{thmCa1} and \Cref{proCa1}, using the Schwartz kernel expression
\[\int_{0}^{1}|\eta_{\varphi_\gamma}(\widetilde{D},t)|\;dt\leq\sup_{t\in[0,1]}|\varphi_\gamma((\dot{u}_t(\widetilde{D})u_t^{-1}(\widetilde{D}))\hat{\otimes}((u_t(\widetilde{D})-\mathbbm{1})\hat{\otimes}(u_t^{-1}(\widetilde{D})-\mathbbm{1}))^{\hat{\otimes} m})|<\infty\]
For the second integral it is useful to work in the unitization $(\mathscr{A}(\widetilde{M},\mathcal{S})^\Gamma)^+$, for which $\overline{\varphi}_\gamma$ is well defined and continuous on the projective tensor product $((\mathscr{A}(\widetilde{M},\mathcal{S})^\Gamma)^+)^{\hat{\otimes}^{2m+1}_\pi}$ by \Cref{thmCa1}. Then following an argument similar to that of \cite[Proposition 3.30]{CWXY19} we have that $\int_{1}^{\infty}\eta_{\varphi_\gamma}(\widetilde{D},t)\;dt$ is bounded above by 
\[\bigg[\sup_{t\in[1,\infty)}\prod_{i=0}^{n}\left(\sum_{g_i\in\Gamma}(1+||g_i||)^{-2k}||\widetilde{D}e^{-\widetilde{D}^2}\hat{\otimes}(\bar{u}_t(\widetilde{D})\hat{\otimes}\bar{u}_t^{-1}(\widetilde{D}))^{\hat{\otimes} m}||_{(\mathscr{A}^+)^{\hat{\otimes}2m+1},k}^2\right)^{1/2}\]
\[\times\int_{1}^{\infty}R_{\alpha,k}Ce^{-(t^2-1)r^2/2}\;dt\bigg]<\infty\]
where $C,R_{\alpha,k}$ and $r$ are positive constants. 
\end{proof}

\begin{mythm}\label{thmCa2}
The higher delocalized eta invariant is independent of the choice of cocycle representative. Explicitly, if $[\varphi_\gamma]=[\phi_\gamma]\in HC^{2m}(\mathbb{C}\Gamma,\mathrm{cl}(\gamma))$, then $\eta_{\varphi_\gamma}(\widetilde{D})=\eta_{\phi_\gamma}(\widetilde{D})$
\end{mythm}
\begin{proof}
By hypothesis, $\varphi_\gamma$ and $\phi_\gamma$ are cohomologous via a coboundary $b\varphi$ belonging to $BC^{2m}(\mathbb{C}\Gamma,\mathrm{cl}(\gamma))$. By the results of \Cref{B2} we can assume $\varphi\in(C^{2m-1}(\mathbb{C}\Gamma,\mathrm{cl}(\gamma)),b)$ to be a skew cochain of polynomial growth, and it suffices to prove that $\eta_{b\varphi}(\widetilde{D})=0$. Working in $(\mathscr{A}(\widetilde{M},\mathcal{S})^\Gamma)^+$ we obtain the transgression formula of \cite[eq(3.23)]{CWXY19}
\begin{equation}\label{Ca:Transgression}
m\eta_{b\varphi}(\widetilde{D},t)=\frac{d}{dt}\overline{\varphi}((\bar{u}_t(\widetilde{D})\hat{\otimes} \bar{u}_t^{-1}(\widetilde{D}))^{\hat{\otimes} m})\end{equation}
Ignoring the constant term $\frac{m!}{\pi i}$ in the definition of the delocalized higher eta invariant and integrating both sides with respect to $t$
\[m\eta_{b\varphi}(\widetilde{D})=m\lim_{T\longrightarrow\infty}\int_{1/T}^{T}\eta_{b\varphi}(\widetilde{D},t)\;dt=m\lim_{T\longrightarrow\infty}\int_{1/T}^{T}\frac{d}{dt}\overline{\varphi}((\bar{u}_t(\widetilde{D})\hat{\otimes} \bar{u}_t^{-1}(\widetilde{D}))^{\hat{\otimes} m})\;dt\]
\[=m\lim_{T\longrightarrow\infty}\varphi((u_T(\widetilde{D})-\mathbbm{1}\hat{\otimes} u_T^{-1}(\widetilde{D})-\mathbbm{1})^{\hat{\otimes} m}-m\lim_{T\longrightarrow0}\varphi((u_T(\widetilde{D})-\mathbbm{1}\hat{\otimes} u_T^{-1}(\widetilde{D})-\mathbbm{1})^{\hat{\otimes} m})\]
That $\lim_{T\longrightarrow0}\varphi((u_T(\widetilde{D})-\mathbbm{1}\hat{\otimes} u_T^{-1}(\widetilde{D})-\mathbbm{1})^{\hat{\otimes} m})=0$ follows from \Cref{proCa1}. For $T\longrightarrow\infty$, since $\widetilde{D}$ has a spectral gap at 0, we have by the properties of holomorphic functional calculus that for $x\in(0,\infty)$ both $u_T(\widetilde{D})-\mathbbm{1}$ and $u_T^{-1}(\widetilde{D})-\mathbbm{1}$ converge in the $||\cdot||_{\mathscr{A},k}$ norm to 0. It thus follows from the bounds of \Cref{thmCa1} that 
\begin{equation}
\lim_{T\longrightarrow\infty}|\varphi((u_T(\widetilde{D})-\mathbbm{1}\hat{\otimes} u_T^{-1}(\widetilde{D})-\mathbbm{1})^{\hat{\otimes} m}|=0
\end{equation}
\end{proof}

\begin{mypro}\label{proCa2}
Let $S_\gamma^*:HC^{2m}(\mathbb{C}\Gamma,\mathrm{cl}(\gamma))\longrightarrow HC^{2m+2}(\mathbb{C}\Gamma,\mathrm{cl}(\gamma))$ be the delocalized Connes periodicity operator, then $\eta_{[\varphi_\gamma]}(\widetilde{D})=\eta_{[S_\gamma\varphi_\gamma]}(\widetilde{D})$ for every $[\varphi_\gamma]\in HC^{2m}(\mathbb{C}\Gamma,\mathrm{cl}(\gamma))$. 
\end{mypro}
\begin{proof}
We may assume by \Cref{corBb1} that $\varphi_\gamma$ is of polynomial growth; by the definition of $S_\gamma$ the cocyle $S_\gamma\varphi_\gamma$ is also of polynomial growth. Since our expression for $S_\gamma$ coincides with that of \cite[Definition 3.32]{CWXY19} the result follows from \cite[Proposition 3.33]{CWXY19}
\end{proof}

\subsection{Delocalized Pairing of Higher Rho Invariant}\label{C2}
To begin with, we recall the assumptions put on the spin manifold $M$, namely that it is closed and odd dimensional with a positive scalar curvature metric $g$. Let $\widetilde{D}$ be the Dirac operator lifted to the universal cover $\widetilde{M}$, $s$ a section of the spinor bundle $\mathcal{S}$, and $\nabla:C^\infty(\widetilde{M},\mathcal{S})\longrightarrow C^\infty(\widetilde{M},T^*\widetilde{M}\otimes\mathcal{S})$ the connection on $\mathcal{S}$. Since $\widetilde{M}$ has positive scalar curvature $\kappa>0$ associated to $\widetilde{g}$, then Lichnerowicz's formula shows that $\widetilde{D}$ is invertible, hence there exists a spectral gap at 0. Since $\widetilde{D}$ is an elliptic essentially self-adjoint operator, using a suitable normalizing function $\psi$, the operator $\psi(\widetilde{D})$ is bounded pseudo-local and self-adjoint. In particular, to emphasize the relationship with the delocalized higher eta invariant it is particularly useful that for $t\in(0,\infty)$ we consider
\begin{equation}\label{Cb:RhoNormFunc}
\psi(tx)=\frac{2}{\sqrt{\pi}}\int_{0}^{tx}e^{-s^2}\;ds
\end{equation}
Since there exist a spectral gap at 0 with respect to $\widetilde{D}$, the limit $||\lim_{t\longrightarrow0}\psi(\widetilde{D}/t)||_{op}$ exists and converges to $||(\widetilde{D}|\widetilde{D}|^{-1})||_{op}$, where $\widetilde{D}|\widetilde{D}|^{-1}=\mathrm{signum}(\widetilde{D})$. Define an operator $H_0=\frac{1}{2}(\mathbbm{1}+\widetilde{D}|D|^{-1})$ and let $\{\phi_{s,j}\}$ be a partition of unity subordinate to the $\Gamma$-invariant locally finite open cover $\{U_{s,j}\}_{s,j\in\mathbb{N}}$ of $\widetilde{M}$. For each $s$ we take $\mathrm{diam}(U_{s,j})<\frac{1}{s}$, and so for $t\geq0$ we follow the construction of \cite[Section 2.3]{XY14} to form the operator
\begin{equation}\label{Cb:RhoConstr}
H(t)=\sum_j(s+1-t)\phi_{s,j}^{1/2}H_0\phi_{s,j}^{1/2}+(t-s)\phi_{s+1,j}^{1/2}H_0\phi_{s+1,j}^{1/2}\;:t\in[s,s+1]
\end{equation}
Since the support $\mathrm{supp}(\phi_{s,j})$ of each member of the partition of unity is a subset of $U_{s,j}$ the propagation of $H(t)$ tends to 0 as $t\longrightarrow\infty$. Together with $H(t)$ being a pseudo-local, self-adjoint bounded operator, this gives us that $H(t)\in D^*(\widetilde{M},\mathcal{S})^\Gamma$; moreover, due to the choice of $\chi$ we have $H'(t),H(t)^2-\mathbbm{1}\in C^*(\widetilde{M},\mathcal{S})^\Gamma$. Moreover, as $H(t)$ is a projection, the path of invertibles 
\[S=\{u(t)=\exp(2\pi iH(t))|\;t\in[0,\infty)\}\] 
belong to $(C^*(\widetilde{M},\mathcal{S})^\Gamma)^+$, and since $\exp(2\pi i\cdot\mathrm{signum}(x))=1$ for any $x\neq0$ we have by construction that $u(0)=\mathbbm{1}$. It follows that $u$ belongs to the kernel of the evaluation map \[\mathsf{ev}:(C_L^*(\widetilde{M},\mathcal{S})^\Gamma)^+\longrightarrow(C^*(\widetilde{M},\mathcal{S})^\Gamma)^+\]
and so the path $S$ gives rise to a K-theory class $[u]\in K_1(C_{L,0}^*(\widetilde{M},\mathcal{S})^\Gamma)$, which is by definition the higher rho invariant $\rho(\widetilde{D},\widetilde{g})$ of Higson and Roe \cite{HR05I,HR05II,HR05III}. Before defining the pairing between cyclic cocyles and the higher rho invariant it is useful to introduce a few technical notions which will be needed later on. By \Cref{proAa1} any class of invertible $[u]\in K_1(C_{L,0}^*(\widetilde{M},\mathcal{S})^\Gamma)$ is directly equivalent to a class of invertible $[u]\in K_1(\mathscr{B}_{L,0}(\widetilde{M},\mathcal{S})^\Gamma)$, and we also recall that the (localized)-equivariant Roe algebra is independent of the choice of admissible module, hence we will work within the framework of $\mathscr{B}(\widetilde{M})^\Gamma$.
\begin{mydef}\label{defCb1}
Consider the unitization $(\mathscr{B}(\widetilde{M})^\Gamma)^+$ of the algebra $\mathscr{B}(\widetilde{M})^\Gamma$, and its suspension $S\mathscr{B}(\widetilde{M})^\Gamma$. If $\mathcal{A}$ is a $C^*$-algebra recall that the suspension $S\mathcal{A}$ is defined as 
\[\{f\in C([0,1],\mathcal{A})\;|\;f(0)=f(1)=0\}\]
Identify $S^1$ with the quotient space $[0,1]/(0\sim1)$, and call an element $f\in S\mathscr{B}(\widetilde{M})^\Gamma$ invertible if it is a piecewise smooth loop $f:S^1\longrightarrow(\mathscr{B}(\widetilde{M})^\Gamma)^+$ of invertible elements satisfying $f(0)=f(1)=\mathbbm{1}$. The map $f$ is \textit{local} if there exists $f_L\in S\mathscr{B}_L(\widetilde{M})^\Gamma$ such that the following hold
\begin{enumerate}[label={(\roman{enumi})}]
\item $f_L:S^1\longrightarrow(\mathscr{B}_L(\widetilde{M})^\Gamma)^+$ is a loop of invertible elements satisfying $f_L(0)=f_L(1)=\mathbbm{1}$.
\item $f$ is the image of $f_L$ under the evaluation map $\mathsf{ev}:S\mathscr{B}_L(\widetilde{M})^\Gamma\longrightarrow S\mathscr{B}(\widetilde{M})^\Gamma$
\end{enumerate}
\end{mydef}
Recall that identifying the Bott generator $b$ as the class $[e^{2\pi i\theta}]\in K_1(C_0(\mathbb{R}))$ the Bott periodicity map $\beta$ provides the following relationship between idempotents of a $C^*$-algebra $\mathcal{A}$ and invertibles of the suspension
\begin{equation}
\beta:K_0(\mathcal{A})\longrightarrow K_1(S\mathcal{A})\qquad\beta[p]=[bp+(1-p)]
\end{equation}
Combining this with the Baum-Douglas geometric description of K-homology we obtain the following result concerning the propagation properties of local loops which is essentially the same as \cite[Lemma 3.4]{XY19}, and we refer the reader to the proof given in that paper.
\begin{mylem}\label{lemCb1}
If $f\in S\mathscr{B}(\widetilde{M})^\Gamma$ is a local invertible then for any $\varepsilon>0$ there exists an idempotent $p\in\mathscr{B}(\widetilde{M})^\Gamma$ such that $\mathsf{prop}(p)\leq\varepsilon$ and $f(\theta)$ is homotopic to the element $\psi(\theta)=e^{2\pi i\theta}p+(1-p)$ through a piecewise smooth family of invertible elements.
\end{mylem}

We now go through the process of assigning to any class $[u]\in K_1(C_{L,0}^*(\widetilde{M},\mathcal{S})^\Gamma)$ a special representative which will enable the calculations later on in this section. Making use of the results proved in \cite[Proposition 3.5]{XY19}, there exist a piecewise smooth path of invertible elements $h(t)\in\mathscr{B}(\widetilde{M})^\Gamma$ connecting $u(1)$ and $e^{2\pi i\frac{E(1)+1}{2}}$ where the operator $E:[1,\infty)\longrightarrow D^*(\widetilde{M})^\Gamma$ has uniformly bounded operator norm and satisfies 
\[\lim_{t\longrightarrow\infty}\mathsf{prop}(E(t))=0\qquad E'(t)\in \mathscr{B}(\widetilde{M})^\Gamma\qquad E(t)^2-1\in \mathscr{B}(\widetilde{M})^\Gamma\qquad E^*(t)=E(t)\]
where there exists a twisted Dirac operator $\widetilde{D}$ over a $spin^c$ manifold, along with smooth normalizing function $\chi:\mathbb{R}\longrightarrow[-1,1]$ such that $E(t)=\chi(\widetilde{D}/t)$. We can thus define the \textit{regularized representative} of $u$ to be
\begin{equation}
w(t)=\left\{\begin{array}{cc}
u(t) & 0\leq t\leq 1\\
h(t) & 1\leq t\leq 2\\
e^{2\pi i\frac{E(t-1)+1}{2}} & t\geq2
\end{array}\right.
\end{equation}
Moreover, by \cite[Theorem 3.8]{NK99} and \cite[Proposition 3.5]{XY19} if $v$ is another such representative then there exist a family of piecewise smooth maps $\{E_s\}_{s\in[0,1]}$ belonging to $D_{L,0}^*(\widetilde{M})^\Gamma$ and having the same properties as $E$. In particular, the propagation of $E_s(t)$ goes to zero uniformly in $s$ as $t\longrightarrow\infty$, and 
\[\frac{\partial}{\partial t}E_s(t)\in \mathscr{B}(\widetilde{M})^\Gamma\qquad E_s(t)^2-1\in \mathscr{B}(\widetilde{M})^\Gamma\qquad E_s(t)^*=E_s(t)\]
Furthermore there exists piecewise smooth family of invertibles $\{v_s\}_{s\in[0,1]}$ belonging to $(\mathscr{B}_{L,0}(\widetilde{M})^\Gamma)^+$ and which satisfy
\begin{enumerate}[label={(\roman{enumi})}] 
\item $v_0(t)=w(t)$ for $t\in[0,\infty)$, and $v_1(t)=v(t)$ for all $t\notin(1,2)$ 
\item $v_s(t)=\exp(2\pi i\frac{E_s(t-1)+1}{2})$ for all $t\geq2$
\item $v_1v^{-1}:[1,2]\longrightarrow(\mathscr{B}(\widetilde{M})^\Gamma)^+$ is a local loop of invertible elements.
\end{enumerate}

\begin{mydef}\label{defCb2}
Let $a_i=\sum_{g_i\in\Gamma}c_{g_i}\cdot g_i$ be an element of the group algebra $\mathbb{C}\Gamma$, and $\omega_i$ belong to the algebra $\mathscr{R}$ of smooth operators on a closed oriented Riemannian manifold. Denoting $W_i=a_i\otimes\omega_i$, the action of $\varphi_{\gamma}$ on $\mathbb{C}\Gamma$ can be extended to $\mathbb{C}\Gamma\otimes\mathscr{R}$ by 
\[\varphi_{\gamma}(W_0\hat{\otimes}W_1\hat{\otimes}\cdots\hat{\otimes}W_n)=\mathsf{tr}(\omega_0\omega_1\cdots\omega_n)\cdot\varphi_{\gamma}(a_0,a_1,\ldots,a_n)\]
\end{mydef}

\begin{mydef}\label{defCb3}
Given $[\rho(\widetilde{D},\widetilde{g})]\in K_1(\mathscr{B}_{L,0}(\widetilde{M})^\Gamma)$ with $w$ being its regularized representative, associated to each delocalized cyclic cocyle $[\varphi_{\gamma}]\in HC^{2m}(\mathbb{C}\Gamma,\mathrm{cl}(\gamma))$ the determinant map $\tau_{\varphi_\gamma}$ is defined by
\[\tau_{\varphi_\gamma}(\rho(\widetilde{D},\widetilde{g})):=\frac{1}{\pi i}\int_{0}^{\infty}\overline{\varphi}_{\gamma}(\widetilde{\mathsf{ch}}(w(t),\dot{w}(t))\;dt\]
\[\widetilde{\mathsf{ch}}(w,\dot{w})=(-1)^{m}\left(m-1\right)!\sum_{j=1}^{m}\left((w^{-1}\hat{\otimes} w)^{\hat{\otimes} j}\hat{\otimes}(w^{-1}\dot{w})\hat{\otimes}(w^{-1}\hat{\otimes} w)^{\hat{\otimes} (m-j)}\right)\]
\end{mydef}
We remark that the definition of $\widetilde{\mathsf{ch}}$ is directly modeled upon the secondary odd Chern character pairing invertibles of GL$(N,\mathbb{C})$ and traces (see, for example, \cite[Section 1.2]{LMP09}). Moreover, by the property of cyclic cocyles this expression can be simplified so that our coefficients exactly resemble that of the delocalized higher eta invariant. By the action of the cyclic operator $\mathfrak{t}$ we obtain from applying the $n=2(m-j)+1$ fold composition $\mathfrak{t}^n$ for each $j$, that the integrand can be written as
\begin{gather*}
\frac{(-1)^{m}\left(m-1\right)!}{\pi i}\sum_{j=1}^m(-1)^{2m(2m-2j+1)}\overline{\varphi}_{\gamma}\left((w^{-1}(t)\dot{w}(t))\hat{\otimes}(w^{-1}(t)\hat{\otimes}w(t))^{\hat{\otimes}m}\right)
\end{gather*}
from which it follows that there is the simplified expression for the determinant map
\begin{equation}\label{Cb:TauSimple}
\begin{gathered}
\tau_{\varphi_\gamma}(\rho(\widetilde{D},\widetilde{g})):=\frac{(-1)^mm!}{\pi i}\int_{0}^{\infty}\overline{\varphi}_{\gamma}\left((w^{-1}(t)\dot{w}(t))\hat{\otimes}(w^{-1}(t)\hat{\otimes}w(t))^{\hat{\otimes}m}\right)\;dt
\end{gathered}
\end{equation}
It is not at all obvious why the above pairing is well defined, the resolving of this doubt occupying the remainder of this section.
\begin{mythm}\label{thmCb1}
Let $\Gamma=\pi_1(M)$ and $\varphi_\gamma\in (C^n(\mathbb{C}\Gamma,\mathrm{cl}(\gamma)),b)$ be a delocalized cyclic cocyle of polynomial growth, then $\varphi_\gamma$ extends continuously on the algebra $(\mathscr{B}(\widetilde{M})^\Gamma)^{\hat{\otimes}_\pi^{n+1}}$.
\end{mythm}
\begin{proof}
Again using the explicit representation for delocalized cyclic cocyles, we show that $\varphi_{\alpha,\gamma}$ extends to a continuous multi-linear map on $(\mathscr{B}(\widetilde{M})^\Gamma)^{\hat{\otimes}_\pi^{n+1}}$. By fixing a basis, let $B_k\in\mathscr{B}(\widetilde{M})^\Gamma$ be represented by the matrix $(\beta_{ij}^k)_{i,j\in\mathbb{N}}$ with $\beta_{ij}^k\in C^*_r(\Gamma)$. We wish to prove convergence of 
\begin{equation}\label{Cb:CocyleBound}
\begin{gathered}
\varphi_{\alpha,\gamma}(B_0\hat{\otimes}B_1\hat{\otimes}\cdots\hat{\otimes}B_n)=\varphi_{\alpha,\gamma}\left(\sum_{g_0\in\Gamma}B_0(g_0)\cdot g_0\hat{\otimes}\cdots\hat{\otimes}\sum_{g_n\in\Gamma}B_n(g_n)\cdot g_n\right)\\	
=\sum_{g_0\in\Gamma}\sum_{g_1\in\Gamma}\cdots\sum_{g_n\in\Gamma}\mathsf{tr}(B_0(g_0)B_1(g_1)\cdots B_n(g_n))\cdot\varphi_{\alpha,\gamma}\left(g_0, g_1,\cdots,g_n\right)\\
=\sum_{g_0g_1\cdots g_n\in\mathrm{cl}(\gamma)}\mathsf{tr}(C(g_0,\ldots,g_n))\cdot\alpha(h,hg_0,\ldots,hg_0g_1\cdots g_{n-1})
\end{gathered}
\end{equation}
Straight forward matrix multiplication gives the product $C=(c_{ij})_{i,j\in\mathbb{N}}$ as having entries
\[c_{ij}(g_0,\ldots,g_n)=\sum_{k_{n-1}\in\mathbb{N}}\cdots\sum_{{k_1}\in\mathbb{N}}\sum_{{k_0}\in\mathbb{N}}\left(\beta^0_{ik_0}(g_0)\beta^1_{k_0k_1}(g_1)\cdots\beta^n_{k_{n-1}j}(g_n)\right)\]
For ease of notation, we shorten the argument of $\alpha$ by writing $\alpha(\mathbf{g})$; in addition we suppress the argument of the functions $c_{ij}$. Taking the desired trace in the above expression \eqref{Cb:CocyleBound}, and using the fact that $\alpha$ is of polynomial growth, we obtain the inequality
\begin{equation}\label{Cb:CocyleBound2}
\begin{gathered}
|\varphi_{\alpha,\gamma}(B_0\hat{\otimes}B_1\hat{\otimes}\cdots\hat{\otimes}B_n)|\leq\sum_{j\in\mathbb{N}}\sum_{g_0g_1\cdots g_n\in\mathrm{cl}(\gamma)}|c_{jj}||\alpha(\mathbf{g})|\\
\leq\sum_{j\in\mathbb{N}}\sum_{g_0g_1\cdots g_n\in\mathrm{cl}(\gamma)}R_\alpha(1+||g_0||)^{2k}(1+||g_1||)^{2k}\cdots(1+||g_n||)^{2k}|c_{jj}|
\end{gathered}
\end{equation}
for some positive constant $R_\alpha$. Next, considering the following inequality for $|c_{jj}|$ 
\[\sum_{j\in\mathbb{N}}|c_{jj}|\leq\sum_{k_0,\ldots,k_{n-1},j\in\mathbb{N}}\left|\left(\beta^0_{jk_0}(g_0)\cdots\beta^n_{k_{n-1}j}(g_n)\right)\right|\leq\prod_{i=0}^{n}\left(\sum_{k_{i-1},k_i\in\mathbb{N}}|\beta^i_{k_{i-1}k_i}(g_i)|^2\right)^{1/2}\quad \]
where $j=k_{-1}=k_n$. Substituting the above final product into the second line of \eqref{Cb:CocyleBound2}, $|\varphi_{\alpha,\gamma}(B_0\hat{\otimes}B_1\hat{\otimes}\cdots\hat{\otimes}B_n)|$ is bounded above by 
\begin{gather*}
\sum_{g_0g_1\cdots g_n\in\mathrm{cl}(\gamma)}R_\alpha\prod_{i=0}^{n}(1+||g_i||)^{2k}\prod_{i=0}^{n}\left(\sum_{k_{i-1},k_i\in\mathbb{N}}|\beta^i_{k_{i-1}k_i}(g_i)|^2\right)^{1/2}\\
=R_\alpha\prod_{i=0}^{n}\sum_{g_0g_1\cdots g_n\in\mathrm{cl}(\gamma)}\left(\sum_{k_{i-1},k_i\in\mathbb{N}}(1+||g_i||)^{4k}|\beta^i_{k_{i-1}k_i}(g_i)|^2\right)^{1/2}\\
\leq R_\alpha\prod_{i=0}^{n}\left(\sum_{g_i\in\Gamma}\sum_{k_{i-1},k_i\in\mathbb{N}}(1+||g_i||)^{4k}|\beta^i_{k_{i-1}k_i}(g_i)|^2\right)^{1/2}\quad:\;g_0g_1\cdots g_n\in\mathrm{cl}(\gamma)
\end{gather*}
In particular, the proof of \cite[Lemma 6.4]{CM90} shows that each of the double sums in the final expression are in fact bounded by the norm $||B_i||_{\mathscr{B},k}$ hence finite for all $k$, and thus so is any finite product of them. Since $\mathbb{C}\Gamma\otimes\mathscr{R}$ is a smooth dense sub-algebra of $\mathscr{B}(\widetilde{M})^\Gamma$ and $\varphi_{\alpha,\gamma}$ has been proven to be continuous on $\mathscr{B}(\widetilde{M})^\Gamma$, it suffices to prove that for operators $W_0,\ldots,W_n\in\mathbb{C}\Gamma\otimes\mathscr{R}$
\begin{equation}
\mathrm{sgn}(\sigma)\varphi_{\alpha,\gamma}(W_0\hat{\otimes}W_1\hat{\otimes}\cdots\hat{\otimes}W_n)=\varphi_{\alpha,\gamma}(W_{\sigma(0)}\hat{\otimes}W_{\sigma(1)}\hat{\otimes}\cdots\hat{\otimes}W_{\sigma(n)})
\end{equation}
whenever $\sigma\in S_{n+1}$ is a cyclic shift. Write $W_i=a_i\otimes\omega_i$, then using the fact that the trace is invariant under cyclic shifts and that $\varphi_{\alpha,\gamma}$ is a cyclic cocyle on $\mathbb{C}\Gamma$
\[\varphi_{\alpha,\gamma}(W_{\sigma(0)}\hat{\otimes}W_{\sigma(1)}\hat{\otimes}\cdots\hat{\otimes}W_{\sigma(n)})=\mathsf{trace}(\omega_{\sigma(0)}\omega_{\sigma(1)}\cdots\omega_{\sigma(n)})\cdot\varphi_{\alpha,\gamma}(a_{\sigma(0)},a_{\sigma(1)},\ldots,a_{\sigma(n)})\]
\[=\mathsf{trace}(\omega_0\omega_1\cdots\omega_n)\cdot\mathrm{sgn}(\sigma)\varphi_{\alpha,\gamma}(a_0,a_1,\ldots,a_n)\]
\[=\mathrm{sgn}(\sigma)\varphi_{\alpha,\gamma}(a_0\otimes\omega_0,a_1\otimes\omega_1,\ldots,a_n\otimes\omega_n)=\mathrm{sgn}(\sigma)\varphi_{\alpha,\gamma}(W_0\hat{\otimes}W_1\hat{\otimes}\cdots\hat{\otimes}W_n)\]
\end{proof}

\begin{mypro}\label{proCb1}
Let $w$ be a regularized representative of some class $[u]\in K_1(\mathscr{B}_{L,0}(\widetilde{M})^\Gamma)$. For all $t\geq 2$ there exists a finite propagation operator $\varpi\in\mathscr{B}(\widetilde{M})^\Gamma$ such that for every $k>0$ and any $\varepsilon>0$
\[||w(t)-\varpi(t)||_{\mathscr{B},k}<C_k/t^3 \qquad\mathsf{prop}(\varpi(t))<\varepsilon\]
whenever $t\geq\max\{e^2,6r\}$, where $C_k$ and $r$ are positive constants . 
\end{mypro}
\begin{proof}
For real valued $x$, finite $r>0$ and $z=\pi i(x+1)\in B_r(0)$ define the bounded holomorphic function $f(z)=e^{2\pi i\frac{x+1}{2}}$. Consider the Taylor series expansion of $f(z)$ centered at the origin, and for each $m\in\mathbb{N}$ define
\begin{equation}
\begin{gathered}
P_m(z)=\sum_{k=0}^{m}\frac{\left(2\pi i\frac{x+1}{2}\right)^k}{k!}\qquad R_m(z)=\sum_{k=m+1}^{\infty}\frac{\left(2\pi i\frac{x+1}{2}\right)^k}{k!}\\
A_m(t)=P_m(G(t))=P_m\left(2\pi i\frac{E(t-1)+1}{2}\right)
\end{gathered}
\end{equation}
We know that $E(t)$ has compact real spectrum which is symmetric around $\lambda=0$; in particular the spectral radius $\mathrm{rad}(\sigma(E(t)))$ is bounded above by $||E(t)||_{op}$. It is clear that the same holds true for $G(t)$, so choose $r>\sup_{t\geq1}\{||G(t)||_{op}\}$ so that $\sigma(G(t))\subset B_r(0)$, then $f\in H^\infty(B_{r}(0))$ and since $\mathscr{B}(\widetilde{M})^\Gamma$ is closed under holomorphic functional calculus, the usual remainder bound 
\[|R_m(z)|\leq c\frac{|z|^{m+1}}{(m+1)!}\quad\mathrm{if}\quad |f^{(m+1)}(z)|\leq c,\forall z\in B_r(0)\]
has a functional equivalent with respect to every $||\cdot||_{\mathscr{B},k}$. In particular for every $k>0$, there exists a constant $C_{k,t}>0$ such that
\[||w(t)-A_m(t)||_{\mathscr{B},k}=||R_m(G(t))||_{\mathscr{B},k}\leq C_{k,t}||R_m||_\infty=C_{k,t}\sup_{z\in B_r(0)}|R_m(z)|\]
\[\leq C_{k,t}\cdot c\sup_{z\in B_r(0)}\frac{|z|^{m+1}}{(m+1)!}\leq C_{k,t}\cdot c\frac{r^{m+1}}{(m+1)!}=C_{k,t}\frac{r^{m+1}}{(m+1)!}\]
Note that we can take $c=1$ since $|f^{(m+1)}(z)|=|e^z|=|e^{i(\pi x+\pi)}|=1$. Suppose that $m\geq\max\{e^2,6r\}$, then by applying Stirling's approximation 
\[C_{k,t}\frac{r^{m+1}}{(m+1)!}\leq C_{k,t}\frac{r^{m+1}}{\sqrt{2\pi}(m+1)^{m+3/2}e^{-(m+1)}}<\frac{C_{k,t}}{m^3}\]
Since the exponential is an entire function its power series converges uniformly on the compact set $\overline{B}_r(0)$, and thus we can choose a finite $C_k\geq\sup_{t\in[2,\infty)}\{C_{k,t}\}$. Finally, for all $t\geq2$ we define the operator $\varpi(t)$ according to
\begin{equation}
\varpi(t)=\sum_{m=2}^{\infty}1_{[m,m+1)}(t)\cdot A_m(t)
\end{equation}
From the definition of $E(t)$ the propagation tends to 0 as $t\longrightarrow\infty$; hence for all $\varepsilon>0$ there exists $N_\varepsilon\in\mathbb{N}$ such that $\mathsf{prop}(E(t))<\varepsilon/N_\varepsilon$ whenever $t\geq N_\varepsilon$. Recall that if $S$ and $T$ are bounded operators on some module $H_X$ then 
\[\mathsf{prop}(ST)\leq\mathsf{prop}(S)+\mathsf{prop}(T)\qquad\mathsf{prop}(S+T)\leq\max\{\mathsf{prop}(S),\mathsf{prop}(T)\}\]
thus by the definition of $\varpi(t)$, if $t\in[m,m+1]$ for $m\leq N_\varepsilon\leq t$
\[\mathrm{prop}(\varpi(t+1))=\mathrm{prop}(A_m(t+1))=\mathrm{prop}\left(\sum_{k=0}^{m}\frac{\left(2\pi i\frac{E(t)+1}{2}\right)^k}{k!}\right)\leq\mathrm{prop}\left(\frac{\left(2\pi i\frac{E(t)+1}{2}\right)^m}{m!}\right)\]
\[=\mathrm{prop}\left(\frac{E(t)+1}{2}\right)^m\leq m\cdot\mathrm{prop}\left(\frac{E(t)+1}{2}\right)\leq m\cdot\mathrm{prop}(E(t))<\frac{m\varepsilon}{N_\varepsilon}\leq\varepsilon\]
\end{proof}
\begin{mycor}\label{corCb1}
Let $w$ be a regularized representative of some class $[u]\in K_1(\mathscr{B}_{L,0}(\widetilde{M})^\Gamma)$. For all $t\geq 2$ there exists a finite propagation operator $\varpi\in\mathscr{B}(\widetilde{M})^\Gamma$ such that for every $k>0$ and any $\varepsilon>0$
\[||w^{-1}(t)-\varpi^{-1}(t)||_{\mathscr{B},k}<C_k/t^3 \qquad\mathsf{prop}(\varpi^{-1}(t))<\varepsilon\]
whenever $t\geq\max\{e^2,6r\}$, where $C_k$ and $r$ are positive constants . 
\end{mycor}
\begin{proof}
Take $f(z)=e^{-2\pi i\frac{x+1}{2}}$ and apply the same argument as above.
\end{proof}

For each member of the family of invertibles $\{v_s\}_{s\in[0,1]}$ the above also leads to finite propagation operators $\varpi_s(t)$ and $\varpi^{-1}_s(t)$ having analogous properties. The following technical result will be of similar importance in establishing well-definedness of the determinant map. 
\begin{myrem}\label{remCb1}
For $\varphi_\gamma\in C^{n}((\mathbb{C}\Gamma,\mathrm{cl}(\gamma)),b)$ and $B_0,\ldots,B_n\in\mathscr{B}(\widetilde{M},\mathcal{S})^\Gamma$-- or equivalently for $A_0,\ldots,A_n\in\mathscr{A}(\widetilde{M},\mathcal{S})^\Gamma$-- there exists $\varepsilon_{\widetilde{M}}$ which depends only on $M$ such that
\[\varphi_\gamma(B_0\hat{\otimes}B_1\hat{\otimes}\cdots\hat{\otimes}B_n)=0\]
whenever $\mathsf{prop}(B_i)<\varepsilon_{\widetilde{M}}$ for each $0\leq i\leq n$.
\end{myrem}

\begin{mythm}\label{thmCb2}
Let $[u]\in K_1(\mathscr{B}_{L,0}(\widetilde{M})^\Gamma)$ and $w$ be a regularized representative of $u$, then the determinant map $\tau_{\varphi}$ converges absolutely for any $\varphi_\gamma\in (C^{2m}(\mathbb{C}\Gamma,\mathrm{cl}(\gamma)),b)$ of polynomial growth.
\end{mythm}
\begin{proof}
Using the simplified expression of \Cref{Cb:TauSimple} we can write $\tau_{\varphi}(u)$ as 
\[\frac{(-1)^mm!}{\pi i}\int_{0}^{2}\varphi_{\gamma}\left((w^{-1}(t)\dot{w}(t))\hat{\otimes} (w^{-1}(t)\hat{\otimes} w(t))^{\hat{\otimes}m}\right)\;dt\]
\[+\frac{(-1)^mm!}{\pi i}\int_{2}^{\infty}\varphi_{\gamma}\left((w^{-1}(t)\dot{w}(t))\hat{\otimes} (w^{-1}(t)\hat{\otimes} w(t))^{\hat{\otimes}m}\right)\;dt\]
The first integral is finite due to the uniform boundedness of $||B||_{\mathscr{B},k}$ and the results of \Cref{thmCb1}, being bounded above by 
\[2R_\alpha\frac{m!}{\pi i}\sup_{t\in[0,2]}||w(t)||_{\mathscr{B},k}||w^{-1}(t)||_{\mathscr{B},k}||w^{-1}(t)\dot{w}(t)||_{\mathscr{B},k}<\infty\]
By \Cref{proCb1} and its corollary there exist operators $\varpi(t)$ and $\varpi^{-1}(t)$ belonging to $(\mathscr{B}(\widetilde{M})^\Gamma)^+$ such that for any $\varepsilon>0$ there exists $t$ large enough such that
\[\mathsf{prop}(\varpi(t)),\mathsf{prop}(\varpi^{-1}(t))<\varepsilon\]
By basic distribution of tensors over addition and mutlilinearity of cyclic cocyles we obtain the following expansion-- ignoring the constant for the integral over $t\in[2,\infty)$
\begin{equation}\label{Cb:TauDistribution}
\begin{gathered}
\int_{2}^{\infty}\varphi_{\gamma}\left((w^{-1}(t)\dot{w}(t))\hat{\otimes}(w^{-1}(t)-\varpi^{-1}(t)\hat{\otimes} w(t))^{\hat{\otimes}m}\right)\;dt\\
+\int_{2}^{\infty}\varphi_{\gamma}\left((w^{-1}(t)\dot{w}(t))\hat{\otimes}(w^{-1}(t)\hat{\otimes} w(t)-\varpi(t))^{\hat{\otimes}m}\right)\;dt\\
+\int_{2}^{\infty}\varphi_{\gamma}\left((w^{-1}(t)\dot{w}(t))\hat{\otimes} (\varpi^{-1}(t)\hat{\otimes}\varpi(t))^{\hat{\otimes}m}\right)\;dt
\end{gathered}
\end{equation}
Since $t\geq2$ the description of the regularized representative gives us that $w^{-1}(t)\dot{w}(t)=\pi iE'(t-1)$, the propagation of which tends to 0 as $t\longrightarrow\infty$. By \Cref{remCb1} it follows that the integrand $\varphi_{\gamma}\left((w^{-1}(t)\dot{w}(t))\hat{\otimes} (\varpi^{-1}(t)\hat{\otimes}\varpi(t))^{\hat{\otimes}m}\right)=0$ once the propagation of all these operators is less than some $\varepsilon_{\widetilde{M}}$, which occurs for large enough $t$. Hence there exists some $t_\varepsilon$ such that the last integral of \eqref{Cb:TauDistribution} is bounded above by
\[\int_{2}^{t_\varepsilon}\pi i\left|\varphi_{\gamma}\left(E'(t-1)\hat{\otimes} (\varpi^{-1}(t)\hat{\otimes}\varpi(t))^{\hat{\otimes}m}\right)\right|\;dt\]
which is finite from the bounds of \Cref{thmCb1}. Similarly, by \Cref{proCb1} there exists some constant $r>0$ such that for $t\geq6r$ 
\[||\varpi(t)-w(t)||_{\mathscr{B},k}\;\mathrm{and}\;||\varpi(t)^{-1}-w^{-1}(t)||_{\mathscr{B},k}<\frac{C_k}{t^3}\] 
The norm boundedness of all $B\in(\mathscr{B}(\widetilde{M})^\Gamma)^+$ and the results of \Cref{thmCb1} provide existence of $M_k,N_k>0$ such that the second integral of \eqref{Cb:TauDistribution} is bounded by
\[(6r-2)R_\alpha N_k+R_\alpha\int_{6r}^{\infty}||w^{-1}(t)\dot{w}(t)||_{\mathscr{B},k}||w(t)||_{\mathscr{B},k}^m||\varpi(t)-w(t)||_{\mathscr{B},k}^m\;dt\]
\[\leq (6r-2)R_\alpha N_k+R_\alpha\int_{6r}^{\infty}M^2_k||\varpi(t)-w(t)||_{\mathscr{B},k}^m\;dt\]
The finiteness of the integral follows directly from 
\[\int_{6r}^{\infty}M^2_k||\varpi(t)-w(t)||_{\mathscr{B},k}\;dt<M^2_kC_k\int_{6r}^{\infty}\frac{1}{t^{3m}}\;dt\]
An exact replica of this argument applied to $w^{-1}(t)-\varpi^{-1}(t)$ finishes our proof.
\end{proof}
The following three results show that $\tau_\varphi([u])$ is independent of the choice of regularized representatives, with \Cref{thmCb3} providing the proof that the replacement of $u_1$ by $w$ through the path of local loops behaves as intended.

\begin{mylem}\label{lemCb2}
Let $[u]\in K_1(\mathscr{B}_{L,0}^*(\widetilde{M})^\Gamma)$ with $v$ and $w$ both being regularized representatives, and $\{v_s\}_{s\in[0,1]}$ the associated family of piecewise smooth invertibles. For any delocalized cyclic cocyle $\varphi_\gamma\in (C^{2m}(\mathbb{C}\Gamma,\mathrm{cl}(\gamma)),b)$
\[\frac{\partial}{\partial s}\varphi_{\gamma}\left((v_s^{-1}\cdot\partial_tv_s)\hat{\otimes} (v_s^{-1}\hat{\otimes} v_s)^{\hat{\otimes}m}\right)=\frac{\partial}{\partial t}\varphi_{\gamma}\left((v_s^{-1}\cdot\partial_sv_s)\hat{\otimes} (v_s^{-1}\hat{\otimes} v_s)^{\hat{\otimes}m}\right)\]
\end{mylem}
\begin{proof}
Working in the unitization $(\mathscr{B}^*(\widetilde{M})^\Gamma)^+$ and noting that every invertible element in $(\mathscr{B}^*(\widetilde{M})^\Gamma)^+$ can be viewed as one in $(\mathscr{B}_{L,0}^*(\widetilde{M})^\Gamma)$, we wish to show vanishing of 
\begin{equation}\label{Cb:DerivEqual}
\begin{gathered}
\frac{\partial}{\partial s}\varphi_{\gamma}\left((v_s^{-1}\partial_tv_s)\hat{\otimes} (v_s^{-1}\hat{\otimes} v_s)^{\hat{\otimes}m}\right)-\frac{\partial}{\partial t}\varphi_{\gamma}\left((v_s^{-1}\partial_sv_s)\hat{\otimes} (v_s^{-1}\hat{\otimes} v_s)^{\hat{\otimes}m}\right)
\end{gathered}
\end{equation}
By definition every cyclic cocyle $\varphi_\gamma$ belongs to the kernel of the boundary map $b$; in addition $\varphi_{\gamma}(\overline{A}_0\hat{\otimes}\cdots\hat{\otimes} \overline{A}_n)$ vanishes if $\overline{A}_i=\mathbbm{1}$ for any $\overline{A}_i\in\mathcal{A}^+$.
\begin{gather*}
0=\sum_{k=0}^{m}2(b\varphi_{\gamma})\left((v_s^{-1}\cdot\partial_sv_s)\hat{\otimes}(v_s^{-1}\hat{\otimes} v_s)^{\hat{\otimes} k}\hat{\otimes} (v_s^{-1}\cdot\partial_tv_s)\hat{\otimes} (v_s^{-1}\hat{\otimes} v_s)^{\hat{\otimes}(m-k)}\right)
\end{gather*}
Through a tedious but straightforward computation we obtain that the above sum gives precisely the expression for \eqref{Cb:DerivEqual} which thus proves the result.
\end{proof}

\begin{mycor}\label{corCb2}
Let $[u]\in K_1(\mathscr{B}_{L,0}^*(\widetilde{M})^\Gamma)$ with $v$ and $w$ both being regularized representatives and $\{v_s\}_{s\in[0,1]}$ the associated family of piecewise smooth invertibles, then $\tau_{\varphi}(v_1)=\tau_{\varphi}(w)$ for any polynomial growth $\varphi_\gamma\in (C^{2m}(\mathbb{C}\Gamma,\mathrm{cl}(\gamma)),b)$.
\end{mycor}
\begin{proof}
Taking a double integral $\int_{0}^{T}\int_{0}^{1}\;ds\;dt$ over the derivative equality in \eqref{Cb:DerivEqual} gives
\begin{gather*}
\int_{0}^{T}\varphi_{\gamma}\left((v_1(t)^{-1}\partial_tv_1(t))\hat{\otimes} (v_1^{-1}(t)\hat{\otimes} v_1(t))^{\hat{\otimes}m}\right)\;dt\\
-\int_{0}^{T}\varphi_{\gamma}\left((v_0(t)^{-1}\partial_tv_0(t))\hat{\otimes} (v_0^{-1}(t)\hat{\otimes} v_0(t))^{\hat{\otimes}m}\right)\;dt\\
=\int_{0}^{1}\left.\varphi_{\gamma}\left((v_s(t)^{-1}\partial_sv_s(t))\hat{\otimes}(v_s^{-1}(t)\hat{\otimes} v_s(t))^{\hat{\otimes}m}\right)\right|_0^T\;ds
\end{gather*} 
Now $\varphi_{\gamma}\left((v_s(0)^{-1}\partial_sv_s(0))\hat{\otimes}(v_s^{-1}(0)\hat{\otimes} v_s(0))^{\hat{\otimes}m}\right)=0$ since $v_s(0)\equiv v_s^{-1}(0)\equiv\mathbbm{1}$ for all $s\in[0,1]$. Thus using the fact $v_0(t)=w(t)$ and letting $T\longrightarrow\infty$ we obtain
\[\tau_{\varphi}(v_1)-\tau_{\varphi}(w)=\frac{(-1)^mm!}{\pi i}\int_{0}^{1}\lim_{T\longrightarrow\infty}\varphi_{\gamma}\left((v_s^{-1}(T)\partial_sv_s(T))\hat{\otimes} (v_s^{-1}(T)\hat{\otimes} v_s(T))^{\hat{\otimes}m}\right)\;ds\]
It thus only remains to prove that for all $s$, uniformly with respect to the norm $||\cdot||_{\mathscr{B},k}$
\begin{equation}
\lim_{T\longrightarrow\infty}\varphi_{\gamma}\left((v_s^{-1}(T)\partial_sv_s(T))\hat{\otimes} (v_s^{-1}(T)\hat{\otimes} v_s(T))^{\hat{\otimes}m}\right)=0
\end{equation}
By basic distribution of tensors over addition and mutlilinearity of cyclic cocyles the above decomposes as
\begin{gather*}
\lim_{T\longrightarrow\infty}\varphi_{\gamma}\varphi_{\gamma}\left((v_s^{-1}(T)\partial_sv_s(T))\hat{\otimes} (v_s^{-1}(T)-\varpi^{-1}(T)\hat{\otimes} v_s(T))^{\hat{\otimes}m}\right)\\
+\lim_{T\longrightarrow\infty}\varphi_{\gamma}\varphi_{\gamma}\left((v_s^{-1}(T)\partial_sv_s(T))\hat{\otimes} (\varpi^{-1}(T)\hat{\otimes} v_s(T)-\varpi(T))^{\hat{\otimes}m}\right)\\
+\lim_{T\longrightarrow\infty}\varphi_{\gamma}\varphi_{\gamma}\left((v_s^{-1}(T)\partial_sv_s(T))\hat{\otimes} (\varpi^{-1}(T)\hat{\otimes}\varpi(T))^{\hat{\otimes}m}\right)
\end{gather*}
By \Cref{proCb1} and its corollary there exist operators $\varpi(t),\varpi^{-1}(t)\in(\mathscr{B}(\widetilde{M})^\Gamma)^+$ such that for any $\varepsilon>0$ there exists $t$ large enough such that the following hold.  
\[||\varpi(t)-v_s(t)||_{\mathscr{B},k}\;,\;||\varpi(t)^{-1}-v_s^{-1}(t)||_{\mathscr{B},k}<\frac{C_k}{t^3}\;,\;\mathsf{prop}(\varpi(t))\;,\;\mathsf{prop}(\varpi^{-1}(t))<\varepsilon\]
We may assume that $T\geq2$ so the description of the regularized representative gives $v_s^{-1}(T)\cdot\partial_sv_s(T)=\pi i\partial_sE_s(T)$, the propagation of which tends to 0 as $T\longrightarrow\infty$. Thus for large enough $T$ the propagation of all the operators is less than some $\varepsilon_{\widetilde{M}}$ and by \Cref{remCb1} \[\lim_{T\longrightarrow\infty}\varphi_{\gamma}\left((v_s^{-1}(T)\partial_sv_s(T))\hat{\otimes} (\varpi^{-1}(T)\hat{\otimes}\varpi(T))^{\hat{\otimes}m}\right)=0\]
By the results of \Cref{thmCb1} and the norm boundedness of all $B\in(\mathscr{B}(\widetilde{M})^\Gamma)^+$ 
\[\lim_{T\longrightarrow\infty}\left|\varphi_{\gamma}\left((v_s^{-1}(T)\partial_sv_s(T))\hat{\otimes} (v_s^{-1}(T)-\varpi^{-1}(T)\hat{\otimes} v_s(T))^{\hat{\otimes}m}\right)\right|\]
\[\leq R_\alpha\lim_{T\longrightarrow\infty}||v_s^{-1}(T)\partial_sv_s(T)||_{\mathscr{B},k}||\varpi(t)^{-1}-v_s^{-1}(T)||_{\mathscr{B},k}^m||v_s(T)||_{\mathscr{B},k}^m\]
\[\leq R_\alpha C_k^m\lim_{T\longrightarrow\infty}\frac{||v_s^{-1}(T)\partial_sv_s(T)||_{\mathscr{B},k}||v_s(T)||_{\mathscr{B},k}^m}{T^{3m}}=0\]
An exact replica of this argument holds for the term involving $v_s^{-1}(T)-\varpi^{-1}(T)$. 
\end{proof}

\begin{mythm}\label{thmCb3}
Let $[u]\in K_1(\mathscr{B}_{L,0}(\widetilde{M})^\Gamma)$ with $v$ and $w$ both being regularized representatives, then $\tau_{\varphi}(v)=\tau_{\varphi}(w)$ for any polynomial growth $\varphi_\gamma\in(C^{2m}(\mathbb{C}\Gamma,\mathrm{cl}(\gamma)),b)$.
\end{mythm}
\begin{proof}
We have proven in the above corollary that $\tau_{\varphi}(v_1)=\tau_{\varphi}(w)$; by construction $v_1(s)=v(s)$ for all $s\in\mathbb{R}_{\geq0}\setminus(1,2)$, with $v_1v^{-1}:[1,2]\longrightarrow(\mathscr{B}(\widetilde{M})^\Gamma)^+$ being a local loop of invertible elements. Thus there exists a local invertible $f:S^1\longrightarrow(\mathscr{B}(\widetilde{M})^\Gamma)^+$-- which by \Cref{lemCb1} is homotopic to $e^{2\pi i\theta}P(t)+(\mathbbm{1}-P(t))$ for some idempotent $P(t)$-- such that $v(s)$ and $v_1(s)$ differ by $f(\theta)$ as elements in $K_1(\mathscr{B}(\widetilde{M})^\Gamma)$. 
\[\tau_{\varphi}(v)=\tau_{\varphi}(v_1)+\tau{_\varphi}(v_1^{-1}v)=\tau_{\varphi}(w)+\tau{_\varphi}(v_1v^{-1})=\tau_{\varphi}(w)+\tau_{\varphi}(f_L)\]
\[=\tau_{\varphi}(w)+\frac{(-1)^mm!}{\pi i}\int_{0}^{\infty}\int_{0}^{1}\overline{\varphi}_{\gamma}\left((f^{-1}(\theta)\dot{f}(\theta))\hat{\otimes} (f^{-1}(\theta)\hat{\otimes}f(\theta))^{\hat{\otimes}m}\right)\;d\theta\;dt\]
Here $\dot{f}(\theta)=2\pi ie^{2\pi i\theta}P(t)$ refers to the derivative with respect to $\theta$; it is easy to verify that $f^{-1}(\theta)=(e^{-2\pi i\theta}P(t)+(\mathbbm{1}-P(t)))$. It follows that the above integral is equal to
\begin{gather*}
\int_{0}^{\infty}\int_{0}^{1}\overline{\varphi}_{\gamma}\left(2\pi iP(t)\hat{\otimes}((e^{-2\pi i\theta}-1)P(t)+\mathbbm{1}\hat{\otimes}(e^{2\pi i\theta}-1)P(t)+\mathbbm{1})^{\hat{\otimes}m}\right)\;d\theta
\end{gather*}
Using multilinearity of cyclic cocyles and the fact that $\overline{\varphi}_\gamma$ vanishes on the unit the integrand simplifies to
\[2\pi i(e^{-2\pi i\theta}-1)^m(e^{2\pi i\theta}-1)^m\varphi_{\gamma}\left(P(t)^{\hat{\otimes}2m+1}\right)\]
Vanishing of the double integral now follows from \Cref{remCb1} and the fact that for all $\varepsilon>0$ there exists $t_\varepsilon$ large enough such that $\mathsf{prop}(P(t_\varepsilon))\leq\varepsilon$.
\end{proof}

\begin{mythm}\label{thmCb4}
The determinant map pairing the higher rho invariant is independent of the choice of delocalized cyclic cocycle representative. Explicitly, if $[\varphi_\gamma]=[\phi_\gamma]\in HC^{2m}(\mathbb{C}\Gamma,\mathrm{cl}(\gamma))$, then $\tau_{\varphi_\gamma}(\rho(\widetilde{D},\widetilde{g}))=\tau_{\phi_\gamma}(\rho(\widetilde{D},\widetilde{g}))$
\end{mythm}
\begin{proof}
By the previous results we are able to fix a regularized representative $w$, and by hypothesis, $\varphi_\gamma$ and $\phi_\gamma$ are cohomologous via a coboundary $b\varphi\in BC^{2m}(\mathbb{C}\Gamma,\mathrm{cl}(\gamma))$. We obtain an identical transgression formula as was calculated in \Cref{thmCa2}, \eqref{Ca:Transgression} 
\begin{equation}
m(b\overline{\varphi})\left((w^{-1}(t)\dot{w}(t))\hat{\otimes}(w^{-1}(t)\hat{\otimes}w(t))^{\hat{\otimes}m}\right)=\frac{d}{dt}\overline{\varphi}((w^{-1}(t)\hat{\otimes} w(t))^{\hat{\otimes}m})
\end{equation}
Using the simplified definition of the determinant map, provided by \eqref{Cb:TauSimple}
\begin{equation}
m\tau_{b\varphi}(\rho(\widetilde{D},\widetilde{g})):=\frac{(-1)^mm!}{\pi i}\int_{0}^{\infty}m(b\overline{\varphi}_{\gamma})\left((w^{-1}(t)\dot{w}(t))\hat{\otimes}(w^{-1}(t)\hat{\otimes}w(t))^{\hat{\otimes}m}\right)\;dt
\end{equation}
Thus by the transgression formula $\tau_{b\varphi}(\rho(\widetilde{D},\widetilde{g}))$ is equal (up to a constant) to
\begin{equation}
\lim_{t\longrightarrow\infty}\overline{\varphi}((w^{-1}(t)\hat{\otimes} w(t))^{\hat{\otimes}m})-\lim_{t\longrightarrow0}\overline{\varphi}((w^{-1}(t)\hat{\otimes} w(t))^{\hat{\otimes}m})
\end{equation}
Now $\lim_{t\longrightarrow0}\overline{\varphi}((w^{-1}(t)\hat{\otimes} w(t))^{\hat{\otimes}m})=0$ since by construction $w(0)=\mathbbm{1}$ and $\overline{\varphi}$ vanishes on the unit $\mathbbm{1}$. Moreover, for all $t\geq2$ some smooth normalizing function $\chi$ can be chosen such that
\[w(t)=\exp\left(2\pi i\frac{E(t)+1}{2}\right)=\exp\left(2\pi i\frac{\chi(\widetilde{D}/t)+1}{2}\right)\]
Denote $\psi(x)=\exp\left(2\pi i\frac{\chi(x)+1}{2}\right)-1=e^{i(\pi\chi(x)+\pi)}-1$
Since the class of $K$-theory representative is independent of the class of smooth normalizing function, without loss of generality assume that $\lim_{x\longrightarrow\pm\infty}\chi(x)$ converges to $\pm1$ at an exponential rate. From a standard argument it immediately follows that $\psi$ is a Schwartz function and thus by \Cref{proCa1} $\lim_{t\longrightarrow\infty}\overline{\varphi}((w^{-1}(t)\hat{\otimes} w(t))^{\hat{\otimes}m})=0$
\end{proof}

\begin{mypro}\label{proCb2}
Let $S_\gamma^*:HC^{2m}(\mathbb{C}\Gamma,\mathrm{cl}(\gamma))\longrightarrow HC^{2m+2}(\mathbb{C}\Gamma,\mathrm{cl}(\gamma))$ be the delocalized Connes periodicity operator, then $\tau_{[S_\gamma\varphi_\gamma]}(\rho(\widetilde{D},\widetilde{g}))=\tau_{[\varphi_\gamma]}(\rho(\widetilde{D},\widetilde{g}))$ for every $[\varphi_\gamma]\in HC^{2m}(\mathbb{C}\Gamma,\mathrm{cl}(\gamma))$. 
\end{mypro}
\begin{proof}
The proof exactly mirrors that of \Cref{proCa2}.
\end{proof}

\subsection{Proof of \Cref{thm2}}\label{C3}
The following discussion and the proof of \Cref{proCc1} closely align with the proof of Theorem 4.3 in \cite{XY19}. We first recall the construction at the beginning of the previous section of a representative of $\rho(\widetilde{D},\widetilde{g})$ using the path of invertibles 
\[S=\{U(t)=\exp(2\pi iH(t))|\;t\in[0,\infty)\}\] 
This construction can be altered by using the following smooth normalizing function $\psi$, where $F_t$ is as defined in the construction of Lott's higher eta invariant.
\begin{equation}
\psi(t^{-1}x)=F_{1/t}(x)=\frac{1}{\sqrt{\pi}}\int_{-\infty}^{x/t}e^{-s^2}\;ds\qquad t>0
\end{equation}
The invertibility of the Dirac operator $\widetilde{D}$ implies that $\psi(t^{-1}D)$ converges in operator norm to $\frac{1}{2}(\mathbbm{1}+\widetilde{D}|\widetilde{D}|^{-1})$ as $t\rightarrow0$. Since $\psi$ is a smooth normalizing function the operator $\psi(\widetilde{D})^2-\mathbbm{1}$ is locally compact, hence $e^{2\pi iF_t(\widetilde{D})}\equiv\mathbbm{1}$ modulo locally compact operators. Moreover, $\psi$ can be approximated by smooth normalizing functions with compactly supported distributional Fourier transforms, hence the inverse Fourier transform relation
\begin{equation}
\psi(\widetilde{D})=\int_{-\infty}^{\infty}\hat{\psi}(s)e^{is\widetilde{D}}\;ds
\end{equation}
is well defined, and by finite propagation property of the wave operator $e^{is\widetilde{D}}$ it follows that the path $U\in(C_{L,0}^*(\widetilde{M},\mathcal{S})^\Gamma)^+$ defined by
\begin{equation}
U(t)=U_t(\widetilde{D})=e^{2\pi i\psi(\widetilde{D}/t)};t\in(0,\infty)\qquad U_0\equiv\mathbbm{1}
\end{equation}
can be uniformly approximated by paths of invertible elements with finite propagation. In totality, we have that $U$ is an invertible element of $(C_{L,0}^*(\widetilde{M},\mathcal{S})^\Gamma)^+$ and gives rise to a class in $K_1(C_{L,0}^*(\widetilde{M},\mathcal{S})^\Gamma)$.
\begin{mypro}\label{proCc1}
The element $U$ is invertible in $(\mathscr{A}_{L,0}(\widetilde{M},\mathcal{S})^\Gamma)^+$.
\end{mypro}
\begin{proof}
Firstly, we note that by the proof of \Cref{thmCb4} the function $U_t(x)-1$ is a Schwartz function, and so by \Cref{lemCa1} $U_t(\widetilde{D})-\mathbbm{1}$ belongs to $\mathscr{A}(\widetilde{M},\mathcal{S})^\Gamma)$ for all $t\in[0,\infty)$. Being of the form $e^{f(\widetilde{D})}$ this further proves that $U_t(\widetilde{D})\in(\mathscr{A}(\widetilde{M},\mathcal{S})^\Gamma)^+$ is invertible for all $t$. By definition of the localization algebra, to prove that $U$ is an invertible element of $(\mathscr{A}_{L,0}(\widetilde{M},\mathcal{S})^\Gamma)^+$ it suffices to show that $U-\mathbbm{1}$ is a piecewise smooth function on the half-line. By the description of $F_{1/t}$ this is immediate for all $t\in(0,\infty)$, hence we only need to show smoothness at $t=0$ with respect to the Frechet topology generated by the seminorms $||\cdot||_{\mathscr{A},k}$. Denoting by $\dot{\psi}(s)$ the derivative with respect to $s$, for $t\in(0,\infty)$
\[||U_t(\widetilde{D})-\mathbbm{1}||_{\mathscr{A},k}=\exp\left(\int_{0}^{t}2\pi i\dot{\psi}(\widetilde{D}/s)\;ds\right)-\mathbbm{1}=\left|\left|\sum_{n=1}^{\infty}\frac{1}{n!}\left(\int_{0}^{t}2\pi i\dot{\psi}(\widetilde{D}/s)\;ds\right)^n\right|\right|_{\mathscr{A},k}\]
\[\leq\sum_{n=1}^{\infty}\frac{(2\pi)^n}{n!}\left|\left|\left(\int_{0}^{t}\dot{\psi}(\widetilde{D}/s)\;ds\right)\right|\right|_{\mathscr{A},k}^n\leq\sum_{n=1}^{\infty}\frac{(2\pi)^n}{n!}\left|\left|\left.\left(\frac{1}{\sqrt{\pi}}\frac{-2}{s^2}\widetilde{D}e^{-\widetilde{D}^2/s^2}\right)\right|^{t}_0\right|\right|_{\mathscr{A},k}^n\]
Following the arguments on \cite[Page 21]{XY19} for $1/s^2\in[n,n+1)$ there exists $n$ large enough such that for some constants $C_0,C_1>0$
\[||C_0e^{-\widetilde{D}^2/s^2}||_{\mathscr{A},k}\leq||e^{-n\widetilde{D}^2}||_{\mathscr{A},k}\leq e^{-nr^2/2}\leq C_1e^{-r^2/2s^2}\]
Thus for $s>0$ sufficiently small and for some positive constant $C_2$ we obtain-- taking $t>0$ sufficiently small to begin with-- the bound
\[\lim_{t\longrightarrow0}||U_t(\widetilde{D})-\mathbbm{1}||_{\mathscr{A},k}\leq\lim_{t\longrightarrow0}\sum_{n=1}^{\infty}\frac{(2\pi)^n}{n!}\left(\frac{-2}{t^2\sqrt{\pi}}C_2C_1e^{-r^2/2t^2}-\lim_{s\longrightarrow0}\frac{-2}{s^2\sqrt{\pi}}C_2C_1e^{-r^2/2s^2}\right)^n\] 
\[=\lim_{t\longrightarrow0}\sum_{n=1}^{\infty}\frac{(2\pi)^n}{n!}\left(\frac{-2}{t^2\sqrt{\pi}}C_2C_1e^{-r^2/2t^2}\right)^n=\lim_{t\longrightarrow0}\exp\left(\frac{1}{t^2}Ce^{-r^2/2t^2}\right)-1=0\]
Since $U_0(\widetilde{D})=\mathbbm{1}$ it follows that $U-\mathbbm{1}$ is continuous with respect to the family of seminorms; the same holds true for all orders of its derivatives according to the expansion
\[\frac{d}{dt}(U_t(x)-1)=\sum_{m_1+2m_2+\cdots+km_k=k} C_{m_1,\ldots,m_k}\prod_{l=1}^{k}C_k^{m_l}\left(\frac{d^l}{dt^l}\psi(x/t)\right)^{m_l}(U_t(x)-1)\] 
where $0\leq m_l\leq k$. Thus $U-\mathbbm{1}$ is smooth with respect to $||\cdot||_{\mathscr{A},k}$ which proves that $U\in(\mathscr{A}_{L,0}(\widetilde{M},\mathcal{S})^\Gamma)^+$ as desired.
\end{proof}
\begin{mycor}\label{corCc1}
Let $\varphi_\gamma\in(C^{2m}\mathbb{C}\Gamma,\mathrm{cl}(\gamma),b)$ be of polynomial growth, with $\widetilde{D}$ being invertible, then the integral 
\[\int_{0}^{\infty}\overline{\varphi}_\gamma\left(\dot{U}_t(\widetilde{D})U_t^{-1}(\widetilde{D})\hat{\otimes}(U_t(\widetilde{D})\hat{\otimes}U_t^{-1}(\widetilde{D}))^{\hat{\otimes}m}\right)\;dt\]
converges absolutely. 
\end{mycor}
\begin{proof}
This is a direct consequence of \Cref{lemCa2}.
\end{proof}

The construction of a regularized representative of $\rho(\widetilde{D},\widetilde{g})$ involves the choice of some smooth normalizing function $\chi$ with compactly supported distributional Fourier transform $\hat{\chi}$, such that $E(t)=(\chi(\widetilde{D}/t)+\mathbbm{1})/2$ and $E$ has the properties outlined on \cpageref{lemCb1}. Adapting the argument preceding \cite[Proposition 6.11]{CWXY19} we can thus construct a path $w$
\begin{equation}
w(t)=\left\{\begin{array}{cc}
U(t) & 0\leq t\leq 1\\
e^{2\pi i((2-t)\psi(\widetilde{D})+(t-1)E(1))} & 1\leq t\leq 2\\
e^{2\pi iE(t-1)} & t\geq2
\end{array}\right.
\end{equation}
which defines a regularized representative. On the other hand, by definition, $U$ is its own regularized representative and the equality of $[U]$ and $[w]$ as K-theory classes in $K_1(C_{L,0}^*(\widetilde{M},\mathcal{S})^\Gamma)$ follows from their being homotopic in $(\mathscr{B}_{L,0}(\widetilde{M},\mathcal{S})^\Gamma)^+$. Explicitly, we have the homotopy induced by the family of invertibles $h_s:s\in[0,1]$ defined by
\begin{equation}
h_s(t)=\left\{\begin{array}{cc}
U(t) & 0\leq t\leq 1\\
e^{2\pi i((2-t)\psi(\widetilde{D})+(t-1)(sE(1)+(1-s)\psi(\widetilde{D})))} & 1\leq t\leq 1+s\\
e^{2\pi i(sE(t-1)+(1-s)\psi(\frac{\widetilde{D}}{t-1}))} & t\geq1+s
\end{array}\right.
\end{equation}
Thus by \Cref{corCc1} and the proofs of \Cref{lemCb2} and \Cref{corCb2} we obtain that
\[\tau_{[\varphi_\gamma]}(\rho(\widetilde{D},\widetilde{g}))=\tau_{[\varphi_\gamma]}(w)=\tau_{\varphi_\gamma}(U)\]
\[=\frac{(-1)^mm!}{\pi i}\int_{0}^{\infty}\overline{\varphi}_\gamma\left(U_t(\widetilde{D})^{-1}\dot{U}_t(\widetilde{D})\hat{\otimes}(U_t^{-1}(\widetilde{D})\hat{\otimes}U_t(\widetilde{D}))^{\hat{\otimes}m}\right)\;dt\]
\[=\frac{(-1)^mm!}{\pi i}\int_{0}^{\infty}\overline{\varphi}_\gamma\left(\dot{\bar{u}}_t(\widetilde{D})\bar{u}_t^{-1}(\widetilde{D})\hat{\otimes}(\bar{u}_t(\widetilde{D})\hat{\otimes}\bar{u}_t^{-1}(\widetilde{D}))^{\hat{\otimes}m}\right)\;dt=(-1)^m\eta_{[\varphi_\gamma]}(\widetilde{D})\]
where $\bar{u}_t=U_{1/t}$ and we have used the substitution $u_t\longleftrightarrow u_t^{-1}$.

\subsection{Delocalized Higher Atiyah-Patodi-Singer Index Theorem}\label{C4}

\noindent Consider smooth vector bundles $V_1$ and $V_2$ over a compact smooth manifold $M$-- without boundary-- and an elliptic differential operator $D:V_1\longrightarrow V_2$ which acts on the smooth sections of these vector bundles. Since every such $D$ has a pseudo inverse it is a Fredholm operator, with \textit{analytical index} defined by 
\begin{equation}
\mathrm{ind}(D)=\dim\ker(D)-\dim\ker(D^*)
\end{equation}
Let us also recall the \textit{topological index} of $D$ with respect to a non-commutative $K$-theoretic framework
\begin{equation}\label{TopInd}
\int_M\mathsf{ch}(D)\mathsf{Td}(T^*M\otimes\mathbb{C})
\end{equation}
where $\mathsf{Td}(T^*M\otimes\mathbb{C})$ is the Todd class of the complexified tangent bundle of $M$, and $\mathsf{ch}(D)$ is the pullback of the Chern character on a particular $K$-theory class associated to $D$. The original Atiyah-Singer index theorem \cite{AS63} was proven through cohomological means, and asserts that the topological index of $D$ is equal to its analytical index; a $K$-theoretic proof \cite{AS68I,AS68III} was later provided and shown to be equivalent to the one employing cohomology. The Atiyah-Patodi-Singer index theorem \cite{APS75I,APS75II} generalizes this statement to include manifolds with boundary under satisfaction of certain global boundary conditions. By considering $\mathrm{Ind}_G(D)$ rather than $\mathrm{ind}(D)$, this further admits a kind of (delocalized) higher analogue, in our case modeled on that of Lott \cite{JL92} (see the relationship between equations (1) and (66)).

Prior to stating and proving the delocalized version of a higher Atiyah-Patodi-Singer index theorem we will first exhibit a necessary relationship between the determinant map $\tau_{\varphi_\gamma}$ of the previous section and the Chern-Connes character map. In the remainder of this section we will work within the restriction of even dimensional cyclic cocyles and under the condition of a compact spin manifold $M$ having fundamental group $\Gamma$ of polynomial growth. In particular, given a delocalized cyclic cocyle $\varphi_\gamma\in(C^{2m}\mathbb{C}\Gamma,\mathrm{cl}(\gamma),b)$ we will define the $\varphi_\gamma$-component of the Chern-Connes character of an idempotent $p\in\mathscr{B}(\widetilde{M})^\Gamma$ according to that of \cite[Chapter 8]{JL94}
\begin{equation}\label{Cd:ChernConnes}
\mathsf{ch}_{\varphi_\gamma}(p):=\frac{(-1)^m(2m)!}{m!}\varphi_\gamma\left(p^{\hat{\otimes}2m+1}\right)
\end{equation}
Firstly let us prove the usual well-definedness properties.
\begin{mypro}\label{proCd1}
Let $[\varphi_\gamma]\in HC^{2m}(\mathbb{C}\Gamma,\mathrm{cl}(\gamma))$, then the $[\varphi_\gamma]$-component of the Chern-Connes character
\[\mathsf{ch}_{[\varphi_\gamma]}:K_0(\mathscr{B}(\widetilde{M})^\Gamma)\longrightarrow\mathbb{C}\]
is well defined, particularly being independent of the choices of cocycle representative and $K$-theory class representative.
\end{mypro}
\begin{proof}
Since $\varphi_\gamma$ can be chosen to be of polynomial growth and $p\in(\mathscr{B}(\widetilde{M})^\Gamma)^+$	then \Cref{thmCb1} asserts that the formula for the $\varphi_\gamma$-component of the Chern-Connes character makes sense. Suppose that $\varphi_\gamma$ and $\phi_\gamma$ belong to the same cohomology class in $HC^{2m}(\mathbb{C}\Gamma,\mathrm{cl}(\gamma))$; by hypothesis, $\varphi_\gamma$ and $\phi_\gamma$ are cohomologous via a coboundary $b\varphi\in BC^{2m}(\mathbb{C}\Gamma,\mathrm{cl}(\gamma))$.	Independence with respect to cyclic cocyle representatives thus follows from showing that $\mathsf{ch}_{b\varphi}(p)=0$ for any idempotent $p$. A direct computation gives
\begin{equation}
\begin{gathered}
(b\varphi)\left(p^{\hat{\otimes}2m+1}\right)=\varphi\left(p^{\hat{\otimes}2m}\right)=0
\end{gathered}
\end{equation}
since by the definition of the cyclic operator $\mathfrak{t}\varphi\left(p^{\hat{\otimes}2m}\right)=(-1)^{2m-1}\varphi\left(p^{\hat{\otimes}2m}\right)$. Let us now turn our attention to proving that if $p_0,p_1\in(\mathscr{B}(\widetilde{M})^\Gamma)^+$ belong to the same class in $K_0(\mathscr{B}(\widetilde{M})^\Gamma)$ then $\mathsf{ch}_{[\varphi_\gamma]}(p_0)=\mathsf{ch}_{[\varphi_\gamma]}(p_1)$ 
By hypothesis there exist a piecewise smooth family of idempotents $p_t:t\in(0,1)$ connecting $p_0$ and $p_1$, which allows for the usual trick of taking the derivative. Using the fact that $\varphi_\gamma$ belongs to the kernel of the boundary map $b$ a direct calculation gives 
\[\frac{d}{dt}\varphi_\gamma\left(p_t^{\hat{\otimes}2m+1}\right)=(2m+1)\varphi_\gamma\left(\dot{p}_t\hat{\otimes}p_t^{\hat{\otimes}2m}\right)=(b\varphi_\gamma)\left((\dot{p}_tp_t-p_t\dot{p}_t)\hat{\otimes}p_t^{\hat{\otimes}2m+1}\right)=0\]
The desired result now follows immediately from integration 
\[0=\int_{0}^{1}\frac{(-1)^m(2m)!}{m!}\frac{d}{dt}\varphi_\gamma\left(p_t^{\hat{\otimes}2m+1}\right)\;dt=\mathsf{ch}_{[\varphi_\gamma]}(p_1)-\mathsf{ch}_{[\varphi_\gamma]}(p_0)\]
\end{proof}

\begin{mypro}\label{proCd2}
Let $S_\gamma^*:HC^{2m}(\mathbb{C}\Gamma,\mathrm{cl}(\gamma))\longrightarrow HC^{2m+2}(\mathbb{C}\Gamma,\mathrm{cl}(\gamma))$ be the delocalized Connes periodicity operator, then $\mathsf{ch}_{[\varphi_\gamma]}=\mathsf{ch}_{[S_\gamma\varphi_\gamma]}$ for every $[\varphi_\gamma]\in HC^{2m}(\mathbb{C}\Gamma,\mathrm{cl}(\gamma))$. 
\end{mypro}
\begin{proof}
Recalling the definition of the map $\beta$ as given in \eqref{Ba:ConnesPer} of \Cref{B1} it is straightforward to compute the action of $\beta b$ and $b\beta$ as refers to the Chern-Connes character.
\[(\beta\circ b\varphi_\gamma)\left(p^{\hat{\otimes}2m+3}\right)=0\quad\mathrm{and}\quad(b\circ\beta\varphi_\gamma)\left(p^{\hat{\otimes}2m+3}\right)=-(m+1)\varphi_\gamma\left(p^{\hat{\otimes}2m+1}\right)\]
Using the relation $S_\gamma=\frac{1}{(2m+1)(2m+2)}\left(\beta b+b\beta\right)$ as in \Cref{defBa1} we obtain the desired result.
\end{proof}

\begin{mylem}\label{lemCd1}
Let $[\varphi_\gamma]\in HC^{2m}(\mathbb{C}\Gamma,\mathrm{cl}(\gamma))$, then the following diagram commutes
\[\begin{tikzcd}
K_1(C_{L,0}(\widetilde{M})^\Gamma) \arrow[r,"\tau_{[\varphi_\gamma]}"] & \mathbb{C}\\
K_{0}(C^*(\widetilde{M})^\Gamma) \arrow[u,"\partial"] \arrow[r,"\mathsf{ch}_{[\varphi_\gamma]}"] & \mathbb{C} \ar[u,swap,"\times(-2)"]
\end{tikzcd}\]
\end{mylem}
\begin{proof}
By \Cref{proAa1} we know that the $K$-theory of $C^*(\widetilde{M})^\Gamma$ coincides with that of $\mathscr{B}(\widetilde{M})^\Gamma$, and likewise with respect to the localization algebras. Thus we can view every element of $K_{0}(C^*(\widetilde{M})^\Gamma)$ as a formal difference of two idempotents belonging to $(\mathscr{B}(\widetilde{M})^\Gamma)^+$. Each idempotent $p\in\mathscr{B}(\widetilde{M})^\Gamma$ defines an element $F\in\mathscr{B}_L(\widetilde{M})^\Gamma$
\begin{equation}
F(t)=\left\{\begin{array}{cc}
(1-t)p & t\in[0,1]\\
0 & t\in(1,\infty)
\end{array}\right.
\end{equation}
and $\partial[p]=[u]$ defines a $K$-theory class of invertibles in $K_1(\mathscr{B}_{L,0}^*(\widetilde{M})^\Gamma)$, where $\partial:K_{0}(C^*(\widetilde{M})^\Gamma)\longrightarrow K_1(C_{L,0}(\widetilde{M})^\Gamma)$ is the $K$-theoretical connecting map. The proof now follows by mirroring the calculations in \cite[Proposition 7.2]{CWXY19}
\end{proof}

Now we shall set up the necessary preliminaries for a delocalized version of the Atiyah-Patodi-Singer index theorem. To begin with, let $W$ be a compact $n$-dimensional spin manifold with boundary $\partial W=M$ which is closed, and naturally is an $n-1$-dimensional spin manifold. Moreover $W$ is endowed with a Riemannian metric $g$ which has product structure near $M$ and is of positive scalar curvature metric when restricted to $M$. Let $\widetilde{D}_W$ be the Dirac operator lifted to the universal cover $\widetilde{W}$, $\widetilde{g}$ be the metric lifted to $\widetilde{W}$, and by $\partial\widetilde{W}=\widetilde{M}$ denote the lifting of $M$ with respect to the covering map $p:\widetilde{W}\longrightarrow W$. As shown in \cite[Section 3]{XY14} the operator $\widetilde{D}_W$ defines a higher index $\mathrm{Ind}_{\pi_1(W)}(\widetilde{D}_W)\in K_n(C^*(\widetilde{W})^{\pi_1(W)})$, and as we have already detailed in \Cref{C2} in the case of $n-1$ being odd, the Dirac operator $\widetilde{D}_M$ defines a higher rho invariant $\rho(\widetilde{D}_M,\widetilde{g})$ in $K_{n-1}(C^*_{L,0}(\widetilde{M})^{\pi_1(W)})$. Recall that every equivariant coarse map $f:X\longrightarrow Y$ induces a homomorphism $C(f):C^*(X)^G\longrightarrow C^*(Y)^G$, which itself induces a functorial map $K(f)$ on the $K$-theory. Clearly the lifted inclusion map $\widetilde{\imath}:\widetilde{M}\hookrightarrow\widetilde{W}$ is equivariantly coarse and so gives rise to a natural homomorphism 
\begin{equation}
K(\widetilde{\imath}):K_{n-1}(C^*_{L,0}(\widetilde{M})^{\pi_1(W)})\longrightarrow K_{n-1}(C^*_{L,0}(\widetilde{W})^{\pi_1(W)})
\end{equation}
We will denote the image of $\rho(\widetilde{D}_M,\widetilde{g})$ under this map to also be $\rho(\widetilde{D}_M,\widetilde{g})$.

\begin{mythm}[Delocalized APS Index Theorem]\label{thmCd1}
Let $W$ be a compact even dimensional spin manifold with closed boundary $\partial W=M$, and endowed with a Riemannian metric $g$ which has product structure near $M$ and is of positive scalar curvature metric when restricted to $M$. If $\pi_1(W)$ is countable discrete, finitely generated, and of polynomial growth
\[\mathsf{ch}_{[\varphi_\gamma]}\left(\mathrm{Ind}_{\pi_1(W)}(\widetilde{D}_W)\right)=\frac{(-1)^{m+1}}{2}\eta_{[\varphi_\gamma]}(\widetilde{D}_M)\]
for any $[\varphi_\gamma]\in HC^{2m}(\mathbb{C}\Gamma,\mathrm{cl}(\gamma))$
\end{mythm}
\begin{proof}
The proof of \Cref{lemCd1} did not depend on the dimension or boundary structure of $M$, the only necessity being that $\widetilde{M}$ admit a proper and co-compact isometric action of $\Gamma$ (see \Cref{defAa5}); thus the following diagram also commutes.
\begin{equation}
\begin{tikzcd}
K_1(C_{L,0}(\widetilde{W})^{\pi_1(W)}) \arrow[r,"\tau_{[\varphi_\gamma]}"] & \mathbb{C}\\
K_{0}(C^*(\widetilde{W})^{\pi_1(W)}) \arrow[u,"\partial"] \arrow[r,"\mathsf{ch}_{[\varphi_\gamma]}"] & \mathbb{C} \ar[u,swap,"\times(-2)"]
\end{tikzcd}
\end{equation}
Moreover since $\dim(W)=n$ is even, by \cite[Theorem 1.14]{PS14} and \cite[Theorem A]{XY14} the image of the higher index under the connecting map is
\begin{equation}
\partial\left(\mathrm{Ind}_{\pi_1(W)}(\widetilde{D}_W)\right)=\rho(\widetilde{D}_M,\widetilde{g})\in K_{n-1}(C^*_{L,0}(\widetilde{W})^{\pi_1(W)})\cong K_{1}(C^*_{L,0}(\widetilde{W})^{\pi_1(W)})
\end{equation}
Coupling this identity with the main result of \Cref{C3} we obtain
\begin{equation}
\begin{gathered}
-2\mathsf{ch}_{[\varphi_\gamma]}\left(\mathrm{Ind}_{\pi_1(W)}(\widetilde{D}_W)\right)=\tau_{[\varphi_\gamma]}\left(\partial\left(\mathrm{Ind}_{\pi_1(W)}(\widetilde{D}_W)\right)\right)\\
=\tau_{[\varphi_\gamma]}\left(\rho(\widetilde{D}_M,\widetilde{g})\right)=(-1)^m\eta_{[\varphi_\gamma]}(\widetilde{D}_M)
\end{gathered}
\end{equation}
\end{proof}

\nocite{SW99}
\bibliographystyle{plainurl}
\bibliography{PhDReferences}
\textsc{Department of Mathematics, Texas A\&M University}\par\nopagebreak
\noindent\textit{E-mail address}: \texttt{sakajohn@math.tamu.edu}
\end{document}